\documentclass[12pt]{amsart}
\usepackage[margin=1in]{geometry}  
\usepackage{amscd,latexsym,amsthm,amsfonts,amssymb,amsmath,amsxtra}
\usepackage[mathscr]{eucal}
\usepackage{hyperref}
\usepackage{pdfsync}
\renewcommand\theequation{\thesection.\arabic{equation}}

\theoremstyle{break}
\newtheorem{thm}{Theorem}[section]
\newtheorem{cor}[thm]{Corollary}
\newtheorem{prop}[thm]{Proposition}
\newtheorem{lemma}[thm]{Lemma}

\newtheorem{rk}[thm]{Remark}

\newtheorem{defn}[thm]{Definition}

\newtheorem{exmp}[thm]{Example}

\newenvironment{proof-idea}{\noindent{\bf Proof Idea}\hspace*{1em}}{\qed\bigskip\\}
\newenvironment{proof-of-lemma}[1]{\noindent{\bf Proof of Lemma #1}\hspace*{1em}}{\qed\bigskip\\}
\newenvironment{proof-attempt}{\noindent{\bf Proof Attempt}\hspace*{1em}}{\qed\bigskip\\}

\newcommand{\sbst}{\subseteq}

\newcommand{\abs}[1]{\lvert#1\rvert}
\newcommand{\set}[2]{\{#1\,|\,#2\}}

\newcommand{\mtrtwo}[4]{\begin{pmatrix} #1 &#2 \\#3 &#4 \end{pmatrix}}

\newcommand{\BC}{{\mathbb {C}}}

\newcommand{\BZ}{{\mathbb {Z}}}

\newcommand{\RH}{{\mathrm {H}}}
\newcommand{\RI}{{\mathrm {I}}}

\newcommand{\RU}{{\mathrm {U}}}

\newcommand{\Sym}{{\mathrm {Sym}}}

\newcommand{\R}{\mathbb{R}}
\newcommand{\C}{\mathbb{C}}

\renewcommand{\Re}{\text{Re}}

\newcommand{\fg}{\mathfrak{g}}

\newcommand{\fh}{\mathfrak{h}}
\newcommand{\fb}{\mathfrak{b}}

\newcommand{\quo}{\backslash}

\renewcommand{\bar}{\overline}
\renewcommand{\tilde}{\widetilde}

\newcommand{\diag}{{\mathrm{diag}}}
\newcommand{\GL}{{\mathrm{GL}}}
\newcommand{\SU}{{\mathrm{SU}}}
\newcommand{\Hom}{{\mathrm{Hom}}}
\newcommand{\Bil}{{\mathrm {Bil}}}

\newcommand{\Ind}{{\mathrm{Ind}}}
\newcommand{\Lie}{{\mathrm{Lie}}}
\newcommand{\Mat}{{\mathrm{Mat}}}

\newcommand{\tr}{{\mathrm{tr}}}

\makeatletter

\newcommand{\Rmnum}[1]{\expandafter\@slowromancap\romannumeral #1@}
\makeatother

\begin{document}
\renewcommand{\theequation}{\arabic{equation}}
\numberwithin{equation}{section}

\title[Explicit Cohomological Test Vectors for $\GL_{2n}(\C)$]{Archimedean Non-vanishing, Cohomological Test Vectors, and Standard $L$-functions of $\GL_{2n}$: Complex Case}

\author{Bingchen Lin}
\address{School of Mathematics\\
Sichuan University, China}
\email{87928335@qq.com}

\author{Fangyang Tian}
\address{Department of Mathematics\\
National University of Singapore, Singapore}
\email{mattf@nus.edu.sg}


\date{\today}


\thanks{The research of B. Lin is supported in part by the China Scholarship Council No.201706245006, and that of F. Tian is supported in part by AcRF Tier 1 grant R-146-000-277-114 of National University of Singapore.}


\begin{abstract}
The purpose of this paper is to study the local zeta integrals of Friedberg-Jacquet at complex place and to establish similar results to the recent work \cite{ChenJiangLinTianExplicitCohomologicalVectorReal} joint with C. Chen and D. Jiang.
In this paper, we will
(1) give a necessary and sufficient condition on an irreducible essentially tempered cohomological representation $\pi$ of $\GL_{2n}(\C)$ with a non-zero Shalika model; (2) construct a new twisted linear period $\Lambda_{s,\chi}$ and give a different expression of the linear model for $\pi$; (3) give a necessary and sufficient condition on the character $\chi$ such that there exists a uniform cohomological test vector $v\in V_\pi$ (which we construct explicitly) for $\Lambda_{s,\chi}$. As a consequence, we obtain the non-vanishing of local Friedberg-Jacquet integral at complex place. All of the above are essential preparations for attacking a global period relation problem in the forthcoming paper(\cite{J-S-T}).
\end{abstract}

\maketitle


\tableofcontents


\section{Introduction}\label{Section: Introduction}

In 1993, S. Friedberg and H. Jacquet constructed a global integral in \cite{F-J} which relates to the standard $L$-function $L(s,\pi\otimes\chi)$ for an irreducible cuspidal automorphic representation $\pi$ of $\GL_{2n}$, twisted by an idelic character $\chi$, assuming that $\pi$ has a nonzero Shalika model. To be precise, the characters in this paper are not assumed to be unitary; if so, we call them unitary characters in order to be precise. The global integral admits an Euler product. Here we recall the setup for its archimedean local integral.

Let $F$ be either $\R$ or $\C$. Let $\omega$ be a character of $F^\times$, and $\psi$ be a nontrivial unitary character of $F$. We say that an irreducible Casselman-Wallach representation $\pi$ of $\GL_{2n}(F)$ has a non-zero $(\omega, \psi)$-Shalika model, if there exists a non-zero continuous linear functional $\lambda$ on the Fr\'{e}chet space $V_\pi$, which is called a {\sl Shalika functional}, such that
    \begin{equation}\label{Eq: Def of Shalika Functional}
      \langle \lambda, \pi(\mtrtwo{\RI_n}{Y}{}{\RI_n}\mtrtwo{g}{}{}{g})v \rangle =\omega(\det g) \langle \lambda, v\rangle\cdot \psi(\tr(Y)),
    \end{equation}
    for any $v\in V_\pi$, $g\in \GL_n(F)$ and any $n\times n$ matrix $Y\in\Mat_n(F)$. Here and henceforth, $\RI_n$ stands for the $n\times n$ identity matrix. It is clear that if \eqref{Eq: Def of Shalika Functional} holds, then $\omega^n$ is equal to the central character of $\pi$.

Now we assume that $\pi$ is an irreducible generic Casselman-Wallach representation of $\GL_{2n}(F)$ with a nonzero Shalika model. We fix a nonzero Shalika functional $\lambda$ on $V_\pi$ and take a character $\chi$ of $F^\times$. The archimedean local integral of Friedberg-Jacquet is
    \begin{equation}\label{Eq: Def of Local Integral}
       Z(v,s,\chi) = \int_{\GL_n(F)} \langle \lambda, \pi(\mtrtwo{g}{}{}{\RI_n})v \rangle \abs{\det g}_F^{s-\frac{1}{2}}\chi(\det g) dg
    \end{equation}
    for $v\in V_\pi$. It is proved in \cite[Theorem 3.1]{A-G-J} that
    \begin{enumerate}
      \item the local zeta integral $Z(v,s,\chi)$ converges absolutely when $\Re(s)$ is sufficiently large and has a meromorphic continuation in $s$ to the whole complex plane;
      \item it is a holomorphic multiple of the local $L$-function $L(s,\pi\otimes\chi)$ in the sense of meromorphic continuation;
      \item one can choose a smooth vector $v$ (not necessarily $K$-finite) such that the local integral $Z(v,s,\chi)$ is exactly the local $L$-function $L(s,\pi\otimes\chi)$.
    \end{enumerate}
    We note that when both $\chi$ and $\omega$ are trivial, the integral in \eqref{Eq: Def of Local Integral} is exactly the local integral considered in in \cite{A-G1}. Thus whenever $s=s_0$ is not a pole of $L(s,\pi\otimes\chi)$, $Z(v,s,\chi)$ has no pole at $s=s_0$. This implies that the map: $v \mapsto Z(v, s_0,\chi)$ defines a nonzero element in
    $$
    \text{Hom}_{H}(V_\pi, \abs{\det}_F^{-s_0+\frac{1}{2}}\chi^{-1}(\det)\otimes\abs{\det}_F^{s_0-\frac{1}{2}}(\chi\omega)(\det)),
    $$
    which is called the space of {\sl twisted linear functionals} of $\pi$. Here
    \begin{equation*}
       H = \Big\{\begin{pmatrix}g_1& \\ &g_2\end{pmatrix}\Big| g_1,g_2\in \GL_n(F) \Big\} \simeq \GL_n(F) \times \GL_n(F).
    \end{equation*}

    To obtain further arithmetic applications of the Friedberg-Jacquet integral, the representation $\pi$ is often taken to be a cohomological representation. Meanwhile, a lot of research works, such as the work of A. Ash and D. Ginzburg (\cite{A-G1}) on $p$-adic $L$-functions for $\GL_{2n}$ and the work of H. Grobner and A. Raghuram (\cite{G-R}) on arithmetic properties of the critical values of twisted standard $L$-functions of $\GL_{2n}$, are based on the \textbf{assumption} that the archimedean local zeta integral $Z(v,\frac{1}{2},\chi)$ is non-zero when $v$ is a cohomological vector of $\pi$. We recall that when $\pi$ is cohomological, a vector $v\in V_\pi$ is said to be cohomological if it lives in the minimal $K$-type $\tau$ of $\pi$. In a recent work of Binyong Sun (\cite{Sun}), he proves that in the real case, for each $s\in\BC$, the existence of a cohomological test vector $v_s\in V_\pi$, hence local Friedberg-Jacquet integral does not vanish when evaluating at $v_s$. However, the test vector he found is not uniform (in other words, the test vector depends on $s$). This is the key reason that his result is not strong enough to attack the global period relation for the twisted standard $L$-function at all critical places. On the other hand, Sun's work does not involve the non-vanishing of Friedberg-Jacquet integral at complex place. Naturally, we have the following two \textbf{problems}:
    \begin{itemize}
    \item Find a uniform cohomological test vector $v$ for the archimedean local zeta integral of Friedberg-Jacquet.
    \item Show the non-vanishing of the local zeta integral $Z(v,\frac{1}{2},\chi)$ in the complex case.
    \end{itemize}

    In our recent work joint with C. Chen and D. Jiang \cite{ChenJiangLinTianExplicitCohomologicalVectorReal}, we have shown that for each irreducible essentially tempered cohomological representation $\pi$ of $\GL_{2n}(\R)$, $\pi$ automatically has a nonzero Shalika model. Moreover, we explicitly construct a new twisted linear period $\Lambda_{s,\chi}$ for $\pi$ and a uniform cohomological test vector $v\in V_\pi$ for $\Lambda_{s,\chi}$, and hence recover and strengthen Sun's non-vanishing result (\cite[Theorem 5.1]{Sun}). In fact, the vector $v$ we construct is just the uniform cohomological test vector for archimedean local zeta integral of Friedberg-Jacquet. We will see this result in a forthcoming paper \cite{J-S-T}, where the authors of \cite{J-S-T} give a complete proof of the global period relation for the twisted standard $L$-functions at all critical places for irreducible, regular algebraic, cuspidal automorphic representations of $\GL_{2n}$ of (generalized) symplectic type. The goal of this paper is to 
    solve the above two problems in the complex case. In contrast to the real case,
    \begin{enumerate}
      \item \textbf{not all} irreducible essentially tempered cohomological representations of $\GL_{2n}(\C)$ has a non-zero Shalika model and we will classify all such representations with a non-zero Shalika model in Subsection \ref{subsection: classification};
      \item for irreducible essentially tempered cohomological representations $\pi$ of $\GL_{2n}(\C)$, the twisted linear model $\Lambda_{s,\chi}$, hence the local Friedberg-Jacquet integral $Z(v,s,\chi)$ \textbf{can be zero} when evaluating at a cohomological vector $v$. The non-vanishing property of $Z(v,s,\chi)$ will crucially depend on the data $\pi$ and $\chi$. Roughly speaking, when the restriction of the character $\chi$ on the unit circle $S^1$ is 'highly ramified', i.e. $\chi\lvert_{S^1}$ has the form $(\frac{z}{\abs{z}})^l$ with very large integer $l$ (in terms of absolute value), then the local integral vanishes on the minimal $K$-type. The exact statement will be explained in Theorem \ref{thm-main}.
    \end{enumerate}

    The method we use in this paper is similar to the original work of \cite{J-L-T} which uses differentiable operators and classical invariant theory to construct certain polynomials; yet in a revised version of the real case \cite{ChenJiangLinTianExplicitCohomologicalVectorReal}, the strategy is now changed. Hence we would like to keep the strategy of \cite{J-L-T} which might work for other situations. Comparing with the computation in the real case in \cite{J-L-T}, besides the above two new phenomena appearing in this paper, we also face a technical difficulty when solving the non-vanishing problem in the complex case. The difficulty roots in the finite dimensional representation of compact groups. In the complex case, we are able to construct a \textbf{scalar-valued} function (see Corollary \ref{cor: construction of cohomological vector}) that lives the minimal $K$-type of $\pi$ using the method in \cite{J-L-T}. Yet the twisted linear model which we will construct in \eqref{eq: 020}  by the same method as in \cite{J-L-T} requires a \textbf{vector-valued} function. For the real case, this problem was resolved in \cite{J-L-T} by writing the vector-valued function in terms of a basis of the vector space. However, in the complex case, this method is unrealistic, or at least will lead to very complicated computation. This is because the irreducible representations of $\mathrm{O}_2(\R)$ are at most two dimensional, yet the irreducible representation of $\mathrm{U}_2$ can have arbitrary dimension. To overcome this technical difficulty, we will keep a close track on the double induction formula between the two models of $\pi$ in \eqref{eq: 1051}.

    Finally, we would like to thank Dihua Jiang, Binyong Sun and Lei Zhang for helpful suggestions on writing this paper. This joint work was started during Lin's visit at the University of Minnesota. He would also like to take this opportunity to thank the School of Mathematics of UMN for offering excellent hospitality and working conditions.
\subsection{Cohomological representations of $\GL_{2n}(\BC)$}\label{sec-CRGL2n}
Let $G=\GL_{2n}(\C)$, $B$ be the standard Borel subgroup of $G$, and $K = \RU_{2n}$ be the standard maximal compact subgroup of $G$. Set
    \begin{equation}\label{def: subgroup H}
       H = \Big\{\mtrtwo{g_1}{}{}{g_2}\Big\lvert g_1,g_2\in \GL_n(\C) \Big\} \simeq \GL_n(\C) \times \GL_n(\C).
    \end{equation}
    Denoted by $Z$ the center of $G$. Let $d$ be the dimension of $\Lie(H)/\Lie((K\cap H)Z)$, where $\Lie(H)$ is the Lie algebra of $H$. Then
    \begin{equation}\label{Eq: Def of d}
       d =  2n^2-1,
    \end{equation}
    which, as indicated in \cite{A-G1}, is the dimension of the modular symbol generated by the closed subgroup $H$. To fix notation, from now on, we
    will use capital letters $G, H$ etc. for certain Lie groups, $G^0, H^0$ etc. for their identity components, German letters $\mathfrak{g}, \mathfrak{h}$ etc. for their Lie algebras, and $\mathfrak{g}^\C, \mathfrak{h}^\C$, etc. for the complexifications of Lie algebras.

    Let
    \begin{equation}\label{eq: dominant algebraic char}
        \nu: = (\nu_1\geq \nu_2\geq \cdots\geq \nu_{2n}; \nu_{2n+1}\geq \nu_{2n+2}\geq \cdots\geq \nu_{4n})
    \end{equation}
    be two sequence of integers in the decreasing order. Set $\rho_{1, \nu}$ and $\rho_{2,\nu}$ be two highest weight representations of $\GL_{2n}(\BC)$ (regarded as a complex algebraic group) with highest weights $(\nu_1\geq \nu_2\geq \cdots\geq \nu_{2n})$ and $(\nu_{2n+1}\geq \nu_{2n+2}\geq \cdots\geq \nu_{4n})$ respectively. Then we set $(\rho_\nu, F_\nu)$ to be the complex representation of the real algebraic group $G = \GL_{2n}(\BC)$ defined as $\rho_\nu(g) := \rho_{1,\nu}(g)\otimes \rho_{2,\nu}(\bar g).$ We call the $4n$-integers $\nu$ defined in \eqref{eq: dominant algebraic char} the highest weight of $F_\nu$.
    In this situation, we denote $\nu_{2n+j}$ by $\bar{\nu}_j$ for all $j=1,2,\cdots,2n$. The complex conjugate notation $\bar{\nu}_j$ is introduced here simply to indicate that $\bar{\nu}_j$ comes from a composition of a highest weight representation of $\GL_{2n}(\BC)$ with the complex conjugation.

    Let $(\pi, V_\pi)$ be an irreducible essentially tempered Casselman-Wallach representation of $G$ such that the total relative Lie algebra cohomology
    $$\RH^*(\mathfrak{g},K,\pi\otimes F_\nu^\vee)\ne 0,$$
    here and henceforth, $\vee$ stands for the contragredient representation. By abuse of notation, we also use $\pi$ for its underlying $(\mathfrak{g}, K)$-module when no confusion arises. Throughout this paper, we assume
    that $\nu$ satisfies the purity condition as in \cite{Clo}: there exists an integer $m\in \mathbb{Z}$ such that for all $j=1,2,\cdots,2n$,
    \begin{equation}\label{eq: purity condition}
         \bar{\nu}_j+\nu_{2n-j+1} = m.
    \end{equation}This is an essential condition for the existence of such a representation $\pi$. By \cite[Lemma 3.14]{Clo}, the $j^{th}$-cohomology group
    \begin{equation*}
        \RH^j(\mathfrak{g},K, \pi\otimes F_\nu^\vee)\ne 0,
    \end{equation*}
    if and only if
    \begin{equation*}
           2n^2-n \leq j \leq 2n^2+n-1.
    \end{equation*}
    In particular, we have
    \begin{equation*}
        \RH^d(\mathfrak{g},K,\pi\otimes F_\nu^\vee)\ne 0,
    \end{equation*}
    which coincides the assumption in \cite{A-G1} when $F_\nu$ is the trivial representation. Now we recall the Langlands parameter for $\pi$, which is discussed in \cite[Section 2.4.2]{R}.
    \begin{prop}\label{structure-pi}
     Let $(\pi, V_\pi)$ be an irreducible essentially tempered Casselman-Wallach representation of $G$ such that the $d$-th ($d=2n^2-1$) relative Lie algebra cohomology
     $$ \RH^d(\mathfrak{g},K,\pi\otimes F_\nu^\vee)\ne 0.$$
    Then $\pi$ is equivalent to the principal series representation
\begin{equation}\label{eq: cohomological rep parameter}
\Ind_{B}^{G} z^{\nu_1+\frac{2n-1}{2}}\bar{z}^{m-\nu_1-\frac{2n-1}{2}}\otimes z^{\nu_2+\frac{2n-3}{2}}\bar{z}^{m-\nu_2-\frac{2n-3}{2}}\otimes\cdots\otimes z^{\nu_{2n}+\frac{1-2n}{2}}\bar{z}^{m-\nu_{2n}-\frac{1-2n}{2}},
\end{equation}
     where $\nu$ and $m$ are described in \eqref{eq: dominant algebraic char} and \eqref{eq: purity condition} respectively.
\end{prop}

\subsection{Shalika model and linear model}\label{sec-SMLM}
     Unlike the real case discussed in \cite{J-L-T}, not all cohomological representations $\pi$ as given in \eqref{eq: cohomological rep parameter} have a non-zero Shalika model. In Subsection \ref{subsection: classification}, we will classify all irreducible cohomological generic representations of $\GL_{2n}(\C)$ (in Theorem \ref{thm: generic repns with Shalika models}) which have a non-zero Shalika model. The method we use should be considered as an archimedean analogue of a successive work of N. Matringe (see \cite{Mat1},\cite{Mat2}). In our scenario, for any integer $L$, we write $\chi_{L}$ the unitary character of $\C^\times$ sending $z$ to $(\frac{z}{\abs{z}})^{L}$. We set
    \begin{equation}\label{eq: l_i eq}
       l_j = 2\nu_j+(2n+1-2j)-m,
    \end{equation}
    then $(l_1,l_2,\cdots,l_{2n})$ is a sequence of integers in a strictly decreasing order such that
    \begin{equation}\label{eq: 1020}
        l_j+l_{2n+1-j} = 2\nu_j+2\nu_{2n+1-j}-2m
    \end{equation}
    is an even integer. We then rewrite the cohomological representation $\pi$ given in \eqref{eq: cohomological rep parameter} as
    \begin{equation}\label{Eq: pi parabolic induction parameter}
       \pi \simeq \Ind_{B}^{G} \abs{\quad}_\C^{\frac{m}{2}} \chi_{l_1}\otimes\abs{\quad}_\C^{\frac{m}{2}} \chi_{l_2}\otimes\cdots\otimes\abs{\quad}_\C^{\frac{m}{2}} \chi_{l_{2n}},
    \end{equation}
    where $|z|_{\mathbb{C}}=|z|^2=z\bar{z}$ for $z\in\mathbb{C}$. By applying Theorem \ref{thm: generic repns with Shalika models}, we immediately conclude that $\pi$ has a nonzero Shalika model if and only if there exists an integer $L$ with the property that
    \begin{equation}\label{eq: 1001}l_{j}+l_{2n+1-j}=2L,\end{equation}
    (or equivalently $\nu_j+\nu_{2n+1-j} = m+L$) for all $j=1,2,\cdots,2n$. If \eqref{eq: 1001} holds, then we can rewrite $\pi$ as the normalized parabolically induced representation
    \begin{equation}\label{Eq: pi parabolic induction parameter GL2}
       \begin{aligned}
       \pi &\simeq \Ind_{P}^{G}\, \sigma_1\otimes\sigma_2\otimes\cdots\otimes\sigma_n,
       \end{aligned}
    \end{equation}
    where $P$ is the standard parabolic subgroup of $G$ associated with the partition $[2^n]$, and each $\sigma_j$ is the principal series of $\GL_2(\C)$:
    \begin{equation}\label{Eq: pi parabolic induction parameter GL2 principal parameter}
        \sigma_j := \Ind^{\GL_2(\C)}_{B_{\GL_2}}\, \abs{\quad}^{\frac{m}{2}}_\C \chi_{l_j}\otimes\abs{\quad}^{\frac{m}{2}}_\C \chi_{l_{2n+1-j}}.
    \end{equation}
    All the principal series $\sigma_j$ share the same central character, denoted by $\omega$:
    \begin{equation}\label{eq: central char GL2}
         \omega(a\RI_2) = \abs{a}_\C^{m}\chi_{2L}(a),
    \end{equation}
    Here in the above, by abuse of notations, we also regard $\omega$ as a character of $\C^\times$ via the isomorphism $Z_{\GL_2(\C)}\simeq \C^\times$. The central character $\omega_\pi$ of $\pi$ takes the form
    \begin{equation}\label{eq: central char}
       \omega_\pi(a \RI_{2n}) = \abs{a}_\C^{mn}\chi_{2nL}(a).
    \end{equation}
    The maximal compact subgroup of $\C^\times$ is $\mathrm{U}_1$, which is isomorphic to the unit circle $S^1$. The restriction of the character $\omega$ on $S^1$ is just the character $\chi_{2L}$.

    Given a character $\chi$ of $\C^\times$, there exists an integer $l$ such that $\chi = \chi_l$ when restricted on $S^1$. With the model of $\pi$ as in \eqref{Eq: pi parabolic induction parameter GL2}, we can construct in Subsection \ref{Subsection: Another Linear Model} a nonzero twisted linear functional $\Lambda_{s,\chi}$ for $\pi$ without using the Shalika model.

\subsection{Cohomological vectors and non-vanishing property}\label{sec-MR} According to the work of F. Chen and B. Sun (see \cite[Theorem B]{Ch-Sun}), given a character $\chi$ of $\C^\times$, we can conclude that for all but countably many $s\in\BC$,
    \begin{equation}\label{Eq: uniqueness of twisted linear model}
        \text{dim Hom}_{H}(\pi, \abs{\det}_\C^{-s+\frac{1}{2}}\chi^{-1}(\det)\otimes\abs{\det}_\C^{s-\frac{1}{2}}(\chi\omega)(\det))\leq 1.
    \end{equation}
Thus, in order to prove that the archimedean local integral $Z(v,s,\chi)$
(as in \eqref{Eq: Def of Local Integral}, $F=\C$) does not vanish on the minimal $K$-type $\tau$ of $\pi$, it suffices to show that the same non-vanishing property for the twisted linear functional $\Lambda_{s,\chi}$ . The real case of this non-vanishing problem was solved in the paper \cite{Sun}. In the original work of \cite{J-L-T} joint with D. Jiang (and a published version \cite{ChenJiangLinTianExplicitCohomologicalVectorReal} with a simplified proof), we proved this non-vanishing property independently by giving an explicit construction of a cohomological vector $v_0$ of $\Lambda_{s,\chi}$. Such a cohomological vector $v_0$ is in fact independent on $s\in\BC$ and is further proved to be a uniform cohomological test vector of $Z(\,\cdot\,,s,\chi)$ in the paper \cite{J-S-T}. In this paper, we will imitate the method that we used in \cite{J-L-T} to construct a cohomological vector explicitly in the complex case. More precisely, for the sequence of integers $(l_1,l_2,\cdots,l_{2n})$ in the decreasing order that satisfies \eqref{eq: 1001}, we set $N_j = l_j-L$ (the integer $L$ is defined in \eqref{eq: 1001}). Then the cohomological representation $\pi$ in \eqref{Eq: pi parabolic induction parameter} is isomorphic to
\begin{equation}\label{eq: 1050}\begin{aligned}\pi&\simeq\Ind_{B}^{G}\abs{\quad}_{\C}^\frac{m}{2}\chi_{N_1+L}\otimes\abs{\quad}_{\C}^\frac{m}{2}\chi_{-N_1+L}\otimes\abs{\quad}_{\C}^\frac{m}{2}\chi_{N_2+L}\otimes\abs{\quad}_{\C}^\frac{m}{2}\chi_{-N_2+L}\otimes\cdots\otimes\\
                         &\qquad\qquad\abs{\quad}_{\C}^\frac{m}{2}\chi_{N_{n}+L}\otimes\abs{\quad}_{\C}^\frac{m}{2}\chi_{-N_{n}+L}.\end{aligned}\end{equation}

Using the model \eqref{eq: 1050}, we will construct an explicit cohomological vector $\varphi\in V_\pi$ in Corollary \ref{cor: construction of cohomological vector}, by restricting a polynomial function $F_{\vec{N},\chi_{-l}\otimes\chi_{l+2L}}$ constructed in Theorem \ref{thm: construction of bi-equivariant polynomial} on the maximal compact subgroup $K$. Via the isomorphism $\eta$ between the two models of $\pi$ (i.e.  \eqref{Eq: pi parabolic induction parameter GL2} and \eqref{eq: 1050})
\begin{equation}\label{eq: 1051}\begin{aligned}\eta: &\Ind_{P}^{G}\sigma_1\otimes\sigma_2\cdots\otimes\sigma_n\\
    \simeq &\Ind_{B}^{G}\abs{\quad}_{\C}^\frac{m}{2}\chi_{N_1+L}\otimes\abs{\quad}_{\C}^\frac{m}{2}\chi_{-N_1+L}\otimes\abs{\quad}_{\C}^\frac{m}{2}\chi_{N_2+L}\otimes\abs{\quad}_{\C}^\frac{m}{2}\chi_{-N_2+L}\otimes\cdots\otimes\\
                         &\qquad\qquad\abs{\quad}_{\C}^\frac{m}{2}\chi_{N_{n}+L}\otimes\abs{\quad}_{\C}^\frac{m}{2}\chi_{-N_{n}+L},\end{aligned}\end{equation}
we are able to evaluate $\Lambda_{s,\chi}(\eta^{-1}(\iota(\varphi)))$. The following Main Theorem of this paper holds:
\begin{thm}[Main Theorem]\label{thm-main}
Let $(\pi, V_\pi)$ be an irreducible essentially tempered Casselman-Wallach representation of $\GL_{2n}(\C)$ such that the total relative Lie algebra cohomology
    $$\RH^*(\mathfrak{g},K,\pi\otimes F_\nu^\vee)\ne 0.$$
Then there exists integers $m,L$ and a sequence of positive integers $(N_1,N_2,\cdots,N_n)$ in the strictly decreasing order such that \eqref{eq: 1050} holds. Given a character $\chi = \abs{\quad}_\C^{u_0}\chi_l$ of $\C^\times$.
\begin{enumerate}
  \item If the integer $l$ satisfies $\abs{l+L}>N_n$, then the local Friedberg-Jacquet integral $Z(v,s,\chi)$ defined in \eqref{Eq: Def of Local Integral} ($F=\C$) vanishes identically on the minimal $K$-type $\tau$ of $\pi$.
  \item If $l$ satisfies $\abs{l+L}\leq N_n$, there exists a cohomological vector $v = \eta^{-1}(\varphi)\in V_\tau$ and a holomorphic function $G(s,\chi)$ such that
$$
Z(v,s,\chi)=e^{G(s,\chi)}\cdot L(s,\pi\otimes\chi),
$$
where $\varphi$ is the cohomological vector explicitly constructed in Corollary \ref{cor: construction of cohomological vector} and $\eta$ is the isomorphism \eqref{eq: 1051}. In particular, $Z(v,s,\chi)\ne 0.$
\end{enumerate}
\end{thm}


\section{Cohomological Generic Representations with Shalika Models and Linear Models}\label{Section: Shalika Models and Linear Models}

\subsection{Cohomological Generic Representations of $\GL_{2n}(\C)$ with Shalika Models}\label{subsection: classification}
As we mentioned in the Introduction, not all cohomological representations $\pi$ given in \eqref{eq: cohomological rep parameter} (or equivalently in \eqref{Eq: pi parabolic induction parameter}) have a non-zero Shalika model. The first goal of this Section is to provide a necessary and sufficient condition for $\pi$ under which $\pi$ has a non-zero Shalika model. As in the Introduction, we write $\chi_{l}$ for the unitary character of $\C^\times$ sending $z$ to $(\frac{z}{\abs{z}})^{l}$.

\begin{thm}\label{thm: generic repns with Shalika models}
     Let $\pi$ be the cohomological representation given in \eqref{Eq: pi parabolic induction parameter} with a central character $\omega_\pi$. Given a nontrivial unitary character $\psi$ of $\C$ and a character $\omega$ of $\C^\times$ such that $\omega^n=\omega_\pi$, we have the following equivalent statements:
    \begin{enumerate}
      \item $\pi$ has a non-zero $(\omega, \psi)$-Shalika model defined at the beginning of the Introduction;
      \item There exists a discrete, countable subset $S\sbst\C$ such that for every complex number $s\notin S$, $\pi$ has a nonzero twisted linear model, i.e.
          \begin{equation}\label{eq: 1002}
    \mathrm{Hom}_{H}(V_\pi, \abs{\det}_\C^{-s+\frac{1}{2}}\chi^{-1}(\det)\otimes\abs{\det}_\C^{s-\frac{1}{2}}(\chi\omega)(\det))\ne 0.
    \end{equation}
      \item There exists an integer $L$ such that \eqref{eq: 1001} holds.
    \end{enumerate}
\end{thm}
The proof of Theorem \ref{thm: generic repns with Shalika models} is standard, yet lengthy. Let us first deal with the easy part of Theorem \ref{thm: generic repns with Shalika models}.
\begin{proof}[Proof of Theorem \ref{thm: generic repns with Shalika models}]

\textbf{(3)$\Rightarrow$(1):} Let $\sigma$ be a generic Casselman-Wallach representation of $\GL_2(\C)$ with a central character $\omega_\sigma$. We fix a nontrivial unitary character $\psi$ of $\BC$. Then $\sigma$ admits a Whittaker model $\mathcal{W}(\sigma,\psi)$, i.e. there exists a nonzero continuous linear functional $\lambda_\sigma$ on the Fr\'{e}chet space $V_\sigma$ such that
    \begin{equation*}
        \langle \lambda, \sigma(\mtrtwo{1}{x}{}{1})v \rangle = \psi(x)\langle \lambda, v\rangle.
    \end{equation*}
    Thus
    \begin{equation*}
        \langle \lambda, \sigma(\mtrtwo{1}{x}{}{1}\mtrtwo{a}{}{}{a})v \rangle = \omega_\sigma(a)\psi(x)\langle \lambda, v\rangle,
    \end{equation*}
    which exactly coincides with \eqref{Eq: Def of Shalika Functional} for $n=1$. Hence any such $\sigma$ has a non-zero $(\omega_\sigma, \psi)$-Shalika model.

    Now we assume that there exists an integer $L$ such that \eqref{eq: 1001} holds. Using the trick as in \eqref{Eq: pi parabolic induction parameter GL2}, we can write $\pi$ as
    $$ \pi \simeq \Ind_{P}^{G}\, \sigma_1\otimes\sigma_2\otimes\cdots\otimes\sigma_n,$$
    where $P$ is the standard parabolic subgroup corresponding to the partition $[2^n]$ and all $\sigma_j$ are generic representations of $\GL_2(\C)$ sharing the same central character $\omega$. Hence all $\sigma_j$ have nonzero Shalika models associated with the same characters $\omega$ and $\psi$. Thus by \cite[Thoerem 2.1]{ChenJiangLinTianExplicitCohomologicalVectorReal}, $\pi$ automatically has a nonzero $(\omega, \psi)$-Shalika model.

    \textbf{(1)$\Rightarrow$(2):} We assume that $\pi$ has a nonzero $(\omega, \psi)$-Shalika model. Then by \cite[Theorem 3.1]{A-G-J}, the local integral $Z(v,s,\chi)$ defines a nonzero element in
    \begin{equation*}
    \mathrm{Hom}_{H}(V_\pi,\, \abs{\det}_\C^{-s+\frac{1}{2}}\chi^{-1}(\det)\otimes\abs{\det}_\C^{s-\frac{1}{2}}(\chi\omega)(\det))
    \end{equation*}
    whenever $s$ is not a pole of the standard local $L$-function $L(s,\pi\otimes\chi)$. Thus, Statement \textbf{(2)} follows from the fact that the poles of the local $L$-function $L(s,\pi\otimes\chi)$ form a countable discrete subset of $\C$.
    \end{proof}

    Hence it remains to prove \textbf{(2)$\Rightarrow$(3)} in Theorem \ref{thm: generic repns with Shalika models}, which is the most difficult part of the proof. The main strategy is to use the Bruhat theory.
    Our main reference is \cite[Chapter 5]{WarnerHarmonicAnalysis} (also see the original work \cite{Br}). In the following, given any Lie group $M$ and two Casselman-Wallach representations $\pi_1$ and $\pi_2$ of $M$, we denote by $\mathrm{Bil}_M(\pi_1, \pi_2)$ the space of intertwining forms between $\pi_1$ and $\pi_2$.

    Let $w_0$ be the Weyl element which changes the sequence
    \begin{equation*}
        (1,2,3,\cdots,2n)
    \end{equation*}
    into
    \begin{equation*}
        (1,3,5,\cdots,2n-1,2,4,6,\cdots, 2n).
    \end{equation*}
    As usual, we will not distinguish an Weyl element with the permutation matrix representing it. Set
    \begin{equation}\label{eq: def of w}
        w := w_0^{-1}.
    \end{equation}
    Write $H = H_1\times H_2$, where $H_1 \simeq \GL_n(\BC)$ is located in the upper left corner; $H_2 \simeq \GL_n(\BC)$ is located in the lower right corner. To better apply the double coset computation in \cite{Mat1}, we introduce a conjugacy of $H$ in $G$. Define
    \begin{equation}\label{eq: 1013}
       H' = \mathrm{Ad}(w)H,\quad H'_1 = \mathrm{Ad}(w)H_1,\quad H'_2 = \mathrm{Ad}(w)H_2,
    \end{equation}
    and
    \begin{equation}\label{eq: shalika model epsilon}
         \epsilon = w\mtrtwo{\RI_n}{}{}{-\RI_n}w^{-1} = \mathrm{diag}(1,-1,1,-1,\cdots,1,-1).
    \end{equation}
    Then by \cite[Proposition 4.1.3.1]{WarnerHarmonicAnalysis} and the Reciprocity Law \cite[Theorem 5.3.3.1]{WarnerHarmonicAnalysis} (also see \cite[Theorem 6;4]{Br}),
    \begin{equation}\label{eq: 1003}
    \begin{aligned}
     &\mathrm{dim Hom}_{H}(V_\pi,\, \abs{\det}_\C^{-s+\frac{1}{2}}\chi^{-1}(\det)\otimes\abs{\det}_\C^{s-\frac{1}{2}}(\chi\omega)(\det)) \\
     = &\mathrm{dim Hom}_{H'}(V_\pi,\, \abs{\det}_\C^{-s+\frac{1}{2}}\chi^{-1}(\det)\otimes\abs{\det}_\C^{s-\frac{1}{2}}(\chi\omega)(\det))\\
     = &\mathrm{dim Bil}_{H'}(V_\pi,\, \abs{\det}_\C^{s-\frac{1}{2}}\chi(\det)\otimes\abs{\det}_\C^{-s+\frac{1}{2}}(\chi\omega)^{-1}(\det))\\
     =&\mathrm{dim Bil}_{G}(V_\pi,\, \Ind_{H'}^{G}\,\abs{\det}_\C^{s-\frac{1}{2}}\chi(\det)\otimes\abs{\det}_\C^{-s+\frac{1}{2}}(\chi\omega)^{-1}(\det)).
    \end{aligned}
    \end{equation}
    To simplify our notation, we set
    \begin{equation}\label{eq: shalika model classification sigma}
        \sigma  =  \abs{\quad}_\C^{\frac{m}{2}} \chi_{l_1}\otimes\abs{\quad}_\C^{\frac{m}{2}} \chi_{l_{2}}\otimes\cdots\otimes\abs{\quad}_\C^{\frac{m}{2}} \chi_{l_{2n}}
    \end{equation}
    and
    \begin{equation}\label{eq: shalika model classification eta}
        \eta(s,\chi,\omega) =  \abs{\det}_F^{s-\frac{1}{2}}\chi(\det)\otimes\abs{\det}_F^{-s+\frac{1}{2}}(\chi\omega)^{-1}(\det).
    \end{equation}
    Now we rewrite \eqref{eq: 1003} as
    \begin{equation}\label{eq: 1009}
        \begin{aligned}
     &\mathrm{dim Hom}_{H}(V_\pi,\, \abs{\det}_\C^{-s+\frac{1}{2}}\chi^{-1}(\det)\otimes\abs{\det}_\C^{s-\frac{1}{2}}(\chi\omega)(\det)) \\
     = &\mathrm{dim Bil}_{G}(\Ind_{B}^{G}\, \sigma,\,\Ind_{H'}^{G}\,\eta(s,\chi,\omega)).
     \end{aligned}
    \end{equation}
    The right-hand side of \eqref{eq: 1009} can be estimated by \cite[Theorem 5.3.2.3]{WarnerHarmonicAnalysis} (also see \cite[Theorem 6;3]{Br}). Some standard notations need to be introduced in the framework of Bruhat theory before we proceed the estimate.

    Since $H$ (hence $H'$) is a symmetric subgroup of $G$, it is spherical. Thus, there are finitely many $B$-orbits in $G/H'$. The representative of each $x\in B\quo G/H'$ is in fact computed in the work of \cite{Mat1}, and we will come back to describe these representatives in the next subsection. For each $x\in B\quo G/H'$, we set $H_x' := B\cap xH'x^{-1}$ and define $\eta(s,\chi,\omega)^x$ to be the character of $H'_x$ via the formula
       \begin{equation}\label{eq: 1010}\eta(s,\chi,\omega)^x(y) := \eta(s,\chi,\omega)(x^{-1}yx).\end{equation}
    As usual, we write $\delta_{H_x'}$ for the modular character of $H_x'$ and $\fh'$  (resp. $\fb$) for the Lie algebra of $H'$ (resp. $B$). Then we define some finite-dimensional representations of $H_x'$ coming from transversal derivatives (see \cite[Page 411]{WarnerHarmonicAnalysis}):
    \begin{enumerate}
      \item For all $x\in B\quo G/H'$, we set $\Lambda_{x,0}$ to be the trivial representation of $H_x'$.
      \item If $BxH'$ is an open orbit, then for all positive integers $k$, we set $\Lambda_{x,k}$ to be the zero vector space.
      \item If $BxH'$ is not an open orbit (which implies that the quotient vector space $\fg/(\fb+x\fh'x^{-1})$ is nonzero), then we set $\Lambda_{x,1}$ to be the adjoint representation of $H_x'$ on $\fg/(\fb+x\fh'x^{-1})$; and for all positive integers $k$, we set
          $\Lambda_{x,k}$ to be the $k$-th symmetric power $\Sym^k \Lambda_{x,1}$.
    \end{enumerate}
    Now we are ready to state the estimate directly obtained from Bruhat theory.
    \begin{prop}\label{prop: bruhat theory estimate}
     With the notations introduced as above, we have that
       \begin{equation}\label{temp: eq 2500}
        \begin{aligned}
      &\mathrm{dim Hom}_{H}(V_\pi,\, \abs{\det}_\C^{-s+\frac{1}{2}}\chi^{-1}(\det)\otimes\abs{\det}_\C^{s-\frac{1}{2}}(\chi\omega)(\det)) \\
     \leq&\sum_{x\in B\quo G/H'}\sum_{k=0}^{2\dim G +1} \dim\Bil_{H'_x}( \sigma\otimes \eta(s,\chi,\omega)^x,\, \delta_B^{\frac{1}{2}}\delta_{H_x'}^{-1}\otimes\Lambda_{x,k}^\vee).
     \end{aligned}
    \end{equation}
    \end{prop}
    \begin{proof}
       It follows directly from \cite[Theorem 5.3.2.3]{WarnerHarmonicAnalysis} and the remark below it (also see \cite[Theorem 6;3]{Br} and \cite[Equation 6;97]{Br}) that
       \begin{equation}\label{eq: estimate intertwining number}
       \begin{aligned}
          &\mathrm{dim Bil}_{G}(\Ind_{B}^{G}\, \sigma,\,\Ind_{H'}^{G}\,\eta(s,\chi,\omega))\\
          \leq&\sum_{x\in B\quo G/H'}\sum_{k=0}^{2 \dim G+1} \dim\Bil_{H'_x}( \sigma\otimes \eta(s,\chi,\omega)^x,\, \delta_B^{\frac{1}{2}}\delta_{H_x'}^{-1}\otimes\Lambda_{x,k}^\vee).
       \end{aligned}
       \end{equation}
       Here we emphasize that the notation $\delta$ in \cite{WarnerHarmonicAnalysis} is the \textbf{right} modular character (see \cite[Appendix 1, Page 474]{WarnerHarmonicAnalysis}), while the modular characters we use in this paper are all \textbf{left} modular characters, to be precise. When $k>0$ and $BxH'$ is an open orbit, we enforce $\Lambda_{x,k}$ to be a zero vector space as above in order for
       $$ \dim\Bil_{H'_x}( \sigma\otimes \eta(s,\chi,\omega)^x,\, \delta_B^{\frac{1}{2}}\delta_{H_x'}^{-1}\otimes\Lambda_{x,k}^\vee) =0.$$
       This exactly coincides with \cite[Theorem 5.3.2.3]{WarnerHarmonicAnalysis}. Now the Proposition follows from \eqref{eq: 1009} and \eqref{eq: estimate intertwining number}
    \end{proof}
    \begin{cor}\label{cor: temp 001}
       Retain the notations and assumptions in Theorem \ref{thm: generic repns with Shalika models}. If we further assume that Statement \textbf{(2)} of Theorem \ref{thm: generic repns with Shalika models} holds, then there exists an $x\in B\quo G/H'$ and some $k\geq 0$ such that for uncountably many $s\in\BC$,
       \begin{equation}\label{eq: temp 2700} \dim\Bil_{H'_x}( \sigma\otimes \eta(s,\chi,\omega)^x,\, \delta_B^{\frac{1}{2}}\delta_{H_x'}^{-1}\otimes\Lambda_{x,k}^\vee) \ne 0.\end{equation}
    \end{cor}
    \begin{proof}
        It is clear that Statement \textbf{(2)} of Theorem \ref{thm: generic repns with Shalika models} implies that \eqref{eq: 1002} holds for uncountably many $s\in\BC$. Note that the right-hand side of \eqref{eq: estimate intertwining number} is a finite sum. Then Corollary \ref{cor: temp 001} follows  directly from Proposition \ref{prop: bruhat theory estimate}.
    \end{proof}

    \subsection{Double Coset Decomposition $B\quo G/H'$}
    In view of Corollary \ref{cor: temp 001}, we aim to study all representatives $x\in B\quo G/H'$. This is already done in \cite{Mat1}. To be precise, N. Matringe computes the representatives $u_s^{-1}$ (see \cite[Proposition 3.2]{Mat1}) of $H'\quo G/B$ (the group $H'$ is denoted by $H$ in \cite{Mat1}). By taking the inverse, these $u_s$ are exactly the representatives of $B\quo G/H'$ we use in this paper. For consistency of notations, we will later write these representatives $u_s$ as $x_\mathcal{N}$. The notation $s$ is saved for a complex number for the local Friedberg-Jacquet integral. Now we provide the details.

    For $1\leq i,j\leq 2n, 1\leq t \leq 2n$, we consider nonnegative integers $n_{i,j}, n_{t,t}^+, n_{t,t}^-$ which satisfy the following properties:
    \begin{equation}\label{eq: temp 2600}
         n_{i,j}=n_{j,i},\quad n_{t,t} = n_{t,t}^{+}+n_{t,t}^-,\quad \sum_{j=1}^{2n} n_{i,j}= 1,\quad \sum_{t=1}^{2n} n_{t,t}^+ = \sum_{t=1}^{2n} n_{t,t}^-.
    \end{equation}
    Thus, the integer $2n$ has a partition
    \begin{equation}\label{eq: shalika model complex partition}
        2n = \sum_{i=1}^{2n}\sum_{j=1}^{2n} n_{i,j} = \sum_{1\leq i\ne j \leq 2n} n_{i,j} + \sum_{1\leq t\leq 2n} n_{t,t}^++n_{t,t}^-.
    \end{equation}
    By \cite[Proposition 3.2]{Mat1},  each element of $H'\backslash G/B$ corresponds to a sequence of the following form
    \begin{equation}\label{eq: 1011}
       \begin{aligned}\mathcal{N} = \set{n_{i,j},n_{t,t}^{+}, n_{t,t}^-}{&\text{ all } n_{i,j}, n_{t,t}^+, n_{t,t}^- \text{\ are non-negative integers satisfying }\\ &\text{ Property \eqref{eq: temp 2600} and }1\leq i, j\leq 2n, 1\leq t \leq 2n}.
       \end{aligned}
    \end{equation}
    For a matrix $A$, whenever $1\leq i\ne j\leq 2n$ or $1\leq t\leq 2n$, we can define the $(i,j)$-th row, $(t,t)^+$th row and $(t,t)^-$-th row of $A$ according to the partition \eqref{eq: shalika model complex partition}. Then it makes sense to say the $(i,j)$-th, $(t,t)^+$-th, $(t,t)^-$-th diagonal entry for a diagonal matrix $A$ in $G$. If $n_{i,j}=0$ ($n_{t,t}^+=0$, $n_{t,t}^-=0$ resp.), then there is no $(i,j)$-th ($(t,t)^+$-th, $(t,t)^-$-th resp.) row or diagonal entry. The representatives $x_{\mathcal{N}}^{-1}\in H'\quo G/B$ corresponding to the sequence $\mathcal{N}$ (defined in \eqref{eq: 1011}) can be chosen in the following way. We take a Weyl element $w_{\mathcal{N}}$ (i.e. a permutation matrix) such that $w_{\mathcal{N}}\epsilon w_{\mathcal{N}}^{-1}$ is the diagonal matrix (the matrix $\epsilon$ is defined in \eqref{eq: shalika model epsilon}) with the property that
    \begin{enumerate}
     \item for $i<j$, the $(i,j)$-th diagonal entry of $w_{\mathcal{N}}\epsilon w_{\mathcal{N}}^{-1}$ is 1; for $i>j$, the $(i,j)$-th diagonal entry of $w_{\mathcal{N}}\epsilon w_{\mathcal{N}}^{-1}$ is $-1$;
     \item the $(t,t)^+$-th diagonal entry of $w_{\mathcal{N}}\epsilon w_{\mathcal{N}}^{-1}$ is 1; the $(t,t)^-$-th diagonal entry of $w_{\mathcal{N}}\epsilon w_{\mathcal{N}}^{-1}$ is $-1$.
    \end{enumerate}
    We also take a block diagonal matrix $a_\mathcal{N}$ such that the $(t,t)^+$-th and $(t,t)^-$-th diagonal entries are 1, and for all $i<j$, the $((i,j),(j,i))$-th diagonal block is $\frac{1}{\sqrt{2}}\mtrtwo{1}{-1}{1}{1}$ (i.e. in the $(i,j)$-th row, the $(i,j)$-th entry is $\frac{1}{\sqrt{2}}$, and the $(j,i)$-th entry is $-\frac{1}{\sqrt{2}}$, the others are 0; while in the $(j,i)$-th row, the $(i,j)$-th and the $(j,i)$-th entries are both $\frac{1}{\sqrt{2}}$, the other entries are 0). Then we have the following representative $x_\mathcal{N}^{-1}\in H'\quo G/B$ with
    \begin{equation}\label{eq: 1014}
       x_\mathcal{N} := a_\mathcal{N}w_\mathcal{N}.
    \end{equation}
    Thus, we can take $x_\mathcal{N}$ to be the representatives of $B\quo G/H'$.

    With the representative $x=x_{\mathcal{N}}$, we write $H_{\mathcal{N}}'$ for the corresponding subgroup $H_x'$, that is, $$H'_{\mathcal{N}} := B \cap x_\mathcal{N}H'x_\mathcal{N}^{-1}.$$ Accordingly, we write $\delta_\mathcal{N}$ for the modular character of $H'_{\mathcal{N}}$ for simplicity. By \cite[Proposition 3.4]{Mat1}, the subgroup $H'_\mathcal{N}$ is a semidirect product of a diagonal subgroup $M_\mathcal{N}$ and a unipotent subgroup $U_\mathcal{N}$, where $M_\mathcal{N}$ consists of invertible matrices of the form:
    \begin{equation}\label{eq: 1012}
       m_\mathcal{N} = \mathrm{diag}(m_{1,1}^+,m_{1,1}^-,m_{1,2},\cdots,m_{1,2n},m_{2,1},m_{2,2}^+,m_{2,2}^-,m_{2,3},\cdots,m_{2,2n},\cdots,m_{2n,2n}^+,m_{2n,2n}^-),
    \end{equation}
    such that $m_{i,j}=m_{j,i}$ whenever $i\ne j$.

    To make the above description much more transparent, we provide an example in the case of $\GL_4(\BC)$ (that is $n=2$), which will also be useful in finishing the proof of Theorem \ref{thm: generic repns with Shalika models} in the next subsection.
    \begin{exmp}\label{example: GL4 computation}
       For $G = \GL_4(\C)$, the group $H'$ for $\GL_4(\C)$ consists of invertible matrices of the form
    $$\left(
\begin{array}{cccc}
 a & 0 & b & 0 \\
 0 & e & 0 & f \\
 c & 0 & d & 0 \\
 0 & g & 0 & h \\
\end{array}
\right).$$
   \textbf{Case 1}.
   Consider the sequence
   \begin{equation}\label{eq: set N 1}
       \mathcal{N} = \{n_{12}=n_{21}=1, n_{34}=n_{43}=1, \text{ all other } n_{ij}, n_{t,t}^+, n_{t,t}^- \text{ are zeroes}\}.
   \end{equation}
   Accordingly, we relabel the row number and column number of a $4\times 4$ matrix as follows
   \begin{align*}
   &\begin{array}{cccc}
     12 & 21 & 34 & 43
   \end{array}\\
   \begin{array}{c}
     12 \\
     21 \\
     34 \\
     43
   \end{array}
   &\left(
     \begin{array}{cccc}
       \star & \star & \star & \star \\
       \star & \star & \star & \star \\
       \star & \star & \star & \star \\
       \star & \star & \star & \star \\
     \end{array}
   \right).
   \end{align*}
   This means, for instance, that the usual second row is now called the $(2,1)$-th row, the usual third column is now called the $(3,4)$-th column. Then it is not hard to see from the descriptions of $w_\mathcal{N}$ and $a_{\mathcal{N}}$ that $w_\mathcal{N} = \RI_4$
   and
   $$a_\mathcal{N} = \left(
\begin{array}{cccc}
 \frac{1}{\sqrt{2}} & -\frac{1}{\sqrt{2}} & 0 & 0 \\
 \frac{1}{\sqrt{2}} & \frac{1}{\sqrt{2}} & 0 & 0 \\
 0 & 0 & \frac{1}{\sqrt{2}} & -\frac{1}{\sqrt{2}} \\
 0 & 0 & \frac{1}{\sqrt{2}} & \frac{1}{\sqrt{2}} \\
\end{array}
\right).$$
   Thus, by a direct matrix computation, the group $H'_{\mathcal{N}} = B\cap x_\mathcal{N}H'x_\mathcal{N}^{-1}$ consists of matrices of the form
   $$ \left(
\begin{array}{cccc}
 a & 0 & b & f \\
 0 & a & f & b \\
 0 & 0 & d & 0 \\
 0 & 0 & 0 & d \\
\end{array}
\right).$$

   \textbf{Case 2}.
   Consider the sequence
   \begin{equation}\label{eq: set N 2}\mathcal{N} = \{n_{13}=n_{31}=1, n_{24}=n_{42}=1, \text{ all other } n_{ij}, n_{t,t}^+, n_{t,t}^- \text{ are zeroes}\}.\end{equation}
   Accordingly, we relabel the row number and column number of a $4\times 4$ matrix as follows
   \begin{align*}
   &\begin{array}{cccc}
     13 & 24 & 31 & 42
   \end{array}\\
   \begin{array}{c}
     13 \\
     24 \\
     31 \\
     42
   \end{array}
   &\left(
     \begin{array}{cccc}
       \star & \star & \star & \star \\
       \star & \star & \star & \star \\
       \star & \star & \star & \star \\
       \star & \star & \star & \star \\
     \end{array}
   \right).
   \end{align*}
    Similarly, we have
   $$w_\mathcal{N} = \left(
\begin{array}{cccc}
 1 & 0 & 0 & 0 \\
 0 & 0 & 1 & 0 \\
 0 & 1 & 0 & 0 \\
 0 & 0 & 0 & 1 \\
\end{array}
\right),\qquad a_\mathcal{N} = \left(
\begin{array}{cccc}
 \frac{1}{\sqrt{2}} & 0 & -\frac{1}{\sqrt{2}} & 0 \\
 0 & \frac{1}{\sqrt{2}} & 0 & -\frac{1}{\sqrt{2}} \\
 \frac{1}{\sqrt{2}} & 0 & \frac{1}{\sqrt{2}} & 0 \\
 0 & \frac{1}{\sqrt{2}} & 0 & \frac{1}{\sqrt{2}} \\
\end{array}
\right).$$
   Then the group $H'_{\mathcal{N}}$ consists of matrices of the form
   $$ \left(
\begin{array}{cccc}
 a & b & 0 & 0 \\
 0 & d & 0 & 0 \\
 0 & 0 & a & b \\
 0 & 0 & 0 & d \\
\end{array}
\right).$$

   \textbf{Case 3}.
   Consider the sequence
   \begin{equation}\label{eq: set N 3}
       \mathcal{N} = \{n_{14}=n_{41}=1, n_{23}=n_{32}=1, \text{ all other } n_{ij}, n_{t,t}^+, n_{t,t}^- \text{ are zeroes}\}.
   \end{equation}
   Accordingly, we relabel the row number and column number of a $4\times 4$ matrix as follows
   \begin{align*}
   &\begin{array}{cccc}
     14 & 23 & 32 & 41
   \end{array}\\
   \begin{array}{c}
     14 \\
     23 \\
     32 \\
     41
   \end{array}
   &\left(
     \begin{array}{cccc}
       \star & \star & \star & \star \\
       \star & \star & \star & \star \\
       \star & \star & \star & \star \\
       \star & \star & \star & \star \\
     \end{array}
   \right).
   \end{align*}
   In this case, we have
   $$w_\mathcal{N} = \left(
\begin{array}{cccc}
 1 & 0 & 0 & 0 \\
 0 & 0 & 1 & 0 \\
 0 & 1 & 0 & 0 \\
 0 & 0 & 0 & 1 \\
\end{array}
\right),\qquad a_\mathcal{N} = \left(
\begin{array}{cccc}
 \frac{1}{\sqrt{2}} & 0 & 0 & -\frac{1}{\sqrt{2}} \\
 0 & \frac{1}{\sqrt{2}} & -\frac{1}{\sqrt{2}} & 0 \\
 0 & \frac{1}{\sqrt{2}} & \frac{1}{\sqrt{2}} & 0 \\
 \frac{1}{\sqrt{2}} & 0 & 0 & \frac{1}{\sqrt{2}} \\
\end{array}
\right).$$
   Then the group $H'_{\mathcal{N}}$ consists of matrices of the form
   $$ \left(
\begin{array}{cccc}
 a & 0 & 0 & 0 \\
 0 & d & 0 & 0 \\
 0 & 0 & d & 0 \\
 0 & 0 & 0 & a \\
\end{array}
\right).$$
   \end{exmp}

   Let us end this subsection with the following Lemma.
   \begin{lemma}\label{lemma: temp 002}
      Recall that $\delta_\mathcal{N}$ is the modular character of $H'_\mathcal{N}$. Then for $n=2$ and the sequences $\mathcal{N}$ given in \eqref{eq: set N 1}, \eqref{eq: set N 2}, \eqref{eq: set N 3}, the modular character $\delta_B^{-\frac{1}{2}}\delta_\mathcal{N}$ is trivial when restricted to the diagonal subgroup $M_\mathcal{N}$ (see \eqref{eq: 1012}).
   \end{lemma}
   \begin{proof}
       This Lemma can be checked by direct computations in a case by case fashion. We omit the details.
   \end{proof}

   \subsection{Finishing the proof of Theorem \ref{thm: generic repns with Shalika models}}
   In this Subsection, we will finish our proof of Theorem \ref{thm: generic repns with Shalika models}. Let us start with a simple Lemma.
   \begin{lemma}\label{lemma: temp 003}
   With the above choice of $x_\mathcal{N}$ in \eqref{eq: 1014}, for any matrix $m_\mathcal{N}$ as in \eqref{eq: 1012}, the conjugation $(m_\mathcal{N}\mapsto x_\mathcal{N}^{-1} m_\mathcal{N} x_\mathcal{N})$ maps
    \begin{enumerate}
      \item the $(t,t)^+$-th diagonal entry $m_{t,t}^+$ of $m_{\mathcal{N}}$ onto a diagonal entry of $H_1'$ (defined in \eqref{eq: 1013});
      \item the $(t,t)^-$-th diagonal entry $m_{t,t}^-$ of $m_\mathcal{N}$ onto a diagonal entry of $H_2'$ (defined in \eqref{eq: 1013});
      \item the $(i,j)$-th diagonal entry $m_{i,j}$ onto a diagonal entry (denoted by $m_{i,j}'$ for later use) of $H_1'$ whenever $i<j$;
      \item the $(i,j)$-th diagonal entry $m_{i,j}$ onto a diagonal entry (denoted by $m_{i,j}'$ for later use) of $H_2'$ whenever $i>j$.
     \end{enumerate}
   \end{lemma}
   \begin{proof}
      For all $i< j$, the $((i,j), (j,i))$-th diagonal block of $m_\mathcal{N}$ is $\mtrtwo{m_{i,j}}{}{}{m_{j,i}}$, which commutes with the $((i,j), (j,i))$-th diagonal block of $a_\mathcal{N}$ (since $m_{i,j} = m_{j,i}$). This implies that
      $$ x_\mathcal{N}^{-1} m_\mathcal{N} x_\mathcal{N} = w_\mathcal{N}^{-1}a_\mathcal{N}^{-1} m_\mathcal{N} a_\mathcal{N}w_\mathcal{N} = w_\mathcal{N}^{-1} m_\mathcal{N} w_\mathcal{N}.$$
      Then the Lemma follows directly from the definition of $w_\mathcal{N}$.
   \end{proof}
   \begin{prop}\label{prop: temp 004}
      Retain the notations and assumptions in Theorem \ref{thm: generic repns with Shalika models}. If we further assume that Statement \textbf{(2)} of Theorem \ref{thm: generic repns with Shalika models} holds, then there exists a sequence $\mathcal{N}$ as in \eqref{eq: 1011} and an integer $k\geq 0$ satisfying the following properties.
      \begin{enumerate}
        \item All $n_{t,t}^+ = n_{t,t}^- = 0$ ($1\leq t\leq 2n$).
        \item Denote by $\tilde{\Lambda}_{\mathcal{N},k}$ the restriction of $\delta_B^{-\frac{1}{2}}\delta_{\mathcal{N}}\otimes\Lambda_{x_\mathcal{N},k}$ to the subgroup $M_\mathcal{N}$. There exists a one-dimensional subrepresentation of $\tilde{\Lambda}_{\mathcal{N},k}$ denoted by $\xi$ such that
            \begin{equation}\label{eq: 1052}
                 \sigma(m_\mathcal{N})\cdot \prod_{1\leq i<j\leq 2n}\frac{1}{\omega( m_{i,j})} = \xi(m_\mathcal{N})
            \end{equation}
            holds for all $m_\mathcal{N}\in M_\mathcal{N}$.
      \end{enumerate}
      As a consequence, there exists a partition of the set $\{1,2,\cdots\}$ into exact $n$ pairs
      $$(i_1,j_1), (i_2,j_2), \cdots, (i_n,j_n) \qquad i_k<j_k \textrm{ for all $k=1,2,\cdots,n$}$$
      such that only $m_{i_k,j_k},m_{j_k,i_k}$ ($k=1,2,\cdots,n$) appear in $m_\mathcal{N}$.
   \end{prop}
   \begin{proof}
     By Corollary \ref{cor: temp 001}, there exists a sequence $\mathcal{N}$ given as in \eqref{eq: 1011} and some integer $k\geq 0$ such that for uncountably many $s\in\BC$, \eqref{eq: temp 2700} holds with $x = x_\mathcal{N}$ defined in \eqref{eq: 1014}.
     Now we will show that this sequence $\mathcal{N}$ and the integer $k$ satisfy the properties in Proposition \ref{prop: temp 004}.

     By restriction to the subgroup $M_\mathcal{N}$, we conclude from \eqref{eq: temp 2700} that for uncountably many $s\in\BC$, there exists a nonzero $M_\mathcal{N}$-intertwining operator between $\sigma\otimes \eta(s,\chi,\omega)^{x_\mathcal{N}}$ and $\tilde{\Lambda}_{\mathcal{N},k}$. Since $M_\mathcal{N}$ is abelian, the finite-dimensional representation $\tilde{\Lambda}_{\mathcal{N},k}$ is completely reducible. Thus, the character $\sigma\otimes \eta(s,\chi,\omega)^{x_\mathcal{N}}$ must be equal to a one-dimensional subrepresentation of $\tilde{\Lambda}_{\mathcal{N},k}$, denoted by $\xi$.

     By Lemma \ref{lemma: temp 003}, we have
     \begin{equation}\label{eq: 9997}
        \begin{aligned}
           &\eta(s,\chi,\omega)^{x_\mathcal{N}}(m_\mathcal{N})\\
           =&\eta(s,\chi,\omega)(x_\mathcal{N}^{-1}\cdot m_\mathcal{N}\cdot x_\mathcal{N})\\
           =&\Big(\prod_{t=1}^{2n}\frac{\abs{m_{t,t}^{+}}_\C^{s-\frac{1}{2}}\chi(m_{t,t}^+)}{\abs{m_{t,t}^{-}}_\C^{s-\frac{1}{2}}\chi( m_{t,t}^-)}\cdot\frac{1}{\omega(m_{t,t}^-)}\Big)\cdot\Big(\prod_{1\leq i<j\leq 2n}\frac{\abs{ m_{i,j}}_\C^{s-\frac{1}{2}}\chi(m_{i,j})}{\abs{m_{j,i}}_\C^{s-\frac{1}{2}}\chi(m_{j,i})}\cdot\frac{1}{\omega(m_{j,i})}\Big), \\
        \end{aligned}
     \end{equation}
     which can be simplified (by using $m_{i,j}=m_{j,i}$) as
     \begin{equation}\label{eq: 1016}
         \Big(\prod_{t=1}^{2n}\frac{\abs{m_{t,t}^{+}}_\C^{s-\frac{1}{2}}\chi( m_{t,t}^+)}{\abs{m_{t,t}^{-}}_\C^{s-\frac{1}{2}}\chi( m_{t,t}^-)}\cdot\frac{1}{\omega(m_{t,t}^-)}\Big)\cdot\Big(\prod_{1\leq i<j\leq 2n}\frac{1}{\omega( m_{i,j})}\Big).
     \end{equation}

     If $n_{t,t}^+$ or $n_{t,t}^-$ were nonzero for some $1\leq t \leq 2n$, then from \eqref{eq: 1016}, the restriction of the character $\eta(s,\chi,\omega)^{x_\mathcal{N}}$ on $M_\mathcal{N}$ is a character genuinely depending on $s\in\BC$. Note that there are only finitely many one-dimensional subrepresentations of $\tilde{\Lambda}_{\mathcal{N},k}$. For each one-dimensional subrepresentation $\xi_0$ of $\tilde{\Lambda}_{\mathcal{N},k}$, since $\xi_0$ is independent of $s\in\BC$,
     the equation of $s$
       \begin{equation}\label{eq: 1015}\sigma\otimes \eta(s,\chi,\omega)^{x_\mathcal{N}}(m_\mathcal{N}) = \xi_0(m_\mathcal{N}),\qquad \textrm{ for any } m_\mathcal{N}\in M_\mathcal{N},\end{equation}
     has only countably many solutions in $\BC$. This implies that for only countably many $s\in\BC$, \eqref{eq: temp 2700} holds. It is in contradiction to Corollary \ref{cor: temp 001}. Thus, $n_{t,t}^+ = n_{t,t}^- = 0$ for all $1\leq t\leq 2n.$

     In \eqref{eq: 1016}, $n_{t,t}^+ = n_{t,t}^- = 0$ for all $1\leq t \leq 2n$ means that all $m_{t,t}^+, m_{t,t}^-$ do not appear in the diagonal matrix $m_\mathcal{N}$. Hence
     \begin{equation}\label{eq: temp 1052}
      \eta(s,\chi,\omega)^{x_\mathcal{N}}(m_\mathcal{N}) =  \prod_{1\leq i<j\leq 2n}\frac{1}{\omega( m_{i,j})}.
     \end{equation}
     Now the second statement follows directly from \eqref{eq: 1015} and \eqref{eq: temp 1052}, with $\xi_0$ replaced by $\xi$.

     With the fact that $n_{t,t}^+ = n_{t,t}^- = 0$ for all $1\leq t \leq 2n$, it is now clear from the identity $\sum_{j=1}^{2n} n_{i,j}= 1$ that either $n_{i,j}=n_{j,i}=1$ or $n_{i,j}=n_{j,i}=0$. Moreover, if $n_{i,j}=n_{j,i}=1,$ then for all $t\ne j$, $n_{i,t} = n_{t,i} = 0$. This directly leads to a partition of $\{1,2,\cdots,2n\}$ in the final statement of this Proposition.
     \end{proof}

     \begin{lemma}\label{lemma: temp 007}
         Let $\mathcal{N}$ be the sequence in Proposition \ref{prop: temp 004} and $k$ be an integer satisfying the second statement of Proposition \ref{prop: temp 004}. Then for every one-dimensional subrepresentation $\xi$ of $\Lambda_{x_\mathcal{N},k}\lvert_{M_\mathcal{N}}$, there exist integers $u_{i,j}$ $(1\leq i< j\leq 2n)$ satisfying
         \begin{equation}\label{eq: 9999}\sum_{i,j} u_{i,j}=0\end{equation}
         such that when $m_\mathcal{N}\in M_\mathcal{N}$ is written as \eqref{eq: 1012},
         \begin{equation}\label{eq: 1054} \xi(m_\mathcal{N}) = \prod_{1\leq i<j\leq 2n}m_{i,j}^{u_{i,j}}.\end{equation}
         Here in \eqref{eq: 1054}, we adopt the convention that if $m_{i,j}$ does not appear in $m_\mathcal{N}$, then $u_{i,j}=0$.
     \end{lemma}
     \begin{proof}
     Note that when $k=0,$ by definition, $\Lambda_{x_\mathcal{N},0}$ is always a trivial representation, hence the Lemma holds trivially. If $k>0$, then $Bx_\mathcal{N}H'$ is not an open orbit (otherwise $\Lambda_{x_\mathcal{N},k}$ is a zero vector space).          Consider the restriction of $\Lambda_{x_\mathcal{N},1} = \fg/(\fb+x_\mathcal{N}\fh'x_\mathcal{N}^{-1})$ to the abelian subgroup $M_\mathcal{N}$. It is isomorphic to a subrepresentation of the adjoint representation of $M_\mathcal{N}$ on the Lie algebra $\mathfrak{n}^-$, the maximal lower nilpotent subalgebra of $\fg$. Hence every one-dimensional subrepresentation $\xi$ of $\Lambda_{x_\mathcal{N},1}\lvert_{M_\mathcal{N}}$ is a character of the form \eqref{eq: 1054}. For general $k>0$, since $\Lambda_{x_\mathcal{N},k} = \Sym^k \Lambda_{x_\mathcal{N},1}$, every one-dimensional subrepresentation $\xi$ of $\Lambda_{x_\mathcal{N},1}\lvert_{M_\mathcal{N}}$ is also a character of the form \eqref{eq: 1054}. To see that the integers $u_{i,j}$ satisfy \eqref{eq: 9999}, we observe that the center $Z$ of $G$ is contained in $H_\mathcal{N}' = B\cap x_\mathcal{N}H'x_\mathcal{N}^{-1}.$ Thus, $Z$ is a subgroup of $M_\mathcal{N}$. Since the center $Z$ acts trivially on $\mathfrak{n}^-$ under the adjoint action, by restricting \eqref{eq: 1054} to $Z$, we see that
     $$1 = \prod_{1\leq i<j\leq 2n} \, z^{u_{i,j}} = z^{\sum u_{i,j}} \qquad \textrm{ for all } z\in\BC.$$ Thus, \eqref{eq: 9999} holds.
     \end{proof}

     \begin{prop}\label{prop: temp 005}
         Let $\mathcal{N}$ be the subset in Proposition \ref{prop: temp 004}. Then the restriction of the character $\delta_B^{-\frac{1}{2}}\delta_\mathcal{N}$ on the diagonal subgroup $M_\mathcal{N}$ is trivial.
     \end{prop}
     \begin{proof}
         For $\GL_4(\C)$ (i.e. $n=2$), all possible sequences $\mathcal{N}$ satisfying $n_{t,t}^+ = n_{t,t}^- = 0$ (for all $1\leq t \leq 2n$)  are listed in Example \ref{example: GL4 computation}. Hence for $n=2$, this proposition is just Lemma \ref{lemma: temp 002}, which can be checked in a case by case fashion. Now we aim to prove the proposition for general $n$ based on the case $n=2$.

         By the last statement in Proposition \ref{prop: temp 004}, we can group the $2n$ diagonal entries of $m_{\mathcal{N}}$ in exact $n$ pairs in the form $(m_{i,j}, m_{j,i})$ with $m_{i,j}=m_{j,i}$ ($i<j$).
         Write $m'_{i,j}$ for the diagonal entry of $H_1'$ (when $i<j$) or the diagonal entry of $H_2'$ (when $i>j$) as in Lemma \ref{lemma: temp 003}.
         Now consider two distinct such pairs $(m_{i_1,j_1}, m_{j_1,i_1})$ and $(m_{i_2,j_2}, m_{j_2,i_2})$ ($i_1<j_1, i_2<j_2$).
         Accordingly, we write $m'_{i_1,j_1}, m'_{i_2,j_2}$ (resp. $m'_{j_1,i_1}, m'_{j_2,i_2}$) for the corresponding diagonal entries of $H_1'$ (resp. $H_2'$).
         These four diagonal entries $m_{i_1,j_1}, m_{i_2,j_2}, m_{j_1,i_1}, m_{j_2,i_2}$ are located in four different rows and columns. The intersection of these four rows and columns (i.e. a total of sixteen entries) determine a subgroup (denoted by $\GL_{\{(i_1,j_1),(i_2,j_2)\}}$) of $G$ isomorphic to $\GL_4(\C)$.
         For example, when $G = \GL_6(\C)$ and the four diagonal entries $m_{i_1,j_1} = m_{j_1,i_1} = a, m_{i_2,j_2} = m_{j_2,i_2} = d$ are given in the following:
     $$\left(
\begin{array}{cccccc}
 a & 0 & 0 & 0 & 0 & 0 \\
 0 & \star & 0 & 0 & 0 & 0 \\
 0 & 0 & \star & 0 & 0 & 0 \\
 0 & 0 & 0 & a & 0 & 0 \\
 0 & 0 & 0 & 0 & d & 0 \\
 0 & 0 & 0 & 0 & 0 & d \\
\end{array}
\right),$$
     then the subgroup $\GL_{\{(i_1,j_1),(i_2,j_2)\}}$ consists of all invertible matrices of the form
     $$\left(
\begin{array}{cccccc}
 \star & 0 & 0 & \star & \star & \star \\
 0 & 1 & 0 & 0 & 0 & 0 \\
 0 & 0 & 1 & 0 & 0 & 0 \\
 \star & 0 & 0 & \star & \star & \star \\
 \star & 0 & 0 & \star & \star & \star \\
 \star & 0 & 0 & \star & \star & \star \\
\end{array}
\right).$$
     Set $$B_{\{(i_1,j_1),(i_2,j_2)\}}:= B\cap \GL_{\{(i_1,j_1),(i_2,j_2)\}}$$ to be the standard upper Borel subgroup of $\GL_{\{(i_1,j_1),(i_2,j_2)\}}$. The diagonal entries $m_{i_1,j_1}', m_{i_2,j_2}'$, $m_{j_1,i_1}', m_{j_2,i_2}'$ are the nontrivial diagonal entries of the subgroup $w_\mathcal{N}^{-1}\GL_{\{(i_1,j_1),(i_2,j_2)\}}w_{\mathcal{N}} \simeq \GL_4(\BC)$. We see from Lemma \ref{lemma: temp 003} that $H_{{\{(i_1,j_1),(i_2,j_2)\}}}':= H'\cap w_\mathcal{N}^{-1}\GL_{\{(i_1,j_1),(i_2,j_2)\}}w_{\mathcal{N}}$ is a symmetric subgroup of $w_\mathcal{N}^{-1}\GL_{\{(i_1,j_1),(i_2,j_2)\}}w_{\mathcal{N}}$ isomorphic to $\GL_2(\BC)\times \GL_2(\BC).$ Moreover, since it is easy to see (from the definition of $a_\mathcal{N}$) that $a_\mathcal{N}$ normalizes $\GL_{\{(i_1,j_1),(i_2,j_2)\}}$, we get
     \begin{equation}\label{eq: temp 2800}
         \begin{aligned}x_\mathcal{N}H_{{\{(i_1,j_1),(i_2,j_2)\}}}'x_\mathcal{N}^{-1} &= x_\mathcal{N}H'x_\mathcal{N}^{-1}\cap a_\mathcal{N}\GL_{\{(i_1,j_1),(i_2,j_2)\}}a_{\mathcal{N}}^{-1}\\
          &= x_\mathcal{N}H'x_\mathcal{N}^{-1}\cap \GL_{\{(i_1,j_1),(i_2,j_2)\}}.\end{aligned}\end{equation}
     Write $\delta_{\{(i_1,j_1),(i_2,j_2)\}}$ for the modular character of $B_{\{(i_1,j_1),(i_2,j_2)\}} \cap x_\mathcal{N}H_{{\{(i_1,j_1),(i_2,j_2)\}}}'x_\mathcal{N}^{-1}$. We note from \eqref{eq: temp 2800} that
     \begin{equation}\label{eq: temp 2900}
          B_{\{(i_1,j_1),(i_2,j_2)\}} \cap x_\mathcal{N}H_{{\{(i_1,j_1),(i_2,j_2)\}}}'x_\mathcal{N}^{-1} = B_{\{(i_1,j_1),(i_2,j_2)\}} \cap x_\mathcal{N}H'x_\mathcal{N}^{-1}.
     \end{equation}
     This implies that the character $\delta_\mathcal{N}$, being the modular character of $H_\mathcal{N}' = B\cap x_\mathcal{N}H'x_\mathcal{N}^{-1}$, is equal to the product of all $\delta_{\{(i_1,j_1),(i_2,j_2)\}}$ with $(i_1,j_1), (i_2,j_2)$ running over all distinct pairs of the pairs in the final statement of Proposition \ref{prop: temp 004}.

     Now we can use Lemma \ref{lemma: temp 002} to compute the contribution of the two pairs $(m_{i_1,j_1},m_{j_1,i_1})$ and $(m_{i_2,j_2}, m_{j_2,i_2})$ in $\delta_B^{-\frac{1}{2}}(m_\mathcal{N})\delta_\mathcal{N}(m_\mathcal{N})$. Indeed, there are exactly three cases on the positions of these four diagonal entries in the matrix $m_\mathcal{N}$, namely:
     \begin{enumerate}
       \item $(\cdots,m_{i_1,j_1}=a,\cdots,m_{j_1,i_1}=a,\cdots,m_{i_2,j_2}=d,\cdots,m_{j_2,i_2}=d,\cdots)$;
       \item $(\cdots,m_{i_1,j_1}=a,\cdots,m_{i_2,j_2}=d,\cdots,m_{j_1,i_1}=a,\cdots,m_{j_2,i_2}=d,\cdots)$;
       \item $(\cdots,m_{i_1,j_1}=a,\cdots,m_{i_2,j_2}=d,\cdots,m_{j_2,i_2}=d,\cdots,m_{j_1,i_1}=a,\cdots)$.
     \end{enumerate}
     They correspond to the three small cases in Example \ref{example: GL4 computation}.
     Then according to Lemma \ref{lemma: temp 002}, we have
     \begin{equation}\label{eq: 9990}
          \left(\delta_B^{-\frac{1}{2}}\delta_\mathcal{N}\right)(m_\mathcal{N}) = \prod_{\text{all pairs }\{(i_1,j_1),(i_2,j_2)\}} \left(\delta_{B_{\{(i_1,j_1),(i_2,j_2)\}}}^{-\frac{1}{2}}\delta_{\{(i_1,j_1),(i_2,j_2)\}}\right)(m_\mathcal{N}) = 1.
     \end{equation}
     \end{proof}

     \begin{proof}[Proof of Theorem \ref{thm: generic repns with Shalika models}, \textbf{(2)$\Rightarrow$(3)}]

     By Proposition \ref{prop: temp 004}, there exists a sequence $\mathcal{N}$ satisfying the two properties listed in Proposition \ref{prop: temp 004}. By Proposition \ref{prop: temp 005}, the representation $\tilde\Lambda_{\mathcal{N},k}$ is equal to the restriction of $\Lambda_{x_\mathcal{N},k}$ on $M_\mathcal{N}$. Hence, the character $\xi$ of $M_\mathcal{N}$ in \eqref{eq: 1052} is in fact a one-dimensional subrepresentation of $\Lambda_{x_\mathcal{N},k}\lvert_{M_\mathcal{N}}$. Now, by Lemma \ref{lemma: temp 007}, we can write $\xi$ as in \eqref{eq: 1054} with $u_{i,j}$ satisfying \eqref{eq: 9999}. Thus, \eqref{eq: 1052} can be rewritten as
     \begin{equation}\label{eq: temp 3000}
                 \sigma(m_\mathcal{N})\cdot \prod_{1\leq i<j\leq 2n}\frac{1}{\omega( m_{i,j})} = \prod_{1\leq i<j\leq 2n}m_{i,j}^{u_{i,j}}
            \end{equation}
     with $\sum_{i,j} u_{i,j} = 0$.
     Let us write the character $\omega$ of $\BC^\times$ as
     \begin{equation}\label{eq: 1019}\omega(z) = \chi_{L'}(\frac{z}{\abs{z}})\abs{z}_\C^u,\end{equation}
     with $L'\in\BZ$ and $u\in\BC$. Then by taking the absolute value on both sides of \eqref{eq: temp 3000}, we see that

     \begin{equation}\label{eq: 9980}
         \prod_{1\leq i\ne j \leq 2n} \abs{m_{i,j}}_\BC^{\frac{m}{2}}\cdot \prod_{1\leq i<j\leq 2n} \abs{m_{i,j}}_\BC^{-u} = \prod_{1\leq i<j\leq 2n}\abs{m_{i,j}}_\BC^{u_{i,j}}.
     \end{equation}
     By using $m_{i,j} = m_{j,i}$ and comparing the exponents, we see that if $m_{i,j}$ appears in $m_\mathcal{N}$, then $m-u = u_{i,j}$. The property that $\sum_{i,j} u_{i,j} = 0$ directly implies that $m = u$, and hence further implies that all $u_{i,j} = 0$. Thus, \eqref{eq: temp 3000} can be rewritten as
     \begin{equation}\label{eq: temp 3100}
                 \sigma(m_\mathcal{N})\cdot \prod_{1\leq i<j\leq 2n}\frac{1}{\omega( m_{i,j})} = 1.
            \end{equation}

     By the last statement of Proposition \eqref{prop: temp 004}, the diagonal entries of $m_\mathcal{N}$ can be regrouped in exactly $n$ pairs
     $(m_{i_k,j_k}, m_{j_k,i_k})$ $(i_k<j_k, k=1,2,\cdots, n)$ with $m_{i_k,j_k} = m_{j_k,i_k}$. Now by \eqref{eq: shalika model classification sigma} and \eqref{eq: 1019} (with $u=m$), we can rewrite \eqref{eq: temp 3100} as
     \begin{equation}\label{eq: temp 3200}
                 \prod_{k=1}^{n} \chi_{l_{(i_k,j_k)}}(m_{i_k,j_k})\cdot \chi_{l_{(j_k,i_k)}}(m_{j_k,i_k}) = \prod_{k=1}^n \chi_{L'}(m_{i_k,j_k}),
            \end{equation}
     where the set $\set{l_{(i_k,j_k)}, l_{(j_k,i_k)}}{k=1,2,\cdots,n} = \{l_1,l_2,\cdots,l_{2n} \textrm{ given in }\eqref{eq: shalika model classification sigma}\}$. By using $m_{i_k,j_k} = m_{j_k,i_k}$ and comparing the characters on both sides of \eqref{eq: temp 3200}, we see that $l_{i_k,j_k}+l_{j_k,i_k} = L'$ for all $k=1,2,\cdots,n$. In other words, the $2n$ integers $l_1 > l_2 \cdots >l_{2n}$ in a strictly decreasing order can be partitioned into exactly $n$ pairs such that the sum of the two integers in each pair are all equal. The only possible partition is to pair the largest integer $l_1$ with the smallest integer $l_{2n}$, then pair the second largest integer $l_2$ with the second smallest integer $l_{2n-1}$, etc.. Therefore, $l_j+l_{2n+1-j} = L'$ for all $j=1,2,\cdots,n$. Now by \eqref{eq: 1020}, $L'$ is an even integer. The theorem now follows once we write $L'$ as $2L$ with $L\in\BZ$.
     \end{proof}

\subsection{Construct Another Linear Model}\label{Subsection: Another Linear Model}


In this Subsection, we will retain the notations in Subsection \ref{subsection: classification}. By Theorem \ref{thm: generic repns with Shalika models}, if $\pi$ is an irreducible essentially tempered cohomological representation of $G$ with a nonzero Shalika model, then $\pi$ must be isomorphic to the normalized parabolically induced representation
    \begin{equation}\label{eq: 1022}
       \pi \simeq \Ind_{P}^{G}\, \sigma_1\otimes\sigma_2\otimes\cdots\otimes\sigma,
    \end{equation}
where $P$ is the standard parabolic subgroup of $G$ associated with the partition $[2^n]$, and all $\sigma_j$ are principal series of $\GL_2(\C)$ sharing the same central character $\omega$. By Theorem \ref{thm: generic repns with Shalika models}, given any character $\chi$ of $F^\times$, $\pi$ must have a nonzero twisted linear model. Such a twisted linear functional can be given by the local Friedberg-Jacquet integral. The goal of this Subsection is to construct a twisted linear functional for $\pi$ without using the Shalika model of $\pi$. The method we use is just the complex version of \cite[Section 3]{J-L-T}, as we only deal with cohomological representations of $\GL_{2n}(\R)$ there.

Given a generic representation $\sigma$ of $\GL_2(\C)$ with a central character $\omega$ and a nonzero Whittaker model $\mathcal{W}(\sigma,\psi)$, the archimedean local integral
\begin{equation}\label{eq: local int GL(2,R)}
       \lambda_{s,\chi,\sigma}(v) := \int_{\C^\times} W_v(\mtrtwo{a}{}{}{1})\abs{a}_\C^{s-\frac{1}{2}}\chi(a) d^\times a.
\end{equation}
has a meromorphic continuation to the whole complex plane. It is a holomorphic multiple of the $L$-function $L(s,\sigma\otimes \chi)$. Whenever $s = s_0$ is not a pole of the $L-$function $L(s, \sigma\otimes \chi)$, $\lambda_{s_0,\chi,\sigma}$ defines a nonzero continuous linear functional in
\begin{equation*}
     \text{Hom}_{\GL_1(\C)\times \GL_1(\C)}(\sigma, \abs{\quad}_\C^{\frac{1}{2}-s_0}\chi^{-1}\otimes\abs{\quad}_\C^{s_0-\frac{1}{2}}\omega\chi).
\end{equation*}

Let $\pi$ be the representation of $G$ as in \eqref{eq: 1022}, with each $\sigma_j$ being a principal series of $\GL_2(\C)$ with a central character $\omega$. We fix a character $\chi$ of $\C^\times$ and write $\lambda_{s,j} := \lambda_{s,\chi,\sigma_j}$ for the nonzero continuous linear functional in
\begin{equation*}
     \text{Hom}_{\GL_1(\C)\times \GL_1(\C)}(\sigma_j, \abs{\quad}_\C^{\frac{1}{2}-s}\chi^{-1}\otimes\abs{\quad}_\C^{s-\frac{1}{2}}\omega\chi)
\end{equation*}
defined via the local integral $\eqref{eq: local int GL(2,R)}$. The following construction of the twisted linear functional for $\pi$ is almost a word by word repetition of \cite[Subsection 3.2]{J-L-T}. For convenience of the reader, we will outline the key steps.

    Let $w$ be the Weyl element as in \eqref{eq: def of w} and $P=MU$ be the Levi decomposition of $P$. Define $H' = \text{Ad}(w)H$ as in \eqref{eq: 1013}. Then $H'$ is also isomorphic to $\GL_n(\C)\times \GL_n(\C)$. The standard maximal unipotent subgroup $N$ of $H$ is mapped onto the standard maximal unipotent subgroup $N'$ of $H'$. We notice that $N'$ is a subgroup of $U$. Let us write
    \begin{equation*}
        N = N_1\times N_2.
    \end{equation*}
    Then $N_1$ and $N_2$ are both standard maximal unipotent subgroup of $\GL_n(\C)$. To simplify notation, we set
\begin{equation}\label{def: chi1 and chi2}
    \begin{aligned}
    \chi_{1,s}(a) &= \abs{a}^{\frac{1}{2}-s}_\C\chi^{-1}(a),\\
    \chi_{2,s}(a) &= \abs{a}^{s-\frac{1}{2}}_\C\omega(a)\chi(a).
    \end{aligned}
\end{equation}

    Take $\varphi\in V_\pi$, which is a smooth function on $G$ taking value in $$V_{\sigma_1}\hat{\otimes}\cdots\hat{\otimes} V_{\sigma_n},$$ the projective tensor product of Fr\'{e}chet spaces $V_{\sigma_j}$. The continuous linear functional $$\bigotimes_{i=1}^n \lambda_{s,i}$$ on the algebraic tensor space $$V_{\sigma_1}\otimes\cdots\otimes V_{\sigma_n}$$ has a continuous linear extension to the projective tensor product space. By abuse of notation, we will still use $$\bigotimes_{i=1}^n \lambda_{s,i}$$ for such an extension.  We consider a function on $H$ defined by
\begin{equation}\label{Eq: def of F on H}
    F_s(h; \varphi) = \langle \bigotimes_{i=1}^n \lambda_{s,i}, \varphi(wh)\rangle.
\end{equation}
Then $F_s(h; \varphi) = F_s(\textrm{I}_{2n}; \pi(h)\varphi)$. Since $N'$ is a subgroup of $U$, it is clear that for $n_1\in N_1$, $n_2\in N_2$,
\begin{equation*}
    \begin{aligned}
    F_s(\mtrtwo{n_1}{}{}{n_2}h; \varphi)= F_s(h; \varphi).
    \end{aligned}
\end{equation*}
By the equivariance of $\lambda_{s,i}$, it is easy to check the equivariance of $F_s(h; \varphi)$ on the torus:
\begin{equation}\label{Eq: equiv on torus}
    \begin{aligned}
    F_s(\mtrtwo{a_1}{}{}{a_2}h; \varphi)
    = \delta_P^{\frac{1}{2}}(w\mtrtwo{a_1}{}{}{a_2}w^{-1})\cdot\chi_{1,s}(\det a_1)\cdot\chi_{2,s}(\det a_2)F_s(h; \varphi),
    \end{aligned}
\end{equation}
where $a_1$ and $a_2$ are diagonal matrices in $\GL_n(\C)$. Let $B_1$ and $B_2$ be the standard Borel subgroups of $\GL_n(\C)$ having unipotent radicals $N_1$, $N_2$ resp.. Then
\begin{equation}
    \delta_P^{\frac{1}{2}}(w\mtrtwo{a_1}{}{}{a_2}w^{-1}) = \delta_{B_1}(a_1)\delta_{B_2}(a_2).
\end{equation}
Thus, $F_s(h;\varphi)$ satisfies a $B_1\times B_2$ equivariant property:
\begin{equation}\label{Eq: equivariant property of F}
     F_s(\mtrtwo{b_1}{}{}{b_2}h; \varphi) = \delta_{B_1}(b_1)\delta_{B_2}(b_2)\chi_{1,s}(\det b_1)\cdot\chi_{2,s}(\det b_2)\cdot F_s(h; \varphi),
\end{equation}
where $b_1, b_2$ live in $B_1, B_2$ respectively. By \cite[Lemma 5.1.1.4]{WarnerHarmonicAnalysis}, there exists a test function $f(h; \varphi)\in C_c^\infty(H)$ such that
\begin{equation}\label{eq: 1025}
   \begin{aligned}
       F_s(h; \varphi) = \int_{B_1\times B_2} f(\mtrtwo{b_1}{}{}{b_2}h; \varphi)\chi_{1,s}^{-1}(\det b_1)\chi_{2,s}^{-1}(\det b_2) d_lb_1d_lb_2.
   \end{aligned}
\end{equation}
We define a linear functional $\Lambda_{s,\chi}$ on $V_\pi$ according to the formula:
   \begin{equation}\label{eq: def of linear functional}
        \begin{aligned}
        \Lambda_{s,\chi}(\varphi):=&\int_H f(\mtrtwo{h_1}{}{}{h_2};\varphi)\chi_{1,s}^{-1}(\det h_1)\chi_{2,s}^{-1}(\det h_2) dh_1dh_2\\
         =&\int_{K\cap H} F_s(\mtrtwo{k_1}{}{}{k_2}; \varphi)\chi_{1,s}^{-1}(\det k_1)\chi_{2,s}^{-1}(\det k_2)) dk_1dk_2,
        \end{aligned}
   \end{equation}
where $K$ is the maximal compact subgroup of $G$. The linear functional $\Lambda_{s,\chi}$ is well defined, i.e. it is a convergent integral and independent on the choice of $f(h; \varphi)$. It also defines an element in $\text{Hom}_H(\pi, \chi_{1,s}(\det)\otimes\chi_{2,s}(\det))$.
In terms of $\varphi$, $\Lambda_{s,\chi}$ can be written as:
\begin{equation}\label{eq: new def H inv linear functional}
   \Lambda_{s,\chi}(\varphi) = \int_{K\cap H} \langle \bigotimes_{i=1}^n \lambda_{s,i}, \varphi(w\mtrtwo{k_1}{}{}{k_2}\rangle \chi_{1,s}^{-1}(\det k_1)\chi_{2,s}^{-1}(\det k_2)) dk_1dk_2
                    = \langle \bigotimes_{i=1}^n \lambda_{s,i}, \tilde{\varphi}(w) \rangle,
\end{equation}
where $\tilde{\varphi}$ is obtained by averaging $\varphi$ against the character $\chi_{1,s}^{-1}(\det k_1)\chi_{2,s}^{-1}(\det k_2)$ over the compact group $K\cap H$. In particular, if $\varphi$ satisfies the right $K\cap H$-equivariant property:
\begin{equation}\label{eq: right equivariant property}
    \varphi(g\mtrtwo{k_1}{}{}{k_2}) = \chi_{1,s}(\det k_1)\cdot\chi_{2,s}(\det k_2)\cdot \varphi(g),
\end{equation}
or equivalently (since the absolute value of the determinant is $1$ on $K$)
\begin{equation}\label{eq: right equivariant property simplified}
    \varphi(g\mtrtwo{k_1}{}{}{k_2}) = \chi^{-1}(\det k_1)\cdot (\omega\chi)(\det k_2))\cdot \varphi(g),
\end{equation}
then
\begin{equation}\label{eq: 020}
    \Lambda_{s,\chi}(\varphi) = \langle \bigotimes_{i=1}^n \lambda_{s,i}, \tilde{\varphi}(w) \rangle =\langle \bigotimes_{i=1}^n \lambda_{s,i}, \varphi(w) \rangle.
\end{equation}

\begin{prop}\label{Cor: analytic property of Lambda}
    For every $\varphi\in V_\pi$, $\Lambda_{s,\chi}(\varphi)$ defined by \eqref{eq: new def H inv linear functional} has a meromorphic continuation in $s$ to the whole complex plane. It is a holomorphic multiple of $L(s,\pi\otimes\chi)$. $\Lambda_{s,\chi}$ defines a nonzero element in
    \begin{equation*}
              \Hom_H(\pi, \chi_{1,s}(\det)\otimes\chi_{2,s}(\det)) = \Hom_{H}(\pi, \abs{\det}_\C^{-s+\frac{1}{2}}\chi^{-1}(\det)\otimes\abs{\det}_\C^{s-\frac{1}{2}}(\chi\omega)(\det))
    \end{equation*}
    whenever $s$ is not a pole of the $L$-function $L(s,\pi\otimes\chi)$. In particular, one can choose $\varphi$ such that $\Lambda_{s,\chi}(\varphi) = L(s,\pi\otimes\chi)$
\end{prop}
\begin{proof}
   Since each $\lambda_{s,j}(v)$ has a meromorphic continuation to the whole complex plane and $\lambda_{s,j}(v)$ is a holomorphic multiple of $L(s,\sigma_j\otimes\chi)$, in view of \eqref{eq: 020}, we conclude that for any $\varphi$ in $V_\pi$, $\Lambda_{s,\chi}(\varphi)$ has a meromorphic continuation to the whole complex plane and is a holomorphic multiple of
   $$\prod_{j=1}^n L(s,\sigma_j\otimes\chi) = L(s,\pi\otimes\chi).$$
   It defines an element in \begin{equation*}
              \Hom_H(\pi, \chi_{1,s}(\det)\otimes\chi_{2,s}(\det)) = \Hom_{H}(\pi, \abs{\det}_\C^{-s+\frac{1}{2}}\chi^{-1}(\det)\otimes\abs{\det}_\C^{s-\frac{1}{2}}(\chi\omega)(\det)).
    \end{equation*}
   We only need to prove the non-vanishing result.

   We write $\bar{N}_1,\bar{N}_2$ for lower unipotent subgroups of $\GL_n(\C)$ opposite to $N_1,N_2$. Then by Bruhat decomposition, we find that
   \begin{equation}\label{eq: 1024}
       \begin{aligned}
       \Lambda_{s,\chi}(\varphi) = &\int_{\bar{N}_1\times\bar{N}_2}\int_{B_1\times B_2} f(\mtrtwo{b_1\bar{n}_1}{}{}{b_2\bar{n}_2}; \varphi)\chi_{1,s}^{-1}(\det b_1\bar{n}_1)\chi_{2,s}^{-1}(\det b_2\bar{n}_2)) d_lb_1d_lb_2d\bar{n}_1d\bar{n}_2\\
       = &\int_{\bar{N}_1\times\bar{N}_2}F_s(\mtrtwo{\bar{n}_1}{}{}{\bar{n}_2}; \varphi)d\bar{n}_1d\bar{n}_2
       \end{aligned}
   \end{equation}
   whenever the restriction of $F_s(h; \varphi)$ on $\bar{N}_1\times\bar{N}_2$ has compact support. The adjoint action $\mathrm{Ad}(w)$ maps $\bar{N}_1\times\bar{N}_2$ into a unipotent subgroup (denoted by $\bar{N}'$) of $\bar{U}$, the lower unipotent subgroup opposite to $U$. Thus, if we assume that the restriction of $\pi(w)\varphi$ on $\bar{N}'$ has compact support, then
   \begin{equation}\label{eq: 1026}
       \begin{aligned}
       \Lambda_{s,\chi}(\varphi) = &\int_{\bar{N}_1\times\bar{N}_2}\langle \bigotimes_{i=1}^n \lambda_{s,i}, \varphi(w\mtrtwo{\bar{n}_1}{}{}{\bar{n}_2})\rangle d\bar{n}_1d\bar{n}_2\\
                                 = & \int_{\bar{N'}} \langle \bigotimes_{i=1}^n \lambda_{s,i}, (\pi(w)\varphi)(\bar{n'})\rangle d\bar{n'}. \end{aligned}\end{equation}
    Each $\lambda_{s,j}$ is a nonzero continuous linear functional on $V_{\sigma_j}$ and we can choose a smooth  vector $v_j\in V_{\sigma_j}$ such that $\lambda_{s,j}(v_j) = L(s,\sigma_j\otimes\chi)$. We take any positive test function $\phi\in C^\infty_c(\bar{U})$ such that its restriction to $\bar{N}'$ is in $C_c^\infty(\bar{N}')$ and $\phi$ satisfies
   $$\int_{\bar{N}'} \phi(\bar{n}')d\bar{n}' = 1.$$ Define
   $$\tilde{\phi}(u) = \phi(u)v_1\otimes v_2\otimes \cdots \otimes v_n.$$ We can extend $\tilde{\phi}$ to a smooth function on $G$ in $V_\pi$ according to the left $P$-equivariance of $\pi$. Finally, we set $\varphi = \pi(w^{-1})\tilde{\phi}$. With such a choice of $\varphi$, from \eqref{eq: 1026}, we get
   \begin{equation*}\begin{aligned}\Lambda_{s,\chi}(\varphi) &= \int_{\bar{N'}} \langle \bigotimes_{i=1}^n \lambda_{s,i}, (\pi(w)\varphi)(\bar{n'})\rangle d\bar{n'}\\
   &= \int_{\bar{N}'}\langle \bigotimes_{i=1}^n \lambda_{s,i}, \phi(\bar{n}')v_1\otimes v_2\otimes\cdots\otimes v_{n}\rangle d\bar{n}'\\
   &= \int_{\bar{N}'} \phi(\bar{n}')d\bar{n}' \cdot \prod_{i=1}^n \lambda_{s,i}(v_i) = L(s,\pi\otimes\chi).\end{aligned}\end{equation*}
   We are done.
\end{proof}


\section{Cohomological Vectors in the Induced Representation}\label{Section: Cohomological vector realization}


   The goal of this Section is to explicitly construct a cohomological vector of $\pi$ with the desired property for Theorem \ref{thm-main}. To this end, we work with the classical invariant theory and representations of compact Lie groups. Let us begin with some reductions and outline our strategy on the construction of the function in the minimal $K$-type. Throughout this Section, we will use the bold letter $\boldsymbol{i}$ for $\sqrt{-1}$.
   \subsection{Some reductions}\label{subsection: some reductions}
    First, we briefly recall the notations in Subsection \ref{sec-CRGL2n} and Subsection \ref{sec-SMLM}. Let $B$ be the standard Borel subgroup of $G=\GL_{2n}(\C)$, and $T$ the split torus of $G$ contained in $B$. Then $T$ is a product of $2n$ copies of $\C^\times$, and $T\cap K$ is a product of $2n$ copies of $\mathrm{U}_1$. For any integer $N$, we write $\chi_{N}$ the unitary character of $\C^\times$ sending $z$ to $(\frac{z}{\abs{z}})^{N}$. By abuse of notations, we also regard $\chi_N$ as a character of $\mathrm{U}_1$ by restriction.

   Let $\pi$ be the irreducible generic cohomological representation given in \eqref{Eq: pi parabolic induction parameter} with $(l_1,\cdots,l_{2n})$ being a sequence of integers in the decreasing order that satisfies \eqref{eq: 1001}. We set $N_j = l_j-L$ (the integer $L$ is defined in \eqref{eq: 1001}). Then $(N_1,N_2,\cdots, N_{2n})$ is a sequence of integers in the strictly decreasing order satisfying $N_j+N_{2n+1-j}=0$ ($j=1,2,\cdots,n$). Automatically, this implies that $N_1>N_n>\cdots>N_n>0$. The representation $\pi$ can be rewritten as  \begin{equation}\label{eq: 1021}\begin{aligned}\pi&=\Ind_{B}^{G}\abs{\quad}_{\C}^\frac{m}{2}\chi_{N_1+L}\otimes\cdots\otimes\abs{\quad}_{\C}^\frac{m}{2}\chi_{N_{2n}+L}\\
                         &=\Ind_{B}^{G}\abs{\quad}_{\C}^\frac{m}{2}\chi_{N_1+L}\otimes\cdots\otimes\abs{\quad}_{\C}^\frac{m}{2}\chi_{N_{n}+L}\otimes\abs{\quad}_{\C}^\frac{m}{2}\chi_{-N_{n}+L}\otimes\cdots\otimes\abs{\quad}_{\C}^\frac{m}{2}\chi_{-N_{1}+L},\end{aligned}\end{equation}
   which is isomorphic to
   \begin{equation}\label{Eq: permutation of induced representation}
       \begin{aligned}
   \Ind_{B}^{G}&\abs{\quad}_{\C}^\frac{m}{2}\chi_{N_1+L}\otimes\abs{\quad}_{\C}^\frac{m}{2}\chi_{-N_1+L}\otimes\abs{\quad}_{\C}^\frac{m}{2}\chi_{N_2+L}\otimes\abs{\quad}_{\C}^\frac{m}{2}\chi_{-N_2+L}\otimes\cdots\otimes\\
   &\abs{\quad}_{\C}^\frac{m}{2}\chi_{N_{n}+L}\otimes\abs{\quad}_{\C}^\frac{m}{2}\chi_{-N_{n}+L}.
       \end{aligned}
   \end{equation}
   We aim to construct a suitable function in the minimal $K$-type of $\pi$, i.e. a cohomological vector, using the model \eqref{Eq: permutation of induced representation}.

    As before, $K = \RU_{2n}$ is the standard maximal compact subgroup of $G$. Let $\tau$ be the minimal $K$-type of $\pi$. By Frobenius reciprocity law, the $K$-types occuring in $\pi$ are just the ones whose restriction to $T$ contains the character $\chi_{N_1+L}\otimes\cdots\otimes\chi_{N_{n}+L}\otimes\chi_{-N_{n}+L}\otimes\cdots\otimes\chi_{-N_{1}+L}$. Hence the highest weight of the minimal $K$-type $\tau$ is
    \begin{equation}\label{eq: def of weight}
    \begin{aligned}\Lambda:=&(N_1+L,N_2+L,\cdots,N_{2n}+L) \\
                             =& (N_1+L,N_2+L,\cdots,N_n+L,-N_n+L,\cdots, -N_{1}+L).\end{aligned}
    \end{equation}
    \cite[Proposition 8.1]{V2} confirms that $\tau$ is also the minimal $K$-type of the representation \begin{equation}\label{eq: pi K}\pi_K := \Ind_{T\cap K}^{K}\chi_{N_1+L}\otimes\chi_{-N_1+L}\otimes\chi_{N_2+L}\otimes\chi_{-N_2+L}\otimes\cdots\otimes\chi_{N_{n}+L}\otimes\chi_{-N_{n}+L}.\end{equation}

    For simplicity, we set
    \begin{equation}\label{eq: def of N}
    \vec{N}:=(N_1+L,-N_{1}+L,\cdots,N_n+L,-N_n+L)
    \end{equation} and $$\chi_{\vec{N}}:=\chi_{N_1+L}\otimes\chi_{-N_{1}+L}\otimes\cdots\otimes\chi_{N_n+L}\otimes\chi_{-N_n+L}.$$ Then every function in $\pi_K=\Ind_{T\cap K}^{K}\chi_{\vec{N}}$ belongs to the space $C^{\infty}(K)$ of smooth functions on $K$ satisfying the left equivariant property given by $\chi_{\vec{N}}$. Inspired by this, we may regard $C^{\infty}(K)$ as a left $T\cap K$- and right $K$-module under the left and right regular actions. As a $(T\cap K)\times K$-module, $C^\infty(K)_{\textrm{fin}}$, the space of $K$-finite vectors of $C^\infty(K)$, is completely reducible and decomposed into a direct sum:
    \begin{equation}
        \begin{aligned}
        C^\infty(K)_{\textrm{fin}} &= \bigoplus_{(\chi,\eta)\in \widehat{T\cap K}\times \widehat{K}} m_{\chi,\eta}\chi\otimes\eta
                          &= \bigoplus_{\chi\in \widehat{T\cap K}} \big(\bigoplus_{\eta\in \widehat{K}} m_{\chi,\eta}\chi\otimes \eta\big).
        \end{aligned}
    \end{equation}
    where as usual, $\widehat K$ and $\widehat{T\cap K}$ stand for the unitary dual of the corresponding compact group, $m_{\chi,\eta}$ is the multiplicity of $(T\cap K)\times K$-submodule $\chi\otimes\eta$ occurring in $C^\infty(K)_{\textrm{fin}}$.

    Thus $\bigoplus_{\eta\in \widehat{K}} m_{\chi_{\vec{N}},\eta}\chi_{\vec{N}}\otimes \eta$ is the space of $K$-finite vectors of $\pi_K$ and the minimal $K$-type of $\pi_K$ is an irreducible summand in this space with highest weight $(-\vec{N})\times\Lambda$, i.e.
    \begin{equation}\label{eq: highest weight of bi-module}
       \begin{aligned}
    &(-N_1-L,N_{1}-L,\cdots,-N_n-L,N_n-L)\\
    \times &(N_1+L,N_2+L,\cdots,N_n+L,-N_n+L,\cdots,-N_{1}+L).\end{aligned}
    \end{equation}
    Here the left irreducible $T\cap K$-module is given by the tensor product of $2n$ characters on $\textrm{U}_1$, and we write its highest weight as the form $-\vec{N}$.

    Let $\Mat_{2n}$ be the complex vector space of $2n\times 2n$ complex matrices and $\Mat_{2n}^\R$ be its realification (i.e. $\Mat_{2n}^\R$ is a $8n^2$-dimensional real vector space). Equip $C^{\infty}(\Mat_{2n}^{\R})$ a left $T\cap K$- and right $K$-module structures by the left and right regular actions. Then the restriction map $$\iota: C^{\infty}(\Mat_{2n}^{\R})\rightarrow C^\infty (K)$$ that sends a smooth function $f$ to $f|_{K}$ becomes a $T\cap K\times K$-module homomorphism. To explicitly construct a cohomological vector in the minimal $K$-type $\tau$ of $\pi_K$, we only need to produce a polynomial $F_{\vec{N},\chi_{-l}\otimes\chi_{l+2L}}$ in $C^{\infty}(\Mat_{2n}^{\R})_{\textrm{fin}}=\C[\Mat_{2n}^{\R}]$ such that
     \begin{enumerate}
      \item its restriction $\iota(F_{\vec{N},\chi_{-l}\otimes\chi_{l+2L}})$ lies in a $T\cap K\times K$-submodule of $C^{\infty}(K)_{\textrm{fin}}$ with highest weight $(-\vec{N})\times\Lambda$(see \eqref{eq: highest weight of bi-module});
      \item its restriction $\iota(F_{\vec{N},\chi_{-l}\otimes\chi_{l+2L}})$ satisfies the right $K\cap H$-equivariant property \eqref{eq: right equivariant property simplified}.
         \end{enumerate}
    \begin{rk}
       Though not needed below, we remark that one can in fact prove that the restriction map $\iota$ is surjective by using purely topological argument (for instance by using tubular neighbourhood theorem). This provides a conceptual understanding how our method works.
    \end{rk}
    \subsection{Strategy of the construction}
    As is known to all, every highest weight representation of a connected compact Lie group can be realized as the Cartan component of a tensor product of fundamental representations. Following the principle of this realization, we state our construction as follows.

    {\bf(i)} First, we consider the case that the right $K\cap H$-equivariance \eqref{eq: right equivariant property simplified} is given by the trivial characters $\chi = \omega = id$. In this situation, $l=L=0$. To emphasize this special case, we write $\Lambda_0$ and $\vec N_0$ for the corresponding $\Lambda$ and $\vec N$ in \eqref{eq: def of weight} and \eqref{eq: def of N}, that is,
    \begin{equation*}
    \begin{aligned}
    \Lambda_0=&(N_1,\cdots,N_n,-N_n,\cdots,-N_1)\\
    =&\sum_{j=1}^{n-1}(N_j-N_{j+1})(\underbrace{1,\cdots,1}_{j},0,\cdots,0,\underbrace{-1,\cdots,-1}_{j})+N_n(\underbrace{1,\cdots,1}_{n},\underbrace{-1,\cdots,-1}_{n})
    \end{aligned}
    \end{equation*} and
    \begin{equation*}
    \begin{aligned}
    \vec{N}_0=&(N_1,-N_1,\cdots,N_n,-N_n)\\
    =&\sum_{j=1}^{n-1}(N_j-N_{j+1})(\underbrace{1,-1,\cdots,1,-1}_{j\ \textrm{pairs}},0,\cdots,0)+N_n(1,-1,\cdots,1,-1).
    \end{aligned}
    \end{equation*}
    From above decompositions, we see that the "building blocks" of $\Lambda_0$ and $\vec{N}_0$ are $\Lambda_j=(\underbrace{1,\cdots,1}_{j},0,\cdots,0,\underbrace{-1,\cdots,-1}_{j})$ and $\vec{N}_j=(\underbrace{1,-1,\cdots,1,-1}_{j\ \textrm{pairs}},0,\cdots,0)$. Thus, we only need to construct some right $K\cap H=\textrm{U}_n\times\textrm{U}_n$-invariant polynomials $F_j$($1\le j\le n$) in $\C[\Mat_{2n}^{\R}]$ such that $\iota(F_j)$ lies in the minimal $K$-type $\tau_j$ (with the highest weight $\Lambda_j$) of the right $K$-module $\pi_j$ of $C^\infty (K)$, where
    \begin{equation}\label{eq: temp pi j}\pi_j:=\Ind^K_{T\cap K} \underbrace{\chi_{1}\otimes\chi_{-1}\otimes\cdots\otimes \chi_{1}\otimes\chi_{-1}}_{\text{j pairs}}\otimes\text{id}\otimes\cdots\otimes\text{id}.\end{equation}

    {\bf(ii)} For each integer $k$, denote by $S_{k}$ the symmetric group that permutes $\{1,\ldots,k\}$. In particular, we further identify $S_{2n}$ to be the group of permutation matrices in which every element permutes the $2n$-rows of a $2n\times 2n$ matrix. For $1\leq k\leq n$, $S_{2k}$ is naturally a subgroup of $S_{2n}$, hence can be identified with a group of permutation matrices as well. For $j=1,2,\cdots,n$, define
    $$S^{(j)}:=\{s\in S_{2j}|s(2i-1)\ \textrm{is an odd number}, s(2i)=s(2i-1)+1, 1\le i\le j\} \simeq S_j.$$
    From the definition of $\pi_j$ \eqref{eq: temp pi j}, we see that $\pi_j$ is a left $S^{(j)}$-module, where
    the left $S^{(j)}$-action is given by $$(s\cdot F)(Z) := F(s^{-1}Z).$$ Moreover, we can obtain the highest weight function in the minimal $K$-type $\tau_j$ of $\pi_j$ as in \cite[Charpter III, Prop.6.1]{KV} and show it is $S^{(j)}$-invariant by a direct matrix computation. Thus any function in $\tau_j$ is $S^{(j)}$-invariant, since the highest weight function will generate all vectors in $\tau_j$ by the $K$-action on the right. Hence we require that every $\iota(F_j)$ which will be constructed in $\textrm{\bf (i)}$ is left $S^{(j)}$-invariant.

    {\bf(iii)} Finally, we consider the case that the characters $\chi=\chi_l$ and $\omega$ which define the right $K\cap H$-equivariance in \eqref{eq: right equivariant property simplified} are non-trivial. In this situation, except for constructing above functions $F_j$, we also need to produce some extra functions $\Delta_{1,\pm}$ and $\Delta_{2,\pm}$ in $\C[\Mat_{2n}^{\R}]$ such that
    \begin{enumerate}
      \item they contribute to the right $K\cap H$-equivariant property;
      \item they lie in some suitable irreducible $K$-modules whose highest weights match with the data $\vec{N}$ and $\Lambda$ combining with the highest weights $\Lambda_j$ of $\tau_j$ and $\vec{N}_j$ for $1\le j\le n$.
         \end{enumerate}
    The most technical part is to show that $\iota(F_j)$ belongs to the minimal $K$-type $\tau_j$ of $\pi_j$. We need to check $\iota(F_j)$ is an eigenfunction for the Casimir operator corresponding the correct eigenvalue. This will occupy the following two Subsections.
    \subsection{Polynomial Representations of Unitary Groups}\label{subsection: Polynomial repns}
    It is well known that the irreducible representations of a connected compact group are parameterized by dominant analytically integral weights (see \cite{Kn1}). The set of dominant analytically integral weights for the compact group $K = \mathrm{U}_{2n}$ is given by
    \begin{equation}\label{Eq: highest weight}
       P^{ana}_{++}(\mathrm{U}_{2n}) = \set{\nu= (\nu_1,\nu_2,\cdots,\nu_{2n})}{\nu_1\geq\dots\geq\nu_{2n}, \nu_i\in \mathbb{Z}, i=1,2,\cdots,2n}.
    \end{equation}

    Every matrix $Z\in \Mat_{2n}$ can be written as $Z = X+\boldsymbol{i}Y$, where $X$ and $Y$ are real matrices. We consider a smooth representation of $K$ on the polynomial space $$\begin{aligned}\C[\Mat_{2n}^{\R}] :=&\set{f}{Z = X+\boldsymbol{i}Y, X=(x_{i,j})_{1\le i,j\le 2n}, Y=(y_{i,j})_{1\le i,j\le 2n}\\ &\text{ and $f$ is a polynomial in $x_{ij}, y_{ij}$ with complex coefficients}},\end{aligned}$$
    where $K$ acts by right translation. More precisely, for any matrix $k=A+\boldsymbol{i}B\in K$, $k$ acts via the rule
    \begin{equation}\label{eq: action of K on Q(x,y)}
    k\cdot Q(X,Y):=Q((X,Y)\rho_n(k)),
    \end{equation}
    where $\rho_n$ is the following injective homomorphism
    \begin{equation}\label{eq: injective homorphism}
    \rho_n: \GL_{2n}(\C)\longrightarrow \GL_{4n}(\R),\ \ \ \ \ k=A+\boldsymbol{i}B\mapsto \begin{bmatrix} A & B\\ -B & A\end{bmatrix}.
    \end{equation}
    Every polynomial $f\in \C[\Mat_{2n}^\R]$ is a polynomial in $x_{ij}$ and $y_{ij}$ in prior. Yet for convenience of the future calculations, we will regard $f$ as a polynomial in $z_{ij}$ and $\bar{z}_{ij}$ by a change of variables: $$z_{ij}=x_{ij}+\boldsymbol{i}y_{ij};\quad \bar{z}_{i,j} = x_{ij}-\boldsymbol{i}y_{ij}.$$
    By \eqref{eq: action of K on Q(x,y)}, the action of $K$ on $f(Z, \bar{Z})\in \C[\Mat_{2n}^\R]$ is given by
    \begin{equation}\label{eq: action of K on P}
    (k\cdot f)(Z,\bar{Z})=f(Z\cdot k,\bar{Z}\cdot \bar{k}).
    \end{equation}

    Since the group action \eqref{eq: action of K on P} preserves polynomial degree, as a $K$-module, $\C[\Mat_{2n}^{\R}]$ can be decomposed into $K$-submodules
    \begin{equation*}
        \C[\Mat_{2n}^{\R}] = \bigoplus_{d=0}^{\infty}\C^d[\Mat_{2n}^{\R}],
    \end{equation*}
    where $\C^d[\Mat_{2n}^{\R}]$ is the space of all homogeneous polynomials in $\C[\Mat_{2n}^{\R}]$ with degree $d$. By passing to the level of complexified Lie algebra, we can also regard $\C[\Mat_{2n}^{\R}]$ and $\C^d[\Mat_{2n}^{\R}]$ as a representation of the complexified Lie algebra $\mathfrak{k}^\C = \mathfrak{gl}_{2n}(\C)$.

    For any $1\leq\alpha,\beta\leq 2n$, we write $E_{\alpha\beta}$ for the elementary matrix whose only nonzero entry is located the $\alpha\beta$ entry and is equal to $1$. We choose a basis of $\mathfrak{k}$ as follows,
    \begin{equation}\label{Eq: basis}
       \{\boldsymbol{i}H_1, \ldots, \boldsymbol{i}H_{2n}, E_{\alpha\beta}-E_{\beta\alpha}, \boldsymbol{i}(E_{\alpha\beta}+E_{\beta\alpha})|H_{\gamma}=E_{\gamma\gamma}, 1\le\gamma\le 2n, 1\le\alpha<\beta\le 2n\}.
    \end{equation}
    Then the set $\{H_{\gamma}, E_{\alpha\beta}|1\le\gamma\le 2n, 1\le\alpha\neq\beta\le 2n\}$ forms a basis of $\mathfrak{k}^{\C}$. We fix a Cartan subalgebra $\mathfrak{t}$ of $\mathfrak{k}^{\C}$,
    \begin{equation}
    \mathfrak{t}=\{h=\textrm{diag}(h_1,\cdots,h_{2n})|h_i\in\C, 1\le i\le 2n\}.
    \end{equation}
    Then $\epsilon_{\alpha}-\epsilon_{\beta}$ with $(1\le\alpha\neq\beta\le 2n)$ are all the roots of $\mathfrak{k}^{\C}=\mathfrak{gl}_{2n}(\C)$, where $\epsilon_{\alpha}(h)=h_{\alpha}$. For any $1\leq \alpha\ne\beta\leq 2n$, we choose the root vectors $X_{\epsilon_{\alpha}-\epsilon_{\beta}}=E_{\alpha\beta}$.
    The group action of $K$ in \eqref{eq: action of K on P} induces an action of its Lie algebra $\mathfrak{k}$ on $\C[\Mat_{2n}^{\R}]$ by $$(X.f)(Z,\overline{Z})=\frac{d}{dt}\bigg|_{t=0}f(Z\cdot\exp tX, \overline{Z}\cdot\overline{\exp tX}),\ \ \textrm{for any}\ X\in\mathfrak{k}.$$ By replacing $X$ with the elements in \eqref{Eq: basis}, we get the following differential operators:
     \begin{equation}\label{eq: differential operator of H}
        H_{\gamma} : \sum_{\alpha=1}^{2n} z_{\alpha,\gamma}\frac{\partial}{\partial z_{\alpha,\gamma}}-\bar{z}_{\alpha,\gamma}\frac{\partial}{\partial \bar{z}_{\alpha,\gamma}}, 1\le\gamma\le 2n;
     \end{equation}
     \begin{equation}\label{eq: differential operator of basis}
     \begin{aligned}
        E_{\alpha\beta}-E_{\beta\alpha}: &\sum_{\gamma = 1}^{2n} z_{\gamma,\alpha}\frac{\partial}{\partial z_{\gamma,\beta}}-z_{\gamma,\beta}\frac{\partial}{\partial z_{\gamma,\alpha}}+\bar{z}_{\gamma,\alpha}\frac{\partial}{\partial \bar{z}_{\gamma,\beta}}-\bar{z}_{\gamma,\beta}\frac{\partial}{\partial \bar{z}_{\gamma,\alpha}}, 1\le\alpha<\beta\le 2n;\\
        \boldsymbol{i}(E_{\alpha\beta}+E_{\beta\alpha}): &\boldsymbol{i}\sum_{\gamma = 1}^{2n} z_{\gamma,\alpha}\frac{\partial}{\partial z_{\gamma,\beta}}+z_{\gamma,\beta}\frac{\partial}{\partial z_{\gamma,\alpha}}-\bar{z}_{\gamma,\alpha}\frac{\partial}{\partial \bar{z}_{\gamma,\beta}}-\bar{z}_{\gamma,\beta}\frac{\partial}{\partial \bar{z}_{\gamma,\alpha}}, 1\le\alpha<\beta\le 2n.
     \end{aligned}
     \end{equation}
     Then we have the differential operators
     \begin{equation}\label{eq: differential operator of root vector}
        E_{\alpha\beta} : \sum_{\gamma = 1}^{2n} z_{\gamma,\alpha}\frac{\partial}{\partial z_{\gamma,\beta}}-\bar{z}_{\gamma,\beta}\frac{\partial}{\partial \bar{z}_{\gamma,\alpha}}, 1\le\alpha\neq\beta\le 2n.
     \end{equation}
     By a direct matrix computation, the root vectors $X_{\pm(\epsilon_{\alpha}-\epsilon_{\beta})}$ and $H_{\gamma}$ satisfy the following equation,
    \begin{equation}\label{eq: Lie bracket relation 01}
        X_{\epsilon_\alpha-\epsilon_\beta}X_{-\epsilon_\alpha+\epsilon_\beta}-X_{-\epsilon_\alpha+\epsilon_\beta}X_{\epsilon_\alpha-\epsilon_\beta} = H_{\alpha}-H_{\beta}.
    \end{equation}

    With respect to the trace form $(X,Y):= \mathrm{Tr}(XY)$, the dual basis of $X_{\pm(\epsilon_\alpha-\epsilon_\beta)}, H_{\gamma}$ are given by $X_{\mp(\epsilon_\alpha- \epsilon_\beta)}, H_{\gamma}$. We consider the modified Casimir operator $\Omega$ defined via the trace form:
    \begin{equation}\label{eq: Casimir operator 01}
        \Omega=\sum_{\gamma=1}^{2n} H_{\gamma}^2+\sum_{1\leq \alpha<\beta\leq 2n} X_{\epsilon_\alpha-\epsilon_\beta}X_{-\epsilon_\alpha+\epsilon_\beta}+X_{-\epsilon_\alpha+\epsilon_\beta}X_{\epsilon_\alpha-\epsilon_\beta}.
    \end{equation}
    \begin{rk}
        The standard Casimir operator is defined for semisimple Lie algebra via the Killing form. Yet it is more convenient to work with the trace form here. \cite[Proposition 2.3.3]{Gold} guarantees that $\Omega$ as defined in \eqref{eq: Casimir operator 01} lies inside the center of the universal enveloping algebra of $\mathfrak{gl}_{2n}(\C)$.
    \end{rk}
    The trace form is non-degenerate on $\mathfrak{h}$, hence it induces a non-degenerate bilinear pairing on $\mathfrak{h}^*$, denoted by $<\cdot,\cdot>$.
    Given a highest weight representation $(\tau_\nu,V_\nu)$ of $K$ with highest weight $\nu = (\nu_1,\nu_2,\cdots,\nu_{2n})$. It is well known (for instance see \cite[Lemma 3.3.8]{G-W}) that the Casimir operator $\Omega$ acts on $V_\nu$ by the following scalar
    \begin{equation}\label{eq: casimir eigenvalue}
    \begin{aligned}
        <\nu+\rho_K,\nu+\rho_K>-&<\rho_K, \rho_K>=\\
        \sum_{j=1}^{2n}((\nu+\rho_K)(H_j))^2-\sum_{j=1}^{2n}\rho_K(H_j)^2=&\frac{1}{4}\sum_{j=1}^{2n} (2\nu_j+2n-2j+1)^2-(2n-2j+1)^2,
    \end{aligned}
    \end{equation}
    where $\rho_K$ is the half sum of positive roots of $\mathfrak{k}^\C$, i.e. $\rho_K = (\frac{2n-1}{2},\frac{2n-3}{2},\cdots,-\frac{2n-3}{2},-\frac{2n-1}{2})$. One can also get this scalar by testing the action of the Casimir operator $\Omega$ on the highest weight vector.

    Later in this Section, we need to compute the action of the Casimir operator $\Omega$ on certain polynomials. To simplify our computation in the future, we introduce more notations. In the following, we write $e_{i} (i=1,2,\cdots,n)$ for the standard basis of $\C^{n}$. For each $1\leq j\leq 2n$, we define the following row vectors in $\C^n$:
    \begin{equation}\label{def: uj and vj}
       \begin{aligned}
       u_j :=&(z_{j,1},\cdots,z_{j,n}) = \sum_{i = 1}^n z_{j,i}e_{i};\ \ \ \ &\bar{u}_j :=&(\bar{z}_{j,1},\cdots,\bar{z}_{j,n}) = \sum_{i = 1}^n \bar{z}_{j,i}e_{i};\\
       v_j :=&(z_{j,n+1},\cdots,z_{j,2n}) = \sum_{i = 1}^{n} z_{j,n+i}e_{i};\ \ \ \ &\bar{v}_j :=&(\bar{z}_{j,n+1},\cdots,\bar{z}_{j,2n}) = \sum_{i = 1}^n \bar{z}_{j,n+i}e_{i}.
     \end{aligned}
    \end{equation}
    \begin{lemma}\label{lemma: first order operator on component}
       Let $\delta_{t}^{s}$ be the standard Kronecker delta symbol. Then for all $1\leq i \leq n$, $1\leq j, \gamma \leq 2n$ and $1\le \alpha\neq\beta\le 2n$,
       \begin{equation}\label{Eq: first order operator}
       \begin{aligned}
           H_{\gamma}u_j&=z_{j,\gamma}\sum_{i=1}^n\delta_{\gamma}^{i}e_{i};\ \ \ \ \ \ \ \ &H_{\gamma}\bar{u}_j&=-\bar{z}_{j,\gamma}\sum_{i=1}^n\delta_{\gamma}^{i}e_{i};\\
           H_{\gamma}v_j&=z_{j,\gamma}\sum_{i=1}^n\delta^{n+i}_{\gamma}e_{i};\ \ \ \ \ \ \ \ &H_{\gamma}\bar{v}_j&=-\bar{z}_{j,\gamma}\sum_{i=1}^n\delta^{n+i}_{\gamma}e_{i};\\
           X_{\epsilon_\alpha-\epsilon_\beta}u_j&=z_{j,\alpha}\sum_{i=1}^n\delta^{i}_{\beta}e_{i};\ \ \ \ \ \ \ \ &X_{\epsilon_\alpha-\epsilon_\beta}\bar{u}_j&= -\bar{z}_{j,\beta}\sum_{i=1}^n\delta^{i}_{\alpha}e_{i};\\
            X_{\epsilon_\alpha-\epsilon_\beta}v_j&=z_{j,\alpha}\sum_{i=1}^n\delta^{n+i}_{\beta}e_{i};\ \ \ \ \ \ \ \ &X_{\epsilon_\alpha-\epsilon_\beta}\bar{v}_j&= -\bar{z}_{j,\beta}\sum_{i=1}^n\delta^{n+i}_{\alpha}e_{i}.
       \end{aligned}
       \end{equation}
    \end{lemma}
    \begin{proof} Here we prove the formulas for $u_j$ only, the others can be verified in the same way.
    \begin{equation}
     \begin{aligned}
    H_{\gamma}u_j=&\big(\sum_{\alpha=1}^{2n} z_{\alpha,\gamma}\frac{\partial}{\partial z_{\alpha,\gamma}}-\bar{z}_{\alpha,\gamma}\frac{\partial}{\partial \bar{z}_{\alpha,\gamma}}\big)(\sum_{i=1}^nz_{j,i}e_i)\\
    =&\big(z_{j,\gamma}\frac{\partial}{\partial z_{j,\gamma}}-\bar{z}_{j,\gamma}\frac{\partial}{\partial \bar{z}_{j,\gamma}}\big)(\sum_{i=1}^nz_{j,i}e_i)\\
    =&z_{j,\gamma}\sum_{i=1}^n\delta^i_{\gamma}e_i=\left\{\begin{aligned}
     &z_{j,\gamma}e_{\gamma}, \ \ &1\le\gamma\le n;\\
     &0,\ \ &\textrm{otherwise}.
    \end{aligned}\right.
      \end{aligned}
    \end{equation}
    \begin{equation}
     \begin{aligned}
    X_{\epsilon_\alpha-\epsilon_\beta}u_j=E_{\alpha\beta}u_j=&\big(\sum_{\gamma=1}^{2n} z_{\gamma,\alpha}\frac{\partial}{\partial z_{\gamma,\beta}}-\bar{z}_{\gamma,\beta}\frac{\partial}{\partial \bar{z}_{\gamma,\alpha}}\big)(\sum_{i=1}^nz_{j,i}e_i)\\
    =&\big(z_{j,\alpha}\frac{\partial}{\partial z_{j,\beta}}-\bar{z}_{j,\beta}\frac{\partial}{\partial \bar{z}_{j,\alpha}}\big)(\sum_{i=1}^nz_{j,i}e_i)\\
    =&z_{j,\alpha}\sum_{i=1}^n\delta^i_{\beta}e_i=\left\{\begin{aligned}
     &z_{j,\alpha}e_{\beta}, \ \ &1\le\beta\le n;\\
     &0,\ \ &\textrm{otherwise}.
    \end{aligned}\right.
      \end{aligned}
    \end{equation}
    \end{proof}
    Similarly, applying above operators once again, we can show the following equations.
    \begin{lemma}\label{Lemma: 2nd order operators}For all $1\leq i \leq n$, $1\leq j, \gamma \leq 2n$ and $1\le \alpha\neq\beta\le 2n$, we have
      \begin{equation}\label{Eq: 2nd order operator}
      \begin{aligned}
      H^2_{\gamma}u_j&=z_{j,\gamma}\sum_{i=1}^n\delta_{\gamma}^{i}e_{i};\ \ \ \ \ \ \ \ &H^2_{\gamma}\bar{u}_j&=\bar{z}_{j,\gamma}\sum_{i=1}^n\delta_{\gamma}^{i}e_{i};\\
      H^2_{\gamma}v_j&=z_{j,\gamma}\sum_{i=1}^n\delta^{n+i}_{\gamma}e_{i};\ \ \ \ \ \ \ \ &H^2_{\gamma}\bar{v}_j&=\bar{z}_{j,\gamma}\sum_{i=1}^n\delta^{n+i}_{\gamma}e_{i};\\
      X_{\epsilon_\alpha-\epsilon_\beta}X_{\epsilon_\beta-\epsilon_\alpha}u_j&=z_{j,\alpha}\sum_{i=1}^n\delta^{i}_{\alpha}e_{i};\ \ \ \ \ \ \ \ &X_{\epsilon_\alpha-\epsilon_\beta}X_{\epsilon_\beta-\epsilon_\alpha}\bar{u}_j&= \bar{z}_{j,\beta}\sum_{i=1}^n\delta^{i}_{\beta}e_{i};\\
      X_{\epsilon_\alpha-\epsilon_\beta}X_{\epsilon_\beta-\epsilon_\alpha}v_j&=z_{j,\alpha}\sum_{i=1}^n\delta^{n+i}_{\alpha}e_{i};\ \ \ \ \ \ \ \ &X_{\epsilon_\alpha-\epsilon_\beta}X_{\epsilon_\beta-\epsilon_\alpha}\bar{v}_j&= \bar{z}_{j,\beta}\sum_{i=1}^n\delta^{n+i}_{\beta}e_{i}.
      \end{aligned}
      \end{equation}
    \end{lemma}
    \begin{proof}We still show the formulas for $u_j$ only, the others can be proved by similar computation. By applying Lemma \ref{lemma: first order operator on component}, we have
     \begin{equation}
     \begin{aligned}
    H^2_{\gamma}u_j=H_{\gamma}(z_{j,\gamma}\sum_{i=1}^n\delta^i_{\gamma}e_i)=&\big(\sum_{\alpha=1}^{2n} z_{\alpha,\gamma}\frac{\partial}{\partial z_{\alpha,\gamma}}-\bar{z}_{\alpha,\gamma}\frac{\partial}{\partial \bar{z}_{\alpha,\gamma}}\big)(z_{j,\gamma}\sum_{i=1}^n\delta^i_{\gamma}e_i)\\
    =&\big(z_{j,\gamma}\frac{\partial}{\partial z_{j,\gamma}}-\bar{z}_{j,\gamma}\frac{\partial}{\partial \bar{z}_{j,\gamma}}\big)(z_{j,\gamma}\sum_{i=1}^n\delta^i_{\gamma}e_i)=z_{j,\gamma}\sum_{i=1}^n\delta^i_{\gamma}e_i.
      \end{aligned}
    \end{equation}
    \begin{equation}
     \begin{aligned}
    X_{\epsilon_\alpha-\epsilon_\beta}X_{\epsilon_\beta-\epsilon_\alpha}u_j=&X_{\epsilon_\alpha-\epsilon_\beta}(z_{j,\beta}\sum_{i=1}^n\delta^i_{\alpha}e_i)\\
    =&\big(\sum_{\gamma=1}^{2n} z_{\gamma,\alpha}\frac{\partial}{\partial z_{\gamma,\beta}}-\bar{z}_{\gamma,\beta}\frac{\partial}{\partial \bar{z}_{\gamma,\alpha}}\big)(z_{j,\beta}\sum_{i=1}^n\delta^i_{\alpha}e_i)\\
    =&\big(z_{j,\alpha}\frac{\partial}{\partial z_{j,\beta}}-\bar{z}_{j,\beta}\frac{\partial}{\partial \bar{z}_{j,\alpha}}\big)(z_{j,\beta}\sum_{i=1}^n\delta^i_{\alpha}e_i)=z_{j,\alpha}\sum_{i=1}^n\delta^i_{\alpha}e_i.
      \end{aligned}
    \end{equation}
    \end{proof}
    Let $\langle\cdot,\cdot\rangle_{c}$ be the standard hermitian form on $\C^n$, i.e. $\langle u_j, u_l\rangle_{c}=u_j\cdot \bar{u}_l^T=\sum_{i=1}^nz_{j,i}\bar{z}_{l,i}$, for any $u_j,u_l\in\C^n$. Then $\langle u_j, u_l\rangle_{c}\in\C[\Mat_{2n}^\R]$. Applying above differential operators to $\langle u_j, u_l\rangle_{c}$, we get
    \begin{lemma}\label{Lemma: operators act on pairing}For all $1\leq i \leq n$, $1\leq j, l, \gamma\leq 2n$ and $1\le \alpha\neq\beta\le 2n$, we have
    \begin{equation}\label{Eq: operators act on pairing}
      \begin{aligned}
           H_{\gamma}\langle u_j, u_l\rangle_{c}&=H^2_{\gamma}\langle u_j, u_l\rangle_{c}=0;\\
           X_{\epsilon_\alpha-\epsilon_\beta}\langle u_j, u_l\rangle_{c}&=z_{j,\alpha}\sum_{i=1}^n\bar{z}_{l,i}\delta^i_{\beta}-\bar{z}_{l,\beta}\sum_{i=1}^nz_{j,i}\delta^i_{\alpha};\\
           X_{\epsilon_\alpha-\epsilon_\beta}X_{\epsilon_\beta-\epsilon_\alpha}\langle u_j, u_l\rangle_{c}&=\left\{\begin{aligned}
              &z_{j,\alpha}\bar{z}_{l,\alpha}-z_{j,\beta}\bar{z}_{l,\beta},\ \ &\textrm{ if }1\le\alpha\le n<\beta\le 2n;\\
              &z_{j,\beta}\bar{z}_{l,\beta}-z_{j,\alpha}\bar{z}_{l,\alpha}, \ \ &\textrm{ if }1\le\beta\le n<\alpha\le 2n;\\
              &0,\ \ &\textrm{ if }1\le\alpha\neq\beta\le n, \textrm{or}\ n<\alpha\neq\beta\le 2n.
              \end{aligned}\right.
       \end{aligned}
    \end{equation}
    \end{lemma}
    \begin{proof}
    By Lemma \ref{lemma: first order operator on component} and Lemma \ref{Lemma: 2nd order operators},
    \begin{equation*}
       \begin{aligned}
           H_{\gamma}\langle u_j, u_l\rangle_{c}=&(H_{\gamma}u_j)\cdot\bar{u}_l^T+u_j\cdot(H_{\gamma}\bar{u}_l)^T=z_{j,\gamma}\sum_{i=1}^n\delta^{i}_{\gamma}\bar{z}_{l,i}-\bar{z}_{l,\gamma}\sum_{i=1}^n\delta^i_{\gamma}z_{j,i}\\
           =&z_{j,\gamma}\bar{z}_{l,\gamma}-\bar{z}_{l,\gamma}z_{j,\gamma}=0;\\
           H^2_{\gamma}\langle u_j, u_l\rangle_{c}=&(H^2_{\gamma}u_j)\cdot\bar{u}_l^T+2(H_{\gamma}u_j)\cdot(H_{\gamma}\bar{u}_l)^T+u_j\cdot(H^2_{\gamma}\bar{u}_l)^T\\
           =&z_{j,\gamma}\sum_{i=1}^n\delta^{i}_{\gamma}\bar{z}_{l,i}-2z_{j,\gamma}\bar{z}_{l,\gamma}\sum_{i=1}^n\delta^{i}_{\gamma}+\bar{z}_{l,\gamma}\sum_{i=1}^n\delta^i_{\gamma}z_{j,i}=0;\\
           X_{\epsilon_\alpha-\epsilon_\beta}\langle u_j, u_l\rangle_{c}=&(X_{\epsilon_\alpha-\epsilon_\beta}u_j)\cdot \bar{u}_l^T+u_j\cdot(X_{\epsilon_\alpha-\epsilon_\beta}\bar{u}_l)^T=z_{j,\alpha}\sum_{i=1}^n\bar{z}_{l,i}\delta^i_{\beta}-\bar{z}_{l,\beta}\sum_{i=1}^nz_{j,i}\delta^i_{\alpha};\\
           =&\left\{\begin{aligned}
              -&z_{j,\alpha}\bar{z}_{l,\beta},\ \ &\textrm{ if }1\le\alpha\le n<\beta\le 2n;\\
              &z_{j,\alpha}\bar{z}_{l,\beta}, \ \ &\textrm{ if }1\le\beta\le n<\alpha\le 2n;\\
              &0,\ \ &\textrm{ if }1\le\alpha\neq\beta\le n, \textrm{or}\ n<\alpha\neq\beta\le 2n.
              \end{aligned}\right.
        \end{aligned}
       \end{equation*}
       \begin{equation*}
       \begin{aligned}
           &X_{\epsilon_\alpha-\epsilon_\beta}X_{\epsilon_\beta-\epsilon_\alpha}\langle u_j, u_l\rangle_{c}\\
           =&(X_{\epsilon_\alpha-\epsilon_\beta}X_{\epsilon_\beta-\epsilon_\alpha}u_j)\cdot\bar{u}_l^T+u_j\cdot(X_{\epsilon_\alpha-\epsilon_\beta}X_{\epsilon_\beta-\epsilon_\alpha}\bar{u}_l)^T\\
           &+(X_{\epsilon_\alpha-\epsilon_\beta}u_j)\cdot(X_{\epsilon_\beta-\epsilon_\alpha}\bar{u}_l)^T+(X_{\epsilon_\beta-\epsilon_\alpha}u_j)\cdot(X_{\epsilon_\alpha-\epsilon_\beta}\bar{u}_l)^T\\
           =&z_{j,\alpha}\sum_{i=1}^n\delta^{i}_{\alpha}\bar{z}_{l,i}+\bar{z}_{l,\beta}\sum_{i=1}^n\delta^{i}_{\beta}z_{j,i}-z_{j,\alpha}\bar{z}_{l,\alpha}\sum_{i=1}^n\delta^{i}_{\beta}-z_{j,\beta}\bar{z}_{l,\beta}\sum_{i=1}^n\delta^{i}_{\alpha}\\
           =&\left\{\begin{aligned}
              &z_{j,\alpha}\bar{z}_{l,\alpha}-z_{j,\beta}\bar{z}_{l,\beta},\ \ &\textrm{ if }1\le\alpha\le n<\beta\le 2n;\\
              &z_{j,\beta}\bar{z}_{l,\beta}-z_{j,\alpha}\bar{z}_{l,\alpha}, \ \ &\textrm{ if }1\le\beta\le n<\alpha\le 2n;\\
              &0,\ \ &\textrm{ if }1\le\alpha\neq\beta\le n, \textrm{or}\ n<\alpha\neq\beta\le 2n.
              \end{aligned}\right.
       \end{aligned}
       \end{equation*}
    \end{proof}

    \subsection{Some Eigen-Polynomials For Casimir Operator}\label{subsection: Casimir operator}
    We consider the restriction map $\iota: \C[\Mat_{2n}^{\R}]\rightarrow C^\infty (K)$ that sends a polynomial function $f$ to $f|_{K}$. More precisely, if we take a polynomial $f(Z,\bar{Z})\in \C[\Mat_{2n}^{\R}]$, then the value of $\iota(f)$ at $k\in K$ is $f(k,\bar{k})$. The restriction map $\iota$ commutes with the right translation, hence $\iota$ is a $K$-module homomorphism from $\C[\Mat_{2n}^{\R}]$ to $C^\infty (K)$. Let $\mathcal{I}$ be the kernel of the restriction map $\iota$. Then $\mathcal{I}$ is a $K$-submodule automatically. We define:
    \begin{defn}\label{def: mod defining ideal}
         For two polynomials $P_1(Z,\bar{Z}), P_2(Z,\bar{Z}) \in \C[\Mat_{2n}^{\R}]$, we write $$P_1 \equiv P_2 \mod \mathcal{I},$$ if $\iota(P_1) = \iota(P_2)$.
    \end{defn}
    \begin{lemma}\label{lemma: moving variables to the left}
       For all $1\le j\neq l\le 2n$,
       \begin{equation*}
           \langle v_j , v_l \rangle_{c} \equiv -\langle u_j, u_l \rangle_{c} \mod \mathcal{I}.
       \end{equation*}
    \end{lemma}
    \begin{proof}
       By the definition \ref{def: mod defining ideal}, to show the above congruence expressions, we only need to prove that the polynomial $P(Z,\overline{Z})=\langle u_j, u_l\rangle_{c}+\langle v_j, v_l \rangle_{c}\ (j\neq l)$ is identical to zero when restricted to $K=\RU_{2n}$. Here $(u_j,v_j)$(resp.\ $(\bar{u}_j,\bar{v}_j)$) defined in \eqref{def: uj and vj} gives the $j$-th row of the matrix $Z$(resp.\ $\overline{Z}$).
       Since for all $Z\in K$, $Z\overline{Z}^{t}=\RI_{2n}$, we have $\iota\left((u_j,v_j)\cdot(\bar{u}_l,\bar{v}_l)^t\right)=0$ for all $j\neq l$. In other words, $\iota\left(\langle u_j, u_l\rangle_{c}+\langle v_j, v_l \rangle_{c}\right)=0$. This finishes the proof of the Lemma.
       \end{proof}

       Now we will give the construction of some Casimir eigen-polynomials. Since we only care about the restriction of a polynomial $f(Z,\bar{Z})\in \C[\Mat_{2n}^\R]$ on $K$, our eigen-polynomial is in the sense of $\mod\mathcal{I}$. To simplify notations, we set for all $1\leq j\neq l\leq 2n$,
       \begin{equation}\label{eq: def of Phi jl}
           \Phi_{jl} := \langle u_j, u_l \rangle_{c} = u_j\cdot\bar{u}_l^T.
       \end{equation}
       Then $\Phi_{jl}\in \C[\Mat_{2n}^\R]$.
       \begin{prop}\label{prop: degree 2 casimir eigen polynomial}
           Recall the Casimir operator $\Omega$ defined in \eqref{eq: Casimir operator 01}. For all $1\leq j\neq l\leq 2n$,
           \begin{equation}\label{eq: degree 2 casimir eigen polynomial}
              \Omega\Phi_{jl} \equiv 4n\Phi_{jl} \mod \mathcal{I}.
           \end{equation}
       \end{prop}
       \begin{proof}
           By the Lie bracket relation \eqref{eq: Lie bracket relation 01} and Lemma \ref{Lemma: operators act on pairing}, we can simplify
           \begin{equation*}\label{eq: 032}
               \begin{aligned}
               \Omega\Phi_{jl}=&2\sum_{1\leq \alpha<\beta\leq 2n}X_{\epsilon_\beta-\epsilon_\alpha}X_{\epsilon_\alpha-\epsilon_\beta}\Phi_{jl}\\
               =&2\sum_{1\leq \alpha\le n<\beta\leq 2n}z_{j,\alpha}\bar{z}_{l,\alpha}-z_{j,\beta}\bar{z}_{l,\beta}\\
               =&2n\big(\sum^n_{\alpha=1}z_{j,\alpha}\bar{z}_{l,\alpha}-\sum^{2n}_{\beta=n+1}z_{j,\beta}\bar{z}_{l,\beta}\big).
               \end{aligned}
           \end{equation*}
           Applying Lemma \ref{lemma: moving variables to the left}, we get
           \begin{equation*}
           \Omega\Phi_{jl}=2n(\langle u_j, u_l \rangle_{c}-\langle v_j, v_l \rangle_{c})\equiv 4n\Phi_{jl} \mod \mathcal{I}.
           \end{equation*}
       \end{proof}

       \begin{thm}\label{thm: eigen polynomial for even fundamental repns}
           For each integer $1\leq k\leq n$, we define a polynomial $F_{k}(Z,\bar{Z})\in\C[\Mat_{2n}^\R]$ as follows:
           \begin{equation}\label{eq: eigen polynomial of fundamental repns}
                F_{k} := \sum_{s\in S_k} \text{sgn}(s)\cdot\Phi_{1,s(1)+1}\Phi_{3,s(3)+1}\cdots\Phi_{2k-1,s(2k-1)+1},\\
           \end{equation}
           where $S_k$ is the permutation group of the set $\{1, 3, \cdots, 2k-1\}$ and $\text{sgn}(s)$ is the sign of the permutation $s$, i.e. it is $1$ if $s$ is an even permutation, and it is $-1$ if $s$ is an odd permutation.
           Then
           \begin{equation}\label{eq: eigenvalue polynomial of fundamental repns}
              \Omega F_{k} \equiv (4nk-2k(k-1))F_{k} \mod\mathcal{I}.
           \end{equation}
       \end{thm}
       \begin{proof}
           By the definition of Casimir operator \eqref{eq: Casimir operator 01}, we find that
           \begin{equation}\label{eq: 025}
             \begin{aligned}
                \Omega F_{k} =&\sum_{s\in S_k}\text{sgn}(s)\sum_{l=1}^k\big(\prod_{j\neq l}\Phi_{j,s(j)+1}\big)\cdot\Omega\Phi_{l,s(l)+1} \\
                &+\sum_{\gamma=1}^{2n}\sum_{s\in S_k}\text{sgn}(s)\sum_{j\neq l}\big(\prod_{i\neq (j,l)}\Phi_{i,s(i)+1}\big)\cdot H_{\gamma}\Phi_{j,s(j)+1}\cdot H_{\gamma}\Phi_{l,s(l)+1}\\
                &+\sum_{1\le\alpha<\beta\le 2n}\sum_{s\in S_k}\text{sgn}(s)\sum_{j\neq l}\big(\prod_{i\neq (j,l)}\Phi_{i,s(i)+1}\big)\cdot\\
                &\big(X_{\epsilon_\alpha-\epsilon_\beta}\Phi_{j,s(j)+1}\cdot X_{\epsilon_\beta-\epsilon_\alpha}\Phi_{l,s(l)+1}
                                 +X_{\epsilon_\beta-\epsilon_\alpha}\Phi_{j,s(j)+1}\cdot X_{\epsilon_\alpha-\epsilon_\beta}\Phi_{l,s(l)+1}\big).
             \end{aligned}
           \end{equation}
           By Proposition \ref{prop: degree 2 casimir eigen polynomial}, the first summand on the RHS of \eqref{eq: 025} satisfies
           \begin{equation}\label{eq: 039}
              \sum_{s\in S_k}\text{sgn}(s)\sum_{l=1}^k\big(\prod_{j\ne l} \Phi_{j,s(j)+1}\big)\cdot\Omega\Phi_{l,s(l)+1} \equiv 4nk F_{k} \mod\mathcal{I}.
           \end{equation}
           By Lemma \ref{Lemma: operators act on pairing}, the second summand on the RHS of \eqref{eq: 025} is zero, and the third summand is equal to
           \begin{equation}\label{eq: 028}
             \begin{aligned}
               &\sum_{1\le\alpha\le n<\beta\le 2n}\sum_{s\in S_k}\text{sgn}(s)\sum_{j\neq l}\big(\prod_{i\neq (j,l)}\Phi_{i,s(i)+1}\big)(-z_{j,\alpha}\bar{z_{s(j)+1,\beta}}\cdot z_{l,\beta}\bar{z_{s(l)+1,\alpha}}-z_{j,\beta}\bar{z_{s(j)+1,\alpha}}\cdot z_{l,\alpha}\bar{z_{s(l)+1,\beta}})\\
               =&\sum_{s\in S_k}\text{sgn}(s)\sum_{j\neq l}\big(\prod_{i\neq (j,l)}\Phi_{i,s(i)+1}\big)(-\langle u_j, u_{s(l)+1}\rangle_{c}\cdot\langle v_l, v_{s(j)+1}\rangle_{c}-\langle v_j, v_{s(l)+1}\rangle_{c}\cdot\langle u_l, u_{s(j)+1}\rangle_{c}).
             \end{aligned}
           \end{equation}
           By Lemma \ref{lemma: moving variables to the left},
           \begin{equation}\label{eq: congruence identity}
           \begin{aligned}
           \langle v_l, v_{s(j)+1}\rangle_{c} \equiv -\langle u_l, u_{s(j)+1}\rangle_{c} \mod\mathcal{I},\\
           \langle v_j, v_{s(l)+1}\rangle_{c} \equiv -\langle u_j, u_{s(l)+1}\rangle_{c} \mod\mathcal{I}.
           \end{aligned}
           \end{equation}
           Substituting \eqref{eq: congruence identity} into \eqref{eq: 028}, the third summand on the RHS of \eqref{eq: 025} contributes
           \begin{equation}\label{eq: 029}
                2\sum_{s\in S_k}\text{sgn}(s)\sum_{j\neq l}\big(\prod_{i\neq (j,l)}\Phi_{i,s(i)+1}\big)\Phi_{j,s(l)+1}\Phi_{l,s(j)+1} \mod \mathcal{I}.
           \end{equation}
           Let $s_{jl}$ be the transposition in $S_k$ interchanging $j$ and $l$. By changing the variable $s\mapsto s\cdot s_{jl}$,
           \begin{equation}\label{eq: 031}
               \begin{aligned}
               &2\sum_{s\in S_k}\text{sgn}(s)\sum_{j\neq l}\big(\prod_{i\neq (j,l)}\Phi_{i,s(i)+1}\big)\Phi_{j,s(l)+1}\Phi_{l,s(j)+1}\\
               =&2\sum_{s\in S_k}\sum_{j\neq l}\text{sgn}(s\cdot s_{jl})\big(\prod_{i\neq (j,l)}\Phi_{i,s(i)+1}\big)\Phi_{j,s(j)+1}\Phi_{l,s(l)+1}\\
               =&-2k(k-1)F_{k}.
               \end{aligned}
           \end{equation}
           Combining \eqref{eq: 039} and \eqref{eq: 031}, we get \eqref{eq: eigenvalue polynomial of fundamental repns}.
       \end{proof}
       \begin{rk}
           Since the unitary group $\mathrm{U}_n$ preserves the Hermitian pairing in $\C^n$. We can immediately obtain that all $\Phi_{jl}$, and hence all $F_k$, are right $K\cap H = \mathrm{U}_n\times \mathrm{U}_n$ invariant.
       \end{rk}
       \begin{cor}\label{cor: even fundamental polynomial live in even fundamental repn}
         For each $j=1,2,\cdots, n$, let $F_{j}$ be the polynomial constructed in Theorem \ref{thm: eigen polynomial for even fundamental repns}. We regard
         \begin{equation*}
              \pi_j := \Ind_{T\cap K}^K \underbrace{\chi_{1}\otimes\chi_{-1}\otimes\cdots\otimes \chi_{1}\otimes\chi_{-1}}_{\text{j pairs}}\otimes\text{id}\otimes\cdots\otimes\text{id}
         \end{equation*}
         as a $K$-submodule of $C^\infty(K)$. Set $\Lambda_j := (\underbrace{1,\cdots,1}_{j},0,\cdots,0,\underbrace{-1,\cdots,-1}_{j})$. Then $\iota(F_{j})\in\pi_j$ and it generates an irreducible submodule $(\tau_{j}, V_j)$ of $K$ with highest weight $\Lambda_j$.
    \end{cor}
    \begin{proof}
         From a direct matrix computation, we see that
       \begin{equation}\label{eq: 083}
           \begin{aligned}
           &\Phi_{jl}(\diag(e^{\boldsymbol{i}\theta_1},\cdots,e^{\boldsymbol{i}\theta_{2n}})Z,\diag(e^{-\boldsymbol{i}\theta_1},\cdots,e^{-\boldsymbol{i}\theta_{2n}})\bar{Z}) = &e^{\boldsymbol{i}\theta_j}u_j\cdot(e^{-\boldsymbol{i}\theta_l}\bar{u}_l)^T = e^{\boldsymbol{i}(\theta_j-\theta_l)}\Phi_{jl}(Z,\bar{Z}).
           \end{aligned}
       \end{equation}
       Thus, from \eqref{eq: eigen polynomial of fundamental repns} and \eqref{eq: 083}, we can obtain that for $1\le k\le n$,
       \begin{equation}\label{eq: 084}
              \begin{aligned}
                &F_{k}(\diag(e^{\boldsymbol{i}\theta_1},\ldots,e^{\boldsymbol{i}\theta_{2n}})Z,\diag(e^{-\boldsymbol{i}\theta_1},\ldots,e^{-\boldsymbol{i}\theta_{2n}})\bar{Z}) \\
           =&\sum_{s\in S_k} \text{sgn}(s)\cdot\prod_{i=1}^k\Phi_{2i-1,s(2i-1)+1}(\diag(e^{\boldsymbol{i}\theta_1},\cdots,e^{\boldsymbol{i}\theta_{2n}})Z,\diag(e^{-\boldsymbol{i}\theta_1},\cdots,e^{-\boldsymbol{i}\theta_{2n}})\bar{Z})\\
           =&\prod_{i=1}^ke^{\boldsymbol{i}(\theta_{2i-1}-\theta_{2i})}\sum_{s\in S_k} \text{sgn}(s)\cdot\Phi_{1,s(1)+1}(Z,\bar{Z})\cdots\Phi_{2k-1,s(2k-1)+1}(Z,\bar{Z})\\
           =&\prod_{i=1}^ke^{\boldsymbol{i}(\theta_{2i-1}-\theta_{2i})}F_{k}(Z,\bar{Z}).
           \end{aligned}
       \end{equation}
         By \eqref{eq: 084}, $\iota(F_{j})$ belongs to the induced space $\pi_{j}$, hence
         \begin{equation*}
             \iota(F_{j})\in \iota(\C^{2j}[\Mat_{2n}^{\R}])\cap \pi_{j}.
         \end{equation*}
         By \eqref{eq: casimir eigenvalue}, the Casimir operator $\Omega$ acts on the highest weight representation $\tau_j$ via the scalar
          \begin{equation*}
              C_j := <\Lambda_j+\rho_K, \Lambda_j+\rho_K> -<\rho_K,\rho_K>=4nj-2j(j-1).
          \end{equation*}
         By Theorem \ref{thm: eigen polynomial for even fundamental repns},
         \begin{equation}
             \Omega\cdot \iota(F_{j}) = C_j \cdot\iota(F_{j}).
         \end{equation}
         Thus the set of irreducible $K$-submodules of $\iota(C^{2j}[\Mat_{2n}^{\R}])\cap \pi_{j}$ on which $\Omega$ acts by the scalar $C_j$ is non-empty.

         Suppose $\tau_\nu$ is a $K$-submodule with the highest weight $\nu$ living in the intersection $$\iota(\C^{2j}[\Mat_{2n}^{\R}])\cap \pi_{j}.$$
         Applying the Frobenius Reciprocity Law, we get
         \begin{equation}\label{eq: 090}
                \begin{aligned}
               &\Hom_{T\cap K}(\tau_{\nu}|_{T\cap K},\chi_{1}\otimes\cdots\otimes\chi_{1}\otimes\text{id}\otimes\cdots\otimes\text{id}\otimes \chi_{-1}\otimes\cdots\otimes\chi_{-1})\\
               =&\Hom_{K}(\tau_{\nu},\Ind_{T\cap K}^K\chi_{1}\otimes\cdots\otimes\chi_{1}\otimes\text{id}\otimes\cdots\otimes\text{id}\otimes \chi_{-1}\otimes\cdots\otimes\chi_{-1})\\
               \cong&\Hom_K(\tau_\nu, \pi_{j})\ne 0.
               \end{aligned}
         \end{equation}
         This implies that $\nu$ is higher than $\Lambda_j$, i.e. $\nu=\Lambda_j+\delta$ where $\delta$ is a non-negative linear combination of some positive roots. Then the Casimir operator acts on $\tau_\nu$ by the scalar $<\nu+\rho_K, \nu+\rho_K>-<\rho_K,\rho_K>$, which satisfies the following estimate:
         \begin{equation*}
             \begin{aligned}
             &<\nu+\rho_K, \nu+\rho_K>-<\rho_K,\rho_K>\\
             =&<\Lambda_j+\delta+\rho_K, \Lambda_j+\delta+\rho_K>-<\rho_K,\rho_K>\\
             =&<\Lambda_j+\rho_K,\Lambda_j+\rho_K>+2<\delta,\Lambda_j+\rho_K> +<\delta,\delta>-<\rho_K,\rho_K>\\
             \geq&<\Lambda_j+\rho_K,\Lambda_j+\rho_K>-<\rho_K,\rho_K>=C_j.
             \end{aligned}
         \end{equation*}
         The above equality holds only when $\delta=0$, i.e. $\nu=\Lambda_j$. Thus, the only irreducible $K$-submodule of $\iota(\C^{2j}[\Mat_{2n}^{\R}])\cap \pi_{j}$ with the Casimir eigenvalue $C_j$ must have highest weight $\Lambda_j$. This completes the proof of the Corollary.
       \end{proof}
       Now we will give the construction of $\Delta_{1,\pm}$ and $\Delta_{2,\pm}$, which will contribute to the right $K\cap H$-equivariance in the cohomological vector.
       \begin{prop}\label{prop: det lives in fundamental repn}
           Let us consider the following polynomials in $\C[\Mat_{2n}^\R]$:
           \begin{equation}\label{eq: det polynomial}
             \begin{aligned}
              \Delta_{1,+}(Z,\bar{Z}):&=\det\left(
                                        \begin{array}{c}
                                          u_1 \\
                                          u_3 \\
                                          \cdots \\
                                          u_{2n-1} \\
                                        \end{array}
                                      \right),\quad\quad
               \Delta_{1,-}(Z,\bar{Z}):=\det\left(
                                        \begin{array}{c}
                                          \bar{u}_2 \\
                                          \bar{u}_4 \\
                                          \cdots \\
                                          \bar{u}_{2n}, \\
                                        \end{array}
                                      \right),\\
                \Delta_{2,+}(Z,\bar{Z}):&=\det(Z),\quad\quad\quad\quad\quad\quad
               \Delta_{2,-}(Z,\bar{Z}):=\det(\bar{Z}),
               \end{aligned}
           \end{equation}
           where $u_j, \bar{u}_j$ are defined in \eqref{def: uj and vj}. Then $\Delta_{1,+}($resp. $\Delta_{1,-})$ generates an irreducible $K$-submodule $W_1($resp. $W_{-1})$ of $\C^n[\Mat_{2n}^{\R}]$ with highest weight $(\underbrace{1,\cdots,1}_n, 0,\cdots, 0)($resp. highest weight $(0,\cdots, 0, \underbrace{-1,\cdots,-1}_n))$. $\Delta_{2,+}($resp. $\Delta_{2,-})$ generates an irreducible $K$-submodule  $W'_1($resp. $W'_{-1})$ of $\C^{2n}[\Mat_{2n}^{\R}]$ with highest weight $(1,\ldots,1)($resp. $(-1,\ldots,-1))$.
       \end{prop}
       \begin{proof}
           Recall that $K$ acts on $\C[\Mat_{2n}^{\R}]$ by right multiplication, then it acts on each row vector $(u_j,v_j)$ of the matrix $Z$ by right multiplication. By restriction on the identity component, this right multiplication yields a standard representation of $K^0=\SU_{2n}$ on $\C^{2n}$.  Hence, $\Delta_{1,+}$ lives in the representation space $\bigwedge^n(\C^{2n})^{\ast}$. By passing to the complexified Lie algebra, $\bigwedge^n\C^{2n}$ is the $n$-th fundamental representation of $\mathfrak{sl}_{2n}(\C)$, which is irreducible, self dual and has highest weight $\omega_n=\sum_{i=1}^n\epsilon_i-\frac{1}{2}\sum_{i=1}^{2n}\epsilon_i$.

           To compute the highest weight $\mu$ of $\bigwedge^n\C^{2n}$ as a representation of $\mathfrak{k}^\C = \mathfrak{gl}_{2n}(\C)$, we only need to check the central character (see \cite[Theorem 5.5.22]{G-W}). Write $$\mu=\sum_{i=1}^{2n}m_i\epsilon_i=\sum_{i=1}^{2n-1}(m_i-m_{i+1})\lambda_i+m_{2n}\lambda_{2n},$$
           where $\lambda_i=\sum_{k=1}^i\epsilon_k\in\mathfrak{h}^{\ast}$ and $m_i\in\mathbb{Z}$. Since the restriction of $\mu$ on the Cartan subalgebra of $\mathfrak{sl}_{2n}(\C)$ must be equal to $\omega_n$, we get $m_1=\cdots=m_n=m_{n+1}+1$ and $m_{n+1}=\cdots=m_{2n}$. By Lemma \ref{lemma: first order operator on component}, we get
           \begin{equation*}
           \big(\sum_{\gamma=1}^{2n}H_{\gamma}\big)\Delta_{1,+}=n\Delta_{1,+}.
           \end{equation*}
           This implies that $m_1=\cdots=m_n=1$ and $m_{n+1}=\cdots=m_{2n}=0$. Therefore, $\Delta_{1,+}$ generates an irreducible $K$-submodule $W_1$ with the highest weight $\mu=(1,\cdots,1, 0,\cdots, 0)$ (first $n$ coordinates are all $1$).

           Since the conjugate representation $\bar{\C^{2n}}$ of $K $ can be identified with the $(\C^{2n})^{\ast}$ by the $K$-invariant Hermitian inner product on $\C^{2n}$. Then $\Delta_{1,-}$ lives in the representation space $\bigwedge^n(\bar{\C^{2n}})^{\ast}\simeq\bigwedge^n\C^{2n}$. As we have shown in the previous paragraph that the highest weight of the $K$-module $\bigwedge^n(\C^{2n})^{\ast}$ is $(1,\ldots,1, 0,\cdots, 0)$, then $\Delta_{1,-}$ generates an irreducible $K$-submodule $W_{-1}$ of $\C^n[\Mat_{2n}^{\R}]$ with the highest weight $(0,\ldots,0, -1,\cdots, -1)$(last $n$ coordinates are all $-1$).

           Similarly, we can show that $\Delta_{2,+}$ lives in the representation space $\bigwedge^{2n}(\C^{2n})^{\ast}$ with highest weight $(1,\cdots,1)$, and $\Delta_{2,-}$ lies inside the space of $\bigwedge^{2n}(\bar{\C^{2n}})^{\ast}\cong\bigwedge^{2n}\C^{2n}$ with highest weight $(-1,\ldots,-1)$. This completes the proof of the Proposition here.
       \end{proof}

    \subsection{Construction of a Right $K\cap H$-equivariant Cohomological Vector}\label{Subsection: construction of equivariant Cohomological vector}
    We retain all notations in the previous three Subsections. Given the cohomological representation $\pi$ in \eqref{Eq: pi parabolic induction parameter} and the integer $L$ in \eqref{eq: 1001} (determined by $\pi$), we write $\pi$ in the form of \eqref{eq: 1021} as in Subsection \ref{subsection: some reductions}. Then the sequence of integers $(N_1,N_2,\cdots,N_n)$ is a sequence of positive integers in the strictly decreasing order as we explained in Subsection \ref{subsection: some reductions}. For any integer $l$ satisfying
    \begin{equation}\label{eq: condition for l}
         N_n \geq \abs{l+L},
    \end{equation}
    we aim to construct a smooth function $\varphi$ in the minimal $K$-type $\tau$ of $\pi$ satisfying the right $K\cap H$-equivariant property \eqref{eq: right equivariant property simplified}, i.e.
    \begin{equation}\label{eq: 1032}
    \varphi(g\mtrtwo{k_1}{}{}{k_2}) = \chi_{-l}(\det k_1)\cdot\chi_{l+2L}(\det k_2)\cdot \varphi(g).
    \end{equation}
    By Iwasawa decomposition, any function $\varphi\in V_\pi$ is determined by its value on $K$. We will use the compact induction model for $\pi$ for convenience.

    Recall that the restriction map $\iota: \C[\Mat_{2n}^{\R}]\rightarrow C^\infty(K)$ is a $K$-module homomorphism.
    \begin{thm}\label{thm: construction of bi-equivariant polynomial}
       Fix two integers $l,L\in\mathbb{Z}$. Let $\chi_{-l}\otimes\chi_{l+2L}$ be a character of $\mathrm{U}_n\times \mathrm{U}_n$, $\vec{N}:=(N_1+L,-N_1+L,\cdots,N_n+L,-N_n+L)$ be a sequence of integers with the property that $N_1\ge\cdots\ge N_n\ge\abs{l+L}$. We define a polynomial function $F_{\vec{N},\chi_{-l}\otimes\chi_{l+2L}}$ in $\C[\Mat_{2n}^\R]$ as follows,
       \begin{equation}\label{eq: construction of bi-KH-inv poly}
       \begin{aligned}
         &F_{\vec{N},\chi_{-l}\otimes\chi_{l+2L}}:=\\
         &\left\{
         \begin{aligned}
         &\big(\prod_{i=1}^{n-1} F_{i}^{N_i-N_{i+1}}\big)\cdot F_{n}^{N_n-l-L}\Delta_{1,-}^{2(l+L)}\Delta_{2,+}^{l+2L},\ &\text{ if }l+L\ge 0,\ l+2L\ge 0;\\
         &\big(\prod_{i=1}^{n-1} F_{i}^{N_i-N_{i+1}}\big)\cdot F_{n}^{N_n-l-L}\Delta_{1,-}^{2(l+L)}\Delta_{2,-}^{-(l+2L)},\ &\text{ if }l+L\ge 0,\ l+2L\le 0;\\
         &\big(\prod_{i=1}^{n-1} F_{i}^{N_i-N_{i+1}}\big)\cdot F_{n}^{N_n+l+L}\Delta_{1,+}^{-2(l+L)}\Delta_{2,+}^{l+2L},\ &\text{ if }l+L\le 0,\ l+2L\ge 0;\\
         &\big(\prod_{i=1}^{n-1} F_{i}^{N_i-N_{i+1}}\big)\cdot F_{n}^{N_n+l+L}\Delta_{1,+}^{-2(l+L)}\Delta_{2,-}^{-(l+2L)},\ &\text{ if }l+L\le 0,\ l+2L\le 0,
         \end{aligned}\right.
       \end{aligned}
       \end{equation}
       where all $F_{i}$ are the right $K\cap H$-invariant polynomials on $\C[\Mat_{2n}^\R]$ constructed in Theorem \ref{thm: eigen polynomial for even fundamental repns}, $\Delta_{1,\pm}$ and $\Delta_{2,\pm}$ are the determinant polynomials constructed in Proposition \ref{prop: det lives in fundamental repn}. Then $\iota(F_{\vec{N},\chi_{-l}\otimes\chi_{l+2L}})$ lives in the minimal $K$-type $\tau$ of $\pi_K$ (see \eqref{eq: pi K}), which also satisfies the right $K\cap H$-equivariant property in \eqref{eq: 1032}.
    \end{thm}
    \begin{proof}
      1)Let us first verify the left $(T\cap K)$- and the right $K\cap H$-equivariance of $\iota(F_{\vec{N},\chi_{-l}\otimes\chi_{l+2L}})$. We only check the last case (i.e. $l+L\le 0,l+2L\le 0$). The other three cases can be proved one by one in the same manner. Recall that the group $K$ acts on the space $\C[\Mat_{2n}^\R]$ by right translation (see \eqref{eq: action of K on P}). Then for all $(k_1,k_2)\in K\cap H=\mathrm{U}_n\times \mathrm{U}_n$, we have
       \begin{equation}\label{eq: 080}
          \Delta_{1,+}(Z\mtrtwo{k_1}{}{}{k_2}, \bar{Z}\mtrtwo{\bar{k_1}}{}{}{\bar{k_2}}) = \det{(\left(
                                                \begin{array}{c}
                                                  u_1 \\
                                                  u_3 \\
                                                  \cdots \\
                                                  u_{2n-1} \\
                                                \end{array}\right)k_1)} = \det{k_1}\cdot \Delta_{1,+}(Z,\bar{Z})
       \end{equation}
       and
       \begin{equation}\label{eq: 081}
          \begin{aligned}
          \Delta_{2,-}(Z\mtrtwo{k_1}{}{}{k_2}, \bar{Z}\mtrtwo{\bar{k_1}}{}{}{\bar{k_2}}) &=\det(\bar{Z}\mtrtwo{\bar{k_1}}{}{}{\bar{k_2}})\\ &=(\det{k_1})^{-1}(\det{k_2})^{-1}\cdot \Delta_{2,-}(Z,\bar{Z}).
          \end{aligned}
       \end{equation}
       Since all $F_{i}$ are right $K\cap H$-invariant, from \eqref{eq: 080} and \eqref{eq: 081}, it is clear that
       \begin{equation*}
           \begin{aligned}
           &F_{\vec{N},\chi_{-l}\otimes\chi_{l+2L}}(Z\mtrtwo{k_1}{}{}{k_2}, \bar{Z}\mtrtwo{\bar{k_1}}{}{}{\bar{k_2}})\\
           =&\chi_{-l}(\det{k_1})\chi_{l+2L}(\det{k_2})F_{\vec{N},\chi_{-l}\otimes\chi_{l+2L}}(Z,\bar{Z}),
           \end{aligned}
       \end{equation*}
       which confirms the right $K\cap H$-equivariant property of $\iota(F_{\vec{N},\chi_{-l}\otimes\chi_{l+2L}})$.

       To show the desired left $(T\cap K)$-equivariant property, we set
       \begin{equation*}
       F_0:=\big(\prod_{i=1}^{n-1} F_{i}^{N_i-N_{i+1}}\big)\cdot F_{n}^{N_n+l+L}.
       \end{equation*}
       Then by \eqref{eq: 084},
       \begin{equation}\label{eq: 085}
          \begin{aligned}
          &F_0(\diag(e^{\boldsymbol{i}\theta_1},\cdots,e^{\boldsymbol{i}\theta_{2n}})Z, \diag(e^{-\boldsymbol{i}\theta_1},\cdots,e^{-\boldsymbol{i}\theta_{2n}})\bar{Z})\\ =&\prod_{k=1}^{n-1}\big(\prod_{i=1}^{k} e^{\boldsymbol{i}(\theta_{2i-1}-\theta_{2i})}\big)^{N_k-N_{k+1}}\cdot\big(\prod_{i=1}^ne^{\boldsymbol{i}(\theta_{2i-1}-\theta_{2i})}\big)^{N_n+l+L}F_0(Z,\bar{Z})\\
          =&\big(\prod_{i=1}^{n} e^{\boldsymbol{i}(N_i+l+L)(\theta_{2i-1}-\theta_{2i})}\big)F_0(Z,\bar{Z}).
          \end{aligned}
       \end{equation}
       Also, we have
       \begin{equation}\label{eq: 086}
         \begin{aligned}
         &\Delta_{1,+}^{-2(l+L)}(\diag(e^{\boldsymbol{i}\theta_1},\cdots,e^{\boldsymbol{i}\theta_{2n}})Z,\diag(e^{-\boldsymbol{i}\theta_1},\cdots,e^{-\boldsymbol{i}\theta_{2n}})\bar{Z})\\ = &\big(\det\left(
                                                 \begin{array}{c}
                                                   e^{\boldsymbol{i}\theta_1}u_1 \\
                                                   e^{\boldsymbol{i}\theta_3}u_3 \\
                                                   \cdots \\
                                                   e^{\boldsymbol{i}\theta_{2n-1}}u_{2n-1} \\
                                                 \end{array}
                                               \right)\big)^{-2(l+L)}\\
         =& (\prod_{i=1}^n e^{-2\boldsymbol{i}(l+L)\theta_{2i-1}}) \Delta_{1,+}(Z,\bar{Z}),
         \end{aligned}
    \end{equation}
    and similarly
    \begin{equation}\label{eq: 087}
         \Delta_{2,-}^{-(l+2L)}(\diag(e^{\boldsymbol{i}\theta_1},\ldots,e^{\boldsymbol{i}\theta_{2n}})Z,\diag(e^{-\boldsymbol{i}\theta_1},\ldots,e^{-\boldsymbol{i}\theta_{2n}})\bar{Z}) = (\prod_{i=1}^{2n} e^{\boldsymbol{i}(l+2L)\theta_{i}}) \Delta_{2,-}(Z,\bar{Z}).
    \end{equation}
    Thus, the left $(T\cap K)$-equivariant property of $F_{\vec{N},\chi_{-l}\otimes\chi_{l+2L}}$ for $l+L\le 0,l+2L\le 0$ follows from \eqref{eq: 085}, \eqref{eq: 086} and \eqref{eq: 087}. Actually, this means that $\iota(F_{\vec{N},\chi_{-l}\otimes\chi_{l+2L}})\in\iota(\C[\Mat_{2n}^\R])\cap\pi_K$.

    2)Now we will show that the function $\iota(F_{\vec{N},\chi_{-l}\otimes\chi_{l+2L}})$ lives in the minimal $K$-type $\tau$ of $\pi_K$. The key idea is to match $\tau$ with the Cartan component of certain tensor product representations. Here we still only discuss the last case in \eqref{eq: construction of bi-KH-inv poly}. The other three cases can be proved in the same method. By Corollary \ref{cor: even fundamental polynomial live in even fundamental repn}, $\iota(F_{j})$ lies inside an irreducible $K$-submodule $(\tau_{j}, V_j)$ of $C^\infty(K)$ with highest weight $\Lambda_j = (1,\cdots,1,0,\cdots,0,-1,\cdots,-1)$. Also, Proposition \ref{prop: det lives in fundamental repn} confirms that $\Delta_{1,+}$ ($\Delta_{1,-}$, resp.) lives in an irreducible $K$-module $W_1$ ($W_{-1}$ resp.) with highest weight $(1,\cdots,1,0,\cdots,0)$ ($(0,\cdots,0,-1,\cdots,-1)$ resp.); $\Delta_{2,+}$ ($\Delta_{2,-}$ resp.) lives in an irreducible $K$-module $W'_1$ ($W'_{-1}$ resp.) with highest weight $(1,\cdots,1)$ ($(-1,\cdots,-1)$ resp.). Thus as a product of $\iota(F_{j})$, $\iota(\Delta_{i,+})$ and $\iota(\Delta_{i,-})$($i=1,2$), $\iota(F_{\vec{N},\chi_{-l}\otimes\chi_{l+2L}})$ lives in a tensor product of above $K$-submodules.

        When $l+L\le 0,l+2L\le 0$, $\iota(F_{\vec{N},\chi_{-l}\otimes\chi_{l+2L}})$ lives in the tensor product space
        \begin{equation}\label{eq: 101}
            \begin{aligned}
            V :=
            V_{1}^{\otimes N_1-N_2}\otimes V_{2}^{\otimes N_2-N_3}\otimes\cdots\otimes V_{n}^{\otimes N_{n}+l+L}\otimes W_{1}^{\otimes -2(l+L)}\otimes (W'_{-1})^{\otimes -(l+2L)}.
            \end{aligned}
        \end{equation}
       The Cartan component of the tensor product $V$ is an irreducible $K$-module with highest weight $\Lambda=(N_1+L,\cdots,N_n+L,-N_n+L,\cdots,-N_1+L)$. The orthogonal complement of the Cartan component in this tensor product consists of irreducible $K$-modules with highest weights lower than $\Lambda$. Suppose $\tau_\nu$ is a $K$-module with the highest weight $\nu$. Arguing as \eqref{eq: 090}, the Frobenius Reciprocity Law
        \begin{equation}\label{eq: 102}
               \begin{aligned}
               &\Hom_{T\cap K}(\tau_\nu|_{T\cap K},\chi_{N_1+L}\otimes\cdots\otimes\chi_{N_n+L}\otimes\chi_{-N_n+L}\otimes\cdots\otimes\chi_{-N_1+L})\\
               =&\Hom_K(\tau_\nu, \Ind_{T\cap K}^K \chi_{N_1+L}\otimes\cdots\otimes\chi_{N_n+L}\otimes\chi_{-N_n+L}\otimes\cdots\otimes\chi_{-N_1+L})\\
               \cong&\Hom_K(\tau_\nu, \pi_K)
               \end{aligned}
         \end{equation}
        implies that if $\tau_\nu$ occurs in $\pi_K$, then $\nu$ is higher than $\Lambda$. By \eqref{eq: 101}, the intersection between  $\pi_K$ and $V$ is non-empty. This means that the Cartan component of $V$ is exactly the intersection of $V$ and $\pi_K$. Note that the highest weight of the minimal $K$-type $\tau$ of $\pi_K$ is exactly $\Lambda$, this shows that $\iota(F_{\vec{N},\chi_{-l}\otimes\chi_{l+2L}})$ lives in the minimal $K$-type of $\pi_K$.
   \end{proof}
   \begin{cor}\label{cor: construction of cohomological vector}
        Let $\pi$ be the cohomological representation of $G$ given in \eqref{eq: 1021}, i.e.
        $$\begin{aligned}\pi\simeq&\Ind_{B}^{G}\abs{\quad}_{\C}^\frac{m}{2}\chi_{N_1+L}\otimes\abs{\quad}_{\C}^\frac{m}{2}\chi_{-N_1+L}\otimes\abs{\quad}_{\C}^\frac{m}{2}\chi_{N_2+L}\otimes\abs{\quad}_{\C}^\frac{m}{2}\chi_{-N_2+L}\otimes\cdots\otimes\\
                         &\qquad\qquad\abs{\quad}_{\C}^\frac{m}{2}\chi_{N_{n}+L}\otimes\abs{\quad}_{\C}^\frac{m}{2}\chi_{-N_{n}+L},\end{aligned}$$ with $L\in\mathbb{Z}$ and $(N_1, N_2,\cdots, N_n)$ being a sequence of positive integers in the strictly decreasing order. For any integer $l$ satisfying \eqref{eq: condition for l}, we construct a polynomial function $F_{\vec{N},\chi_{-l}\otimes\chi_{l+2L}}$ as in Theorem \ref{thm: construction of bi-equivariant polynomial}. Then $\varphi = \iota(F_{\vec{N},\chi_{-l}\otimes\chi_{l+2L}})$ lives in the minimal $K$-type $\tau$ of $\pi$.
   \end{cor}
   \begin{proof}
        This Corollary follows directly from Theorem \ref{thm: construction of bi-equivariant polynomial} and the fact that the minimal $K$-type of $\pi_K$ coincides with the minimal $K$-type $\tau$ of $\pi$, as we explained right above \eqref{eq: pi K}.
   \end{proof}


\section{Non-vanishing of Archimedean Local Integrals}\label{sec-NV-ALI}


In this Section, we will establish the non-vanishing property for the archimedean local integrals of Friedberg-Jacquet. First, we will show the non-vanishing for the new linear period $\tilde{\Lambda}_{s,\chi}$ defined in \eqref{def: modifed linear functional} when it is evaluated at the cohomological vector $v$ corresponding to $\varphi = \iota(F_{\vec{N},\chi_{-l}\otimes\chi_{l+2L}})$. Then, by the uniqueness of twisted linear period(see \cite[Theorem B]{Ch-Sun}), we get the relation between $\tilde{\Lambda}_{s,\chi}(v)$ and the local Friedberg-Jacquet integral $Z(v,s,\chi)$. Finally, we obtain the non-vanishing property of the latter one. For convenience, we begin with the $\GL_2(\C)$ case and then reduce the $\GL_{2n}(\C)$ case to $\GL_2(\C)$ blocks in the sense of linear periods.

We retain all notations in Subsection \ref{subsection: Casimir operator} and \ref{Subsection: construction of equivariant Cohomological vector}. Let $\pi$ be the irreducible essentially tempered cohomological representation of $G$ given in Corollary \ref{cor: construction of cohomological vector}, and $\chi$ be a character of $\C^\times$. There exists an integer $l$ and $u_0\in\C$ such that
 \begin{equation}\label{eq: 1030}\chi(a) = \abs{a}_\C^{u_0}\chi_l(\frac{a}{\abs{a}}).\end{equation}
The goal of this Section is to use the explicit construction of the twisted linear functional $\Lambda_{s,\chi}$ in Subsection \ref{Subsection: Another Linear Model}  and the cohomological vector in Subsection \ref{Subsection: construction of equivariant Cohomological vector}, to give a necessary and sufficient condition on $l$ under which the archimedean local Friedberg-Jacequt integral does not vanish on the minimal $K$-type $\tau$ of $\pi$. To present the theorem in a more elegant way, we first re-normalize the twisted linear function $\Lambda_{s,\chi}$ for $\pi$ defined in \eqref{eq: new def H inv linear functional}.

Given a non-negative integer $N$ and two integers $m,L$, we consider the principal series
$$\sigma = \Ind_{B_{\GL_2}}^{\GL_2(\C)} \abs{\quad}_\C^\frac{m}{2}\chi_{N+L}\otimes \abs{\quad}_\C^\frac{m}{2}\chi_{-N+L}.$$ The minimal $K$-type $\tau_\sigma$ has highest weight $(N+L,-N+L)$. More precisely, as a representation of $\SU_2$, $\tau_\sigma$ has highest weight $2N$ and the center of $\mathrm{U}_2$ acts by the character $\chi_{2L}$. Thus, $\tau_\sigma$ is a $2N+1$ dimensional vector space and it has a weight space decomposition:
$$\tau_\sigma = \bigoplus_{k=-N}^N \tau_{\sigma,k},$$
where $\tau_{\sigma,k}$ is the one dimensional weight $2k$ space:
$$\tau_{\sigma,k} = \set{v\in \tau_\sigma}{ \sigma(\mtrtwo{e^{\boldsymbol{i}\theta}}{}{}{e^{-\boldsymbol{i}\theta}})v = e^{\boldsymbol{i}\cdot 2k\theta}v}.$$
Given the character $\chi$ in \eqref{eq: 1030}, we consider the continuous linear functional $\lambda_{s,\sigma,\chi}$ defined in \eqref{eq: local int GL(2,R)} (in this case, $F=\C$ in \eqref{eq: local int GL(2,R)}).

\begin{prop}\cite[Proposition 1, Theorem 1]{P}\label{prop: popa}
   If $\abs{l+L}\leq N$, then for all $k\ne -l-L$, $\lambda_{s,\sigma,\chi}$ vanishes on the weight $2k$ space $\tau_{\sigma,k}$; for $k=-l-L$,  there exists a vector $v\in \tau_{\sigma,k}$ such that $\lambda_{s,\sigma,\chi}(v) = L(s,\sigma\otimes\chi).$
   If $\abs{l+L}>N$, then $\lambda_{s,\sigma,\chi}$ vanishes identically on the minimal $K$-type $\tau_\sigma$.
   \end{prop}
In Section \ref{Section: Cohomological vector realization}, we provide an algorithm to construct a function in the minimal $K$-type by restricting a polynomial in $\C[\Mat_{2n}^\R]$ on $\mathrm{U}_{2n}$. In the special case of $n=1$, we have the following Corollary:
\begin{cor}\label{cor: gl2 cohomological vector}
   We assume that $\abs{l+L}\leq N$. Set $\vec{N} = (N+L,N-L)$. We parameterize a $2\times 2$ complex matrix as $Z=\mtrtwo{x_1}{y_1}{x_2}{y_2}$. Then the polynomial $\varphi_\sigma$ defined below
   \begin{equation}\label{eq: gl2 cohomological polynomial}
       \begin{aligned}
       &\varphi_\sigma =
         &\left\{
         \begin{aligned}
         &(x_{1}\bar{x}_{2})^{N-l-L}(\bar{x}_{2})^{2(l+L)}(x_{1}y_{2}-x_{2}y_{1})^{l+2L},&\text{ if }l+L\ge 0,\ l+2L\ge 0;\\
         &(x_{1}\bar{x}_{2})^{N-l-L}(\bar{x}_{2})^{2(l+L)}(\bar{x}_{1}\bar{y}_{2}-\bar{x}_{2}\bar{y}_{1})^{-(l+2L)},&\text{ if }l+L\ge 0,\ l+2L\le 0;\\
         &(x_{1}\bar{x}_{2})^{N+l+L}(x_{1})^{-2(l+L)}(x_{1}y_{2}-x_{2i}y_{1})^{l+2L},&\text{ if }l+L\le 0,\ l+2L\ge 0;\\
         &(x_{1}\bar{x}_{2})^{N+l+L}(x_{1})^{-2(l+L)}(\bar{x}_{1}\bar{y}_{2}-\bar{x}_{2}\bar{y}_{1})^{-(l+2L)},&\text{ if }l+L\le 0,\ l+2L\le 0
         \end{aligned}\right.
       \end{aligned}
    \end{equation}
    is a bi-$(\vec{N},\chi_{-l}\otimes\chi_{l+2L})$-equivariant polynomial function in $\C[\Mat_{2n}^\R]$. The restriction map $\iota_2: \C[\Mat_{2}^\R]\rightarrow C^\infty(\mathrm{U}_2)$ sends the above $\varphi_\sigma$ to a smooth function $\iota_2(\varphi_\sigma)$ living in the minimal $K$-type $\tau_\sigma$ of $\sigma$. More precisely, $\iota_2(\varphi_\sigma)\in \tau_{\sigma,-l-L}.$
\end{cor}
    Since $\tau_{\sigma,k}$ is one dimensional for all $k$, the function $\iota_2(\varphi_\sigma)$ in Corollary \ref{cor: gl2 cohomological vector} must be a nonzero multiple of the vector $v$ in Proposition \ref{prop: popa}. Thus, we can define a continuous linear functional $\tilde{\lambda}_{s,\sigma,\chi}$ on $V_{\sigma}$, which is a nonzero scalar multiple of $\lambda_{s,\sigma,\chi}$ such that for the function $\iota_2(\varphi_\sigma)$ in Corollary \ref{cor: gl2 cohomological vector}, we have
    \begin{equation}\label{eq: 1031}
       \tilde{\lambda}_{s,\sigma,\chi}(\iota_2(\varphi_\sigma)) = \begin{cases} 0 \quad&\text{ if $\abs{l+L}>N$};\\
       L(s,\sigma\otimes\chi) \quad&\text{ if $\abs{l+L}\leq N$}.\end{cases}
    \end{equation}

    Given the irreducible cohomological representation $\pi$ as in Corollary \ref{cor: construction of cohomological vector}, we follow the method in Subsection \eqref{Subsection: Another Linear Model} and rewrite $\pi$ as
    \begin{equation}\label{Eq: double induction}
       \pi \simeq \Ind_{P}^{G}\sigma_1\otimes\sigma_2\cdots\otimes\sigma_n,
    \end{equation}
    where $P$ is the standard parabolic subgroup of $G$ with Levi decomposition $P=MU$ and each $\sigma_j$ is the principal series $$\Ind_{B_{\GL_2}}^{\GL_2(\C)}\abs{\quad}_{\C}^\frac{m}{2}\chi_{N_j+L}\otimes\abs{\quad}_{\C}^\frac{m}{2}\chi_{-N_j+L}.$$ Appying Corollary \ref{cor: gl2 cohomological vector} to the case $\sigma = \sigma_j$, for each $\sigma_j$, we can define a nonzero continuous linear functional $\tilde{\lambda}_{s,j} := \tilde{\lambda}_{s,\sigma_j,\chi}$ such that the equation \eqref{eq: 1031} holds. Then we can define a nonzero continuous linear function $\tilde{\Lambda}_{s,\chi}$ on $\pi$ using the model \eqref{Eq: double induction}: for every $\varphi\in V_\pi$,
    \begin{equation}\label{def: modifed linear functional}
   \tilde{\Lambda}_{s,\chi}(\varphi) := \int_{K\cap H} \langle \bigotimes_{j=1}^n \tilde{\lambda}_{s,j}, \varphi(w\mtrtwo{k_1}{}{}{k_2}\rangle \chi_{-l}(\det k_1)\chi_{l+2L}(\det k_2)) dk_1dk_2,
    \end{equation}which is a nonzero scalar multiple of $\Lambda_{s,\chi}$ constructed in \eqref{eq: new def H inv linear functional}. Here $w$ is the Weyl element defined in \eqref{eq: def of w}. In particular, if $\varphi$ satisfies the right $K\cap H$-equivariant property \eqref{eq: 1032}, then
    \begin{equation}\label{eq: 1033}\tilde{\Lambda}_{s,\chi}(\varphi) = \langle \bigotimes_{j=1}^n \tilde{\lambda}_{s,j}, \varphi(w)\rangle.\end{equation}

    Now we have two equivalent models for $\pi$: one is in Corollary \ref{cor: construction of cohomological vector} where we construct a function in the minimal $K$-type; the other one is \eqref{Eq: double induction} which we use to construct the twisted linear functional $\tilde{\Lambda}_{s,\chi}$. Thus to evaluate the twisted linear functional $\tilde{\Lambda}_{s,\chi}$ at the cohomological vector, we only need to track the isomorphism in the double induction formula:
    \begin{equation}\label{eq: 1034}\begin{aligned}\eta: &\Ind_{P}^{G}\sigma_1\otimes\sigma_2\cdots\otimes\sigma_n\\
    \simeq &\Ind_{B}^{G}\abs{\quad}_{\C}^\frac{m}{2}\chi_{N_1+L}\otimes\abs{\quad}_{\C}^\frac{m}{2}\chi_{-N_1+L}\otimes\abs{\quad}_{\C}^\frac{m}{2}\chi_{N_2+L}\otimes\abs{\quad}_{\C}^\frac{m}{2}\chi_{-N_2+L}\otimes\cdots\otimes\\
                         &\qquad\qquad\abs{\quad}_{\C}^\frac{m}{2}\chi_{N_{n}+L}\otimes\abs{\quad}_{\C}^\frac{m}{2}\chi_{-N_{n}+L}.\end{aligned}\end{equation}
    The above isomorphism sends a smooth function of two variables $F(g,x)\in C^\infty(G\times M)$ to $F(g,e)\in C^\infty(G).$ Thus, if $\varphi = \iota(f)$ is the cohomological vector constructed in Corollary \ref{cor: construction of cohomological vector} using the model of $\pi$ on the RHS of \eqref{eq: 1034}, then $\eta^{-1}(\varphi)(w) = \varphi(xw)$ is a smooth function (defined on $M$) in $V_{\sigma_1}\otimes\cdots\otimes V_{\sigma_n}.$

\begin{thm}\label{thm: cohomological test vector}
    Let $\pi$ be the irreducible essentially tempered cohomological representation in Corollary \ref{cor: construction of cohomological vector}. If the integer $l$ satisfies $\abs{l+L}>N_n$, then $\tilde{\Lambda}_{s,\chi}$ vanishes identically on the minimal $K$-type $\tau$ of $\pi$. If $l$ satisfies $\abs{l+L}\leq N_n$, we take $\varphi = \iota(F_{\vec{N},\chi_{-l}\otimes\chi_{l+2L}})$ to be the cohomological vector constructed in Corollary \ref{cor: construction of cohomological vector} using the model for $\pi$ on the RHS of \eqref{eq: 1034}, then
    $$\tilde{\Lambda}_{s,\chi}(\eta^{-1}(\varphi)) = L(s,\pi\otimes\chi).$$
\end{thm}
\begin{proof}
    Let $\pi$ be the parabolically induced representation in \eqref{Eq: double induction}, and $\tau_j$ be the minimal $K$-type of $\sigma_j$. Then \cite[Proposition 8.1]{V2} confirms that the minimal $K$-type $\tau$ of $\pi$ is contained in the induced representation $\Ind_{M\cap K}^{K} \tau_1\otimes\cdots\otimes\tau_n$. Thus if we take any $\varphi\in V_\tau$, then $\varphi(w\mtrtwo{k_1}{}{}{k_2})\in V_{\tau_1}\otimes V_{\tau_2}\otimes\cdots\otimes V_{\tau_n}$ for all $(k_1,k_2)\in K\cap H$. When $\abs{l+L}>N_n$, Proposition \ref{prop: popa} implies that
    $$\langle \bigotimes_{j=1}^n\tilde{\lambda}_{s,j},\varphi(w\mtrtwo{k_1}{}{}{k_2})\rangle = 0.$$ This shows that $\tilde{\Lambda}_{s,\chi}$ vanishes identically on the minimal $K$-type $\tau$.

    Now we assume that $\abs{l+L}\leq N_n$. We can explicitly construct a function $\varphi = \iota(F_{\vec{N},\chi_{-l}\otimes\chi_{l+2L}})$ in the minimal $K$-type $\tau$ of $\pi$, according to Corollary \ref{cor: construction of cohomological vector}. Then $\varphi$ (and hence $\eta^{-1}(\varphi)$) satisfies the right equivairant property \eqref{eq: 1032}. In order to evaluate $\tilde{\Lambda}_{s,\chi}(\eta^{-1}(\varphi))$, we have to compute $$\eta^{-1}(\varphi)(w) = \varphi(xw) = F_{\vec{N},\chi_{-l}\otimes\chi_{l+2L}}(xw,\bar{x}\bar{w})$$ explicitly, where $w$ is the Weyl element defined in \eqref{eq: def of w}. We claim that if we write $$x=\diag(\mtrtwo{x_1}{y_1}{x_2}{y_2},\mtrtwo{x_3}{y_3}{x_4}{y_4},\cdots,\mtrtwo{x_{2n-1}}{y_{2n-1}}{x_{2n}}{y_{2n}})\in M\cap K,$$ then
    \begin{equation}\label{Eq: polynomial on M}
       \begin{aligned}
       &\eta^{-1}(\varphi)(w) = F_{\vec{N},\chi_{-l}\otimes\chi_{l+2L}}(xw,\bar{x}\bar{w}) = \\
         &\left\{
         \begin{aligned}
         &\prod_{i=1}^{n} (x_{2i-1}\bar{x}_{2i})^{N_i-l-L}(\bar{x}_{2i})^{2(l+L)}(x_{2i-1}y_{2i}-x_{2i}y_{2i-1})^{l+2L},&\text{ if }l+L\ge 0,\ l+2L\ge 0;\\
         &\prod_{i=1}^{n} (x_{2i-1}\bar{x}_{2i})^{N_i-l-L}(\bar{x}_{2i})^{2(l+L)}(\bar{x}_{2i-1}\bar{y}_{2i}-\bar{x}_{2i}\bar{y}_{2i-1})^{-(l+2L)},&\text{ if }l+L\ge 0,\ l+2L\le 0;\\
         &\prod_{i=1}^{n} (x_{2i-1}\bar{x}_{2i})^{N_i+l+L}(x_{2i-1})^{-2(l+L)}(x_{2i-1}y_{2i}-x_{2i}y_{2i-1})^{l+2L},&\text{ if }l+L\le 0,\ l+2L\ge 0;\\
         &\prod_{i=1}^{n} (x_{2i-1}\bar{x}_{2i})^{N_i+l+L}(x_{2i-1})^{-2(l+L)}(\bar{x}_{2i-1}\bar{y}_{2i}-\bar{x}_{2i}\bar{y}_{2i-1})^{-(l+2L)},&\text{ if }l+L\le 0,\ l+2L\le 0.
         \end{aligned}\right.
       \end{aligned}
    \end{equation}
    This can be verified case by case. Here we only show the last one (i.e. $l+L\le 0,\ l+2L\le 0$). The other three cases can be proved in the same manner. By a matrix computation, it is easy to see that
     \begin{equation*}
       xw = \left(
\begin{array}{cccccccc}
 x_1 & 0 & \cdots & 0 & y_1 & 0 & \cdots & 0 \\
 x_2 & 0 & \cdots & 0 & y_2 & 0 & \cdots & 0 \\
 0 & x_3 & \cdots & 0 & 0 & y_3 & \cdots & 0 \\
 0 & x_4 & \cdots & 0 & 0 & y_4 & \cdots & 0 \\
 \vdots & \vdots & \ddots & \vdots & \vdots & \vdots & \ddots & \vdots \\
 0 & 0 & \cdots & x_{2n-1} & 0 & 0 & \cdots & y_{2n-1} \\
 0 & 0 & \cdots & x_{2n} & 0 & 0 & \cdots & y_{2n}
\end{array}
\right).
    \end{equation*}
     Then according to the definition of $F_{i}$ in Theorem \ref{thm: eigen polynomial for even fundamental repns}, we have
      \begin{equation}
        F_i(xw,\bar{xw})=\left\{\begin{aligned}&(x_1\bar{x}_2\cdots x_{2i-1}\bar{x}_{2i})^{N_i-N_{i+1}}, \ &1\le i\le n-1;\\
         &(x_1\bar{x}_2\cdots x_{2n-1}\bar{x}_{2n})^{N_n+l+L},\ &i=n.\end{aligned}\right.
       \end{equation}
       Thus,
       \begin{equation}\label{eq: 103}
       \Big(\big(\prod_{i=1}^{n-1} F_{i}^{N_i-N_{i+1}}\big)\cdot F_{n}^{N_n+l+L}\Big)(mw,\bar{mw})=\prod_{i=1}^{n} (x_{2i-1}\bar{x}_{2i})^{N_i+l+L}.
       \end{equation}
       By the construction of $\Delta_{1,+}$ in Proposition \ref{prop: det lives in fundamental repn}, we get
      \begin{equation}\label{eq: 104}
      \Delta_{1,+}^{-2(l+L)}(xw,\bar{xw})=(x_1x_3\cdots x_{2i-1})^{-2(l+L)}.
      \end{equation}
      Since the determinant of the Weyl element $w$ is equal to one, then
      \begin{equation}\label{eq: 105}
      \Delta_{2,-}^{-(l+2L)}(xw,\bar{xw})=\big(\prod_{i=1}^n(\bar{x}_{2i-1}\bar{y}_{2i}-\bar{x}_{2i}\bar{y}_{2i-1})\big)^{-(l+2L)}.
      \end{equation}
      Now combining \eqref{eq: 103}, \eqref{eq: 104} and \eqref{eq: 105}, we show that for $l+L\le 0,\ l+2L\le 0$
       \begin{equation}
       F_{\vec{N},\chi_{-l}\otimes\chi_{l+2L}}(xw,\bar{x}\bar{w}) = \prod_{i=1}^{n} (x_{2i-1}\bar{x}_{2i})^{N_i+l+L}(x_{2i-1})^{-2(l+L)}(\bar{x}_{2i-1}\bar{y}_{2i}-\bar{x}_{2i}\bar{y}_{2i-1})^{-(l+2L)}.
       \end{equation}

      Let $\varphi_j := \varphi_{\sigma_j}$ be the polynomial function constructed in Corollary \ref{cor: gl2 cohomological vector}, then $\iota_2(\varphi_j)\in V_{\tau_j}$ and $\tilde{\lambda}_{s,j}(\iota_2(\varphi_j)) = L(s,\sigma_j\otimes\chi).$ Comparing the formula \eqref{Eq: polynomial on M} and \eqref{eq: gl2 cohomological polynomial}, we find that
      $$\eta^{-1}(\varphi)(w) = \prod_{j=1}^n \iota_2(\varphi_j).$$
      Therefore,
      $$\tilde{\Lambda}_{s,\chi}(\eta^{-1}(\varphi)) = \langle \bigotimes_{j=1}^n \tilde{\lambda}_{s,j}, \eta^{-1}(\varphi)(w) \rangle = \langle \bigotimes_{j=1}^n \tilde{\lambda}_{s,j}, \prod_{j=1}^n \iota_2(\varphi_j) \rangle\\
      = \prod_{j=1}^n L(s,\sigma_j\otimes\chi) = L(s,\pi\otimes\chi).$$
    \end{proof}
Now we relate $\tilde{\Lambda}_{s,\chi}$ with the local integral $Z(v,s,\chi)$.
\begin{cor}\label{cor: relation between two linear models}
     Suppose that $\pi$ is an irreducible representation of $\GL_{2n}(F)$ given in \eqref{Eq: double induction}. There exists a holomorphic function $G(s,\chi)$ such that
     \begin{equation*}
         Z(v,s,\chi) = e^{G(s,\chi)}\tilde{\Lambda}_{s,\chi}(v).
     \end{equation*}
     As a consequence, if we further assume that $\pi$ is the irreducible cohomological representation in Theorem \ref{thm: cohomological test vector} and $v = \eta^{-1}(\iota(f))$, where $\eta$ is the isomorphism \eqref{eq: 1034} between two equivalent models of $\pi$ and $\iota(f)$ is the cohomological vector constructed in Corollary \ref{cor: construction of cohomological vector}, then $Z(v,s,\chi)$ does not vanish for all $s$ and $\chi$.
\end{cor}
\begin{proof}
     The proof is similar to that of \cite[Corollary 5.2]{ChenJiangLinTianExplicitCohomologicalVectorReal}, we outline the proof here for the convenience of the reader. By \cite[Theorem B]{Ch-Sun}, for almost all $s$ and $\chi$, $L(s,\pi\otimes\chi) \ne \infty$ and
     \begin{equation*}
        \text{dim Hom}_{H}(\pi, \abs{\det}_\C^{-s+\frac{1}{2}}\chi^{-1}(\det)\otimes\abs{\det}_\C^{s-\frac{1}{2}}(\chi\omega)(\det))\leq 1.
    \end{equation*}
    Since for such a pair $(s,\chi)$, both $Z(v,s,\chi)$ and $\tilde{\Lambda}_{s,\chi}$ define a non-zero element in
    \begin{equation*}
       \text{Hom}_{H}(\pi, \abs{\det}_\C^{-s+\frac{1}{2}}\chi^{-1}(\det)\otimes\abs{\det}_\C^{s-\frac{1}{2}}(\chi\omega)(\det)),
    \end{equation*}
    there exists a meromorphic function $C(s,\chi)$ in $s$ and $\chi$ such that
    \begin{equation}\label{eq: 021}
       Z(v,s,\chi) = C(s,\chi)\tilde{\Lambda}_{s,\chi}(v).
    \end{equation}
    Unlike the real case where we choose a cohomological vector, we apply Proposition \ref{Cor: analytic property of Lambda} and pick a smooth vector $v$ such that $\tilde{\Lambda}_{s,\chi}(v) = L(s,\pi\otimes\chi)$. Thus, we obtain that
    \begin{equation*}
       C(s,\chi) = \frac{Z(f,s,\chi)}{L(s,\pi\otimes\chi)}
    \end{equation*}
    must be holomorphic, according to \cite[Theorem 3.1]{A-G-J}. Similarly, also by \cite[Theorem 3.1]{A-G-J}, we can choose a smooth vector $v_0$ such that $Z(v_0,s,\chi)=L(s,\pi\otimes\chi)$. Thus, by the same argument as the above,
    \begin{equation*}
       \frac{1}{C(s,\chi)} = \frac{\tilde{\Lambda}_{s,\chi}(v_0)}{L(s,\pi\otimes\chi)}
    \end{equation*}
    must also be holomorphic, according to Proposition \ref{Cor: analytic property of Lambda}. Hence $C(s,\chi)$ have no zeroes. This implies that there exists a holomorphic function $G(s,\chi)$ such that
    \begin{equation*}
         C(s,\chi) = e^{G(s,\chi)}.
    \end{equation*}
\end{proof}

It is finally clear that Theorem \ref{thm-main} holds.
\bibliographystyle{elsarticle-num}

\end{document}